\documentclass[11pt, leqno]{amsart}
\usepackage{latexsym}
\usepackage{amssymb}
\usepackage{amsmath}
\usepackage[english]{babel}

\usepackage{tikz}
\usepackage{tkz-euclide}
\usetikzlibrary{decorations.fractals}
\usetikzlibrary{decorations.footprints}
\usepackage[colorlinks=true, pdfstartview=FitV, linkcolor=red, citecolor=red, urlcolor=blue, backref=page]{hyperref} 

\usepackage{tikz,amsthm,amsmath,amstext,amssymb,amscd,epsfig,euscript,pspicture,multicol,graphpap,graphics,graphicx,enumerate,subfig,sidecap,wrapfig,color,pict2e}

\usepackage[utf8]{inputenc}

\usepackage[top=1in, bottom=1in, left=1in, right=1in, a4paper]{geometry}

\calclayout
\allowdisplaybreaks

\theoremstyle{plain}
\newtheorem{theorem}[equation]{Theorem}
\newtheorem{lemma}[equation]{Lemma}
\newtheorem{corollary}[equation]{Corollary}
\newtheorem{proposition}[equation]{Proposition}

\theoremstyle{definition}
\newtheorem{definition}[equation]{Definition}

\theoremstyle{remark}
\newtheorem{remark}[equation]{Remark}

\numberwithin{equation}{section}

\newcommand{\vertiii}[1]{{\left\vert\kern-0.25ex\left\vert\kern-0.25ex\left\vert #1 
		\right\vert\kern-0.25ex\right\vert\kern-0.25ex\right\vert}}

\newcommand{\ZZ}{{\mathbb{Z}}}

\newcommand{\dist}{\operatorname{dist}}

\newcommand{\re}{\mathbb{R}}

\newcommand{\ree}{\mathbb{R}^{n+1}}
\newcommand{\N}{\mathbb{N}}
\newcommand{\dd}{\mathbb{D}}

\newcommand{\C}{\mathcal{C}}

\newcommand{\F}{\mathcal{F}}

\newcommand{\W}{\mathcal{W}}

\newcommand{\mut}{\mathfrak{m}}
\newcommand{\pom}{\partial\Omega}

\renewcommand{\P}{\mathcal{P}}

\renewcommand{\emptyset}{\mbox{\textup{\O}}}
\newcommand{\tinyemptyset}{\mbox{\tiny \textup{\O}}}

\DeclareMathOperator{\supp}{supp}

\DeclareMathOperator{\diam}{diam}
\DeclareMathOperator{\interior}{int}

\DeclareMathOperator*{\Lip}{Lip}

\def\div{\mathop{\operatorname{div}}\nolimits}
\def\Cap{\mathop{\operatorname{Cap}_2}\nolimits}

\def\Xint#1{\mathchoice
	{\XXint\displaystyle\textstyle{#1}}%
	{\XXint\textstyle\scriptstyle{#1}}%
	{\XXint\scriptstyle\scriptscriptstyle{#1}}%
	{\XXint\scriptscriptstyle\scriptscriptstyle{#1}}%
	\!\int}
\def\XXint#1#2#3{{\setbox0=\hbox{$#1{#2#3}{\int}$}
		\vcenter{\hbox{$#2#3$}}\kern-.5\wd0}}
\def\aver#1{\Xint-_{#1}}

\renewcommand*{\backref}[1]{}
\renewcommand*{\backrefalt}[4]{%
 \ifcase #1 (Not cited.)%
   \or        (Cited on page~#2.)%
    \else      (Cited on pages~#2.)%
    \fi}

\usepackage{soul}

\begin{document} 

	\allowdisplaybreaks

\title[Elliptic operators on rough domains]{Perturbation of elliptic operators in 1-sided NTA domains satisfying the capacity density condition}

\author[M. Akman]{Murat Akman}
\address{{Murat Akman}\\
Department of Mathematical Sciences
\\
University of Essex
\\
Colchester CO4 3SQ, United Kingdom}	\email{murat.akman@essex.ac.uk}

\author[S. Hofmann]{Steve Hofmann}

\address{{Steve Hofmann}\\
Department of Mathematics \\University of Missouri \\ Columbia, MO 65211, USA} \email{hofmanns@missouri.edu}

\author[J. M. Martell]{José María Martell}
\address{{José María Martell} \\
Instituto de Ciencias Matemáticas CSIC-UAM-UC3M-UCM \\ 
Consejo Superior de Investigaciones Científicas \\ C/ Nico\-lás Cabrera, 13-15 \\
E-28049 Madrid, Spain} \email{chema.martell@icmat.es}

\author[T. Toro]{Tatiana Toro}

\address{{Tatiana Toro}\\
University of Washington \\
Department of Mathematics \\
Box 354350\\
Seattle, WA 98195-4350} \email{toro@uw.edu}

\thanks{The second author was partially supported by NSF grants DMS-1664047 and DMS-2000048. 
	The third author acknowledges financial support from the Spanish Ministry of Science and Innovation, through the ``Severo Ochoa Programme for Centres of Excellence in R\&D'' (CEX2019-000904-S) and through the grant MTM PID2019-107914GB-I00. The third author also acknowledges that the research leading to these results has received funding from the European Research Council under the European Union's Seventh Framework Programme (FP7/2007-2013)/ ERC agreement no. 615112 HAPDEGMT. The fourth author was partially supported by the Craig McKibben \& Sarah Merner Professor in Mathematics, by NSF grant number DMS-1664867 and DMS-1954545, and by the Simons Foundation Fellowship 614610.}

\thanks{This material is based upon work supported by the National Science Foundation under Grant No. DMS-1440140 while the authors were in residence at the Mathematical Sciences Research Institute in Berkeley, California, during the Spring 2017 semester.}

\date{\today}


\begin{abstract}
	Let $\Omega\subset\mathbb{R}^{n+1}$, $n\ge 2$, be a 1-sided non-tangentially accessible domain (aka uniform domain),  that is, $\Omega$ satisfies the interior Corkscrew and Harnack chain conditions, which are respectively scale-invariant/quantitative versions of openness and path-connectedness. Let us assume also that $\Omega$ satisfies the so-called capacity density condition, a quantitative version of the fact that all boundary points are Wiener regular. Consider $L_0 u=-\mathrm{div}(A_0\nabla u)$, $Lu=-\mathrm{div}(A\nabla u)$,  two real (non-necessarily symmetric) uniformly elliptic operators in $\Omega$, and write $\omega_{L_0}$, $\omega_L$ for the respective associated elliptic measures. The goal of this program is to find sufficient conditions guaranteeing that $\omega_L$ satisfies an $A_\infty$-condition or a $RH_q$-condition with respect to $\omega_{L_0}$. In this paper we establish that if the discrepancy of the two matrices satisfies a natural Carleson measure condition with respect to $\omega_{L_0}$, then $\omega_L\in A_\infty(\omega_{L_0})$. Additionally, we can prove that $\omega_L\in RH_q(\omega_{L_0})$ for some specific $q\in (1,\infty)$, by assuming that  such Carleson condition holds with a sufficiently small constant. This ``small constant'' case extends previous work of Fefferman-Kenig-Pipher and Milakis-Pipher together with the last author of the present paper who considered  symmetric operators in Lipschitz and bounded chord-arc domains, respectively.  Here we go beyond those settings, our domains satisfy a capacity density condition which is much weaker than the existence of exterior Corkscrew balls. Moreover, their boundaries need not be Ahlfors regular and the restriction of the $n$-dimensional Hausdorff measure to the boundary could be even locally infinite. The ``large constant'' case, that is, the one on which we just assume that the discrepancy of the two matrices satisfies a Carleson measure condition, is new even in the case of nice domains (such as the unit ball, the upper-half space, or  non-tangentially accessible domains)  and  in the case of symmetric operators.  We emphasize that our results hold in the absence of a nice surface measure: all the analysis is done with the underlying measure $\omega_{L_0}$, which behaves well in the scenarios we are considering. When particularized to the setting of Lipschitz, chord-arc, or  1-sided chord-arc domains, our methods allow us to immediately recover a number of existing perturbation results as well as extend some of them.  Our arguments rely on the square function and non-tangential estimates proved in \cite{AHMT-I}. The ``large constant'' case is obtained using the method of extrapolation of Carleson measure. This is a bootstrapping scheme that allows us to reduce matters to the case on which the discrepancy between the coefficients is small in some sawtooth domains. 
 \end{abstract}

\maketitle
\setcounter{tocdepth}{2}
\tableofcontents


\section{Introduction and Main results}

The purpose of this program is to study some perturbation problems for second order divergence form real elliptic operators with bounded measurable coefficients in domains with rough boundaries. Let $\Omega\subset\mathbb{R}^{n+1}$, $n\geq 2$, be an open set and let $Lu=-\div(A\nabla u)$ be a second order divergence form real  elliptic operator 
defined in $\Omega$. Here the coefficient matrix $A=(a_{i,j}(\cdot))_{i,j=1}^{n+1}$ is real (not 
necessarily symmetric) with $a_{i,j}\in L^\infty(\Omega)$ and is uniformly elliptic, that is, there exists a constant $\Lambda\geq 1$ such that 
\begin{align}
\label{uniformlyelliptic}
\Lambda^{-1} |\xi|^{2} \leq A(X) \xi \cdot \xi,
\qquad\qquad
|A(X) \xi \cdot\eta|\leq \Lambda |\xi|\,|\eta| 
\end{align}
for all $\xi,\eta \in\mathbb{R}^{n+1}$ and for almost every $X\in\Omega$. Associated with $L$ one can construct a family of positive Borel measures  $\{\omega_{L}^{X}\}_{X\in\Omega}$, defined on $\partial\Omega$ with $\omega^X(\pom)\le 1$ for every $X\in\Omega$, so that for each $f\in C_c(\pom)$ one can define its associated weak-solution
\begin{equation}\label{hm-sols}
u(X)=\int_{\partial\Omega} f(z)d\omega_{L}^{X}(z),\quad \mbox{whenever}\, \, X\in\Omega,
\end{equation}
which satisfies $Lu=0$ in $\Omega$ in the weak sense. In principle, unless we assume some further condition, $u$ needs not be continuous all the way to the boundary but still we think of $u$ as the solution to the continuous Dirichlet problem with boundary data $f$. We call $\omega_{L}^{X}$ the elliptic measure of $\Omega$ associated with the operator $L$ with pole at $X\in\Omega$.  For convenience,  we will sometimes write $\omega_{L}$ and call it simply the elliptic measure, dropping the dependence on the  pole. 

Given two such operators $L_0u=-\div(A_0 \nabla u)$ and $Lu=-\div(A \nabla u)$, one may wonder whether one can find conditions on the matrices $A_0$ and $A$ so that  some ``good estimates'' for the Dirichlet problem or for the elliptic measure for $L_0$ might be transferred to the operator $L$. Similarly, one may 
try to see whether $A$ being ``close'' to $A_0$ in some sense gives some relationship between $\omega_{L}$ and $\omega_{L_0}$. In this direction, a celebrated result of Littman, Stampacchia, and Weinberger in \cite{LSW} states that the continuous Dirichlet problem for the Laplace operator $L_0=\Delta$, (i.e., $A_0$ is the identity) is solvable 
if and only if it is solvable for any real elliptic operator $L$. By solvability here we mean that the elliptic measure solutions as in \eqref{hm-sols} are indeed continuous in $\overline{\Omega}$. It is well known that solvability in this sense is in fact equivalent to the fact that all boundary points are regular in the sense of Wiener, a 
condition which entails some capacitary thickness of the complement of $\Omega$. Note that, for this result, one does not need to know that $L$ is ``close'' to the Laplacian in any sense (other than the fact that both operators are uniformly elliptic).  

On the other hand, if $\Omega=\mathbb{R}^2_+$ is the upper-half plane and $L_0=\Delta$, 
then the harmonic measure associated with $\Delta$ 
 is mutually absolutely continuous with respect to the 
surface measure on the boundary, and its Radon-Nykodym derivative
is the classical Poisson kernel. 
However, Caffarelli, Fabes, and Kenig in \cite{CFK} constructed a uniformly real  elliptic operator $L$ in the plane (the pullback of the Laplacian via a quasiconformal mapping of the upper half plane to itself) for which the associated elliptic measure $\omega_L$ is not even absolutely continuous with respect to the surface measure (see also \cite{MM} for another example). Hence, in principle the ``good behavior'' of harmonic measure does not always transfer to any elliptic measure even in a nice domain such as the upper-half plane. Consequently, it is natural to see if those good properties can be transferred by assuming some conditions reflecting the fact that  $L$ is ``close'' to $L_0$ or, in other words, imposing some conditions on the disagreement of $A$ and $A_0$. 

In \cite{AHMT-I} we studied the square function and non-tangential maximal function estimates for solutions. Here we will consider the perturbation results. To put them in context let us recall the development of this field. With $L_0$ and $L$ as above, we define the disagreement of $A$ and $A_0$ as 
\[
\varrho(A,A_0)(X):=\sup_{Y\in B(X, \delta(X)/2)} |A(Y)-A_0(Y)|,
\qquad
X\in\Omega,
\]
where $\delta(X)=\dist(X,\pom)$ (thus, the supremum is taken over a Whitney ball). Define, for every $x\in \pom$ and $0<r<\diam(\pom)$,
\[
h(x,r)=\left(\frac{1}{\sigma(B(x,r)\cap\pom)} \iint_{B(x,r)\cap \Omega} \frac{\varrho(A,A_0)(X)^2}{\delta(X)} dX\right)^{\frac12},
\]
where $\sigma=\mathcal{H}^{n}|_{\partial \Omega}$ (i.e, the $n$-dimensional Hausdorff measure restricted to the boundary). 
The study of perturbation of elliptic operators was initiated by Fabes, Jerison, and Kenig in \cite{FJK} and later studied by Dahlberg \cite{D} for symmetric operators.  Dahlberg in the case of $\Omega=B(0,1)$ observed that if
\[
\lim_{r\to 0}\sup_{|x|=1} h(x,r)=0
\]
and if $\omega_{L_0}\ll \sigma$ with $d\omega_{0}/ d\sigma \in RH_q(\sigma)$ (the classical reverse Hölder condition with respect to the surface measure) for some $1<q<\infty$, 
then $\omega_{L}\ll \sigma$ and $d\omega_{L}/ d\sigma \in RH_q(\sigma)$. 
The importance of these reverse Hölder conditions comes from the fact that $d\omega_{L}/ d\sigma \in RH_q(\sigma)$ is equivalent to the $L^{q'}$-solvability of the Dirichlet problem, that is, the 
non-tangential maximal function for the solution $u$ given in \eqref{hm-sols} is controlled by $f$ in the $L^{q'}(\sigma)$-norm. Dahlberg's approach was to define $A_t=(1-t)A_0+t A$ for $0\leq t\leq 1$, obtaining a differential inequality for the best constant in the reverse Hölder inequality  for  $d\omega_{L_t}/ d\sigma$. Later, Fefferman in \cite{F} made the first attempt to remove the smallness of the function $h$. Working again in the domain $\Omega=B(0,1)$ and with symmetric operators, he showed that an $A_\infty(\sigma)$ condition is still inherited from the first measure  (that is, $\omega_{0}\in A_{\infty}(\sigma)$ implies $\omega_L\in A_{\infty}(\sigma)$) 
provided that $\mathcal{A}(\varrho(A,A_0))\in L^{\infty}(\partial B(0,1))$ (and the bound needs not  
be small). Here, 
\begin{align}
\label{FC}
\mathcal{A}(\varrho(A,A_0))(x):=\left(\iint_{\Gamma(x)} \frac{\varrho(A,A_0)(X)^2}{\delta(X)^{n+1}} dX\right)^{\frac12}
\end{align}
and $\Gamma(x)$ is the non-tangential cone with vertex at $x\in\partial\Omega$ with angular aperture $\theta<\pi/2$. Using Fubini's theorem one can easily see the connection between $h(x,r)$ and $\mathcal{A}(\varrho(A,A_0))(x)$:
\[
h(x,r)\lesssim \left(\frac{1}{\sigma(B(x,C r)\cap\partial\Omega)}\int_{B(x,Cr)\cap\partial\Omega} \mathcal{A}(\varrho(A,A_0))(x)^{2}d\sigma\right)^{\frac12}. 
\]
It was also noted in \cite{FKP} that finiteness of $\|\mathcal{A}(\varrho(A,A_0))\|_{L^{\infty}(\partial B(0,1))}$ does not allow one to preserve the reverse Hölder exponent. Indeed it was shown that for a given $1<p<\infty$, there exist uniformly elliptic symmetric matrices $A_0$ and $A$ with the property that $\mathcal{A}(\varrho(A,A_0))\in L^{\infty}(\partial B(0,1))$, $\omega_{L_0}\in RH_p(\sigma)$ but $\omega_L\notin RH_p(\sigma)$. On the other hand, one of the main results in the pioneering perturbation article by Fefferman, Kenig, and Pipher \cite{FKP} established that if the Carleson norm $\sup_{0<r<1, \, \, |x|=1}h(x,r)$ is merely assumed to be finite (not necessarily  going to zero as $r\to 0$) then $\omega_{L_0}\in A_{\infty}(\sigma)$ implies $\omega_{L}\in A_{\infty}(\sigma)$ in the symmetric case. In the same article, it was shown that the assumption that the 
previous Carleson norm $\sup_{0<r<1, \, \, |x|=1}h(r,x)$ be finite, 
is also necessary and cannot be weakened. 
One of the ingredients in \cite{FKP}  was to see that if $\Omega$ is a Lipschitz domain and if
\begin{align}
\label{FKP-condition}
\sup_{\substack{x\in\pom\\ 0<r<\diam(\pom)}} \left(\frac{1}{\omega_{L_0}(B(x,r)\cap\pom)} \iint_{B(x,r)\cap\Omega} \varrho(A,A_0)^2(X) \frac{G_{L_0}(X)}{\delta(X)^2} dX \right)^{\frac12}
<\varepsilon_0
\end{align}
for $\varepsilon_0$ sufficiently small, 
then $\omega_{L}\in RH_2(\omega_{L_0})$, where $G_{L_0}(X)=G_{L_0}(X_0,X)$ is the Green function for $L_0$ in $\Omega$ with a pole at  some fixed $X_0\in\Omega$. We further 
remark that in \cite{FKP} the authors also considered $L^r$-averages of the disagreement function $\varrho(A,A_0)$ as opposed to the supremum. Using that approach it was shown that there exists $r$ (depending on ellipticity)  such that for each $q>1$ there exists $\varepsilon_q$ so that $\omega_L\in RH_q(\omega_{L_0})$ provided that $L^r$-average of the disagreement function $\varrho(A,A_0)$ satisfies \eqref{FKP-condition} with $\varepsilon_q$.

Milakis, Pipher, and the fourth author of this article in \cite{MPT} made the first attempt to study perturbation problems for symmetric operators beyond the Lipschitz setting. To describe their results we need more notions which will be described briefly here and made precise later. A domain is called non-tangentially accessible (NTA for short) if it satisfies quantitative interior and exterior openness 
as well as quantitative (interior) path-connectedness
(see Definitions \ref{def1.cork}, \ref{def1.hc}, and \ref{def1.1nta} below). 
A boundary of a domain is called Ahlfors regular if  the surface measure of balls with center on the boundary and radius $r$ behaves like $r^{n}$ (in ambient dimension $n+1$)
(see Definition \ref{def1.ADR}).
Note that NTA domains with Ahlfors regular boundaries (called chord-arc domains) are not necessarily Lipschitz domains and in general they cannot be locally represented as graphs. The first result of Fefferman, Kenig, and Pipher discussed above was generalized in \cite{MPT} to
the setting of 
bounded chord-arc domains.  That is, if $\Omega$ is a chord-arc domain and if \eqref{FKP-condition} is satisfied for some $\varepsilon_0>0$ small, 
then $\omega_L\in RH_2(\omega_{L_0})$ (see also \cite{MT}). In addition, 
\cite{MPT} established that if 
$h(x,r)$ is small enough (uniformly in  $x\in\pom$ and $0<r<\diam(\pom)$) and $w_{L_0}\in RH_q(\sigma)$ for some $1<q<\infty$ then $w_{L_0}\in RH_q(\sigma)$. 
Futhermore, assuming that $h(x,r)$ is merely 
bounded (uniformly in  $x\in\pom$ and $0<r<\diam(\pom)$),
if $w_{L_0}\in RH_q(\sigma)$ for some $1<q<\infty$, then 
$w_{L}\in RH_p(\sigma)$ for some $1<p<\infty$. 

We also mention that Escauriaza in \cite{E} 
showed that if $\Omega$ is a Lipschitz domain, and if
$h(x,r)$ converges to 0 uniformly in $x\in\pom$ as $r$ goes to 0, then $\log (d\omega_L/d\sigma)\in \mbox{VMO}(\sigma)$ if $\log (d\omega_{L_0}/d\sigma)\in \mbox{VMO}(\sigma)$; 
here $\mbox{VMO}$ stands for the space of vanish mean oscillation introduced by Sarason. 
This result was further generalized to bounded chord-arc domains in \cite{MPT1}.

In \cite{CHM}, Cavero, and the second and the third authors of this article studied the ``small''and ``large'' perturbation for symmetric operators when the domain is a 1-sided NTA domain with Ahlfors regular boundary (called 1-sided chord-arc domains). Here 1-sided NTA domains (aka uniform domains) satisfy only quantitative interior openness and path-connectedness.  In \cite{CHM}, the perturbation results of  \cite{FKP, MPT} were generalized to 1-sided chord-arc domains. Again, smallness of $h(x,r)$ allowed the authors to preserve the exponent in the reverse Hölder condition, while finiteness yields only that the $A_\infty$ condition is transferred from one operator to the other. It is relevant to mention that the approach in \cite{CHM}, which is different from that of
 \cite{FKP,MPT}, uses the extrapolation of Carleson measure, originally introduced 
by Lewis and Murray in \cite{LM} (but based on the Corona construction of \cite{Carleson1962, CG1975}) and later developed in \cite{HL,HM1,HM}, as well as good properties of sawtooth domains (following the sawtooth construction in \cite{DJK}). The bottom line is that the large perturbation case can be reduced to the small perturbation in some sawtooth subdomains. We would like to note that the arguments of \cite{FKP, MPT, CHM} are written explicitly only in the case of real symmetric coefficients, but we would expect that similar arguments could be carried over to the non-symmetric case as well. We also mention \cite{CHMT}, 
where the non-symmetric case is also considered by using a different method, as well as \cite{MP}, where perturbation theory for certain degenerate elliptic operators is developed in the setting of domains with lower dimensional boundaries.

One common feature in the previous perturbation results is that the surface measures of the boundaries 
of the domains  
always have good properties, since in all cases the boundary is Ahlfors regular. For those results in 
which one is perturbing  $RH_q(\sigma)$ or $A_\infty(\sigma)$, this is natural as one implicitly  
needs to make sense of $\sigma$ and to that extent the Ahlfors regularity is natural. However, if one carefully looks at \eqref{FKP-condition} and the conclusion derived from it, that is,   $\omega_{L}\in RH_2(\omega_{L_0})$, there is no appearance of the surface measure, and these conditions make sense whether or not the surface measure is a well-behaved object. Another natural question that arises from \eqref{FKP-condition} is whether one can target some other reverse Hölder conditions by allowing $\varepsilon_0$ to be larger, or ultimately to investigate what are the conclusions that can be obtained assuming that  $\varepsilon_0$ is just an arbitrary large  finite constant.

The goal of this paper is to answer these questions. Our setting is that of 1-sided NTA domains satisfying the
so called capacity density condition (CDC for short), see Section \ref{section:prelim} for the precise definitions. The latter is a quantitative version of the well-known Wiener criterion and it is weaker than the Ahlfors regularity of the boundary or the existence of exterior Corkscrews. 
This setting guarantees among other things that any elliptic measure is doubling in some appropriate sense, hence one can see that a suitable portion of the boundary of the domain endowed with the Euclidean distance and with a given elliptic measure $\omega_{L_0}$ is a space of homogeneous type. In particular, classes like $A_\infty(\omega_{L_0})$ or $RH_p(\omega_{L_0})$ have the same good features of the corresponding ones in the Euclidean setting. However, our assumptions do not guarantee that the surface measure has any good behavior and could even be locally infinite. 
In one of our main results, we
consider the case in which  \eqref{FKP-condition} holds either with small or large $\varepsilon_0$.  
The small constant case can be seen as an extension of \cite{FKP, MPT} to a setting in 
which surface measure is not a good object. 
The large constant case is new even in nice domains such as balls, upper-half spaces, Lipschitz domains or chord-arc domains. 
To the best of our knowledge,  our work is
the first to establish perturbation results on sets with bad 
surface measures, and our large perturbation results are the first of their type. 
Finally, we do not require the operators to be symmetric. Our main result is formulated as follows: 

\begin{theorem}
\label{thm:main}
Let $\Omega\subset\mathbb{R}^{n+1}$, $n\ge 2$, be a 1-sided NTA domain  (cf. Definition \ref{def1.1nta}) satisfying the capacity density condition (cf. Definition \ref{def-CDC}). 
Let $Lu=-\div(A\nabla u)$ and $L_0u=-\div(A_0\nabla u)$ be real (non-necessarily symmetric) elliptic operators. Define the disagreement between $A$ and $A_0$ in $\Omega$ by
\begin{equation}\label{discrepancia}
\varrho(A, A_0)(X)
:=
\|A-A_0\|_{L^\infty (B(X,\delta(X)/2))},\qquad X\in\Omega,
\end{equation}
where $\delta(X):=\dist(X,\partial\Omega)$, and 
\begin{equation}\label{def-varrho}
\vertiii{\varrho(A,A_0)}
:=
\sup_{B} \sup_{B'}
\frac{1}{\omega_{L_0}^{X_{\Delta}} (\Delta')}
\iint_{B'\cap\Omega}\varrho(A,A_0)(X)^2\frac{G_{L_0}(X_{\Delta},X)}{\delta(X)^2}\,dX,
\end{equation}
where $\Delta=B\cap\pom$, $\Delta'=B'\cap\pom$,  and the sups     are taken respectively over all balls $B=B(x,r)$ with $x\in \pom$ and $0<r<\diam(\pom)$,
and $B'=B(x',r')$ with $x'\in2\Delta$ and $0<r'<r c_0/4$, and $c_0$ is the Corkscrew constant.

\begin{list}{$(\theenumi)$}{\usecounter{enumi}\leftmargin=1cm \labelwidth=1cm \itemsep=0.2cm \topsep=.2cm \renewcommand{\theenumi}{\alph{enumi}}}
	

	\item If $\vertiii{\varrho(A, A_0)}<\infty$, then $\omega_L\in A_\infty(\pom,\omega_{L_0})$ (cf. Definition \ref{d:RHp}). More precisely, there exists $1<q<\infty$ such that $\omega_L\in RH_q(\pom, \omega_{L_0})$ (cf.Definition \ref{d:RHp}). Here, $q$ and $[\omega_{L}]_{RH_q(\pom,\omega_0)}$ (cf. Definition \ref{d:RHp}) depend only on dimension, the 1-sided NTA and CDC constants, the ellipticity constants of $L_0$ and $L$, and $\vertiii{\varrho(A, A_0)}$.

	\item Given $1<p<\infty$, there exists $\varepsilon_p>0$ (depending only on dimension,  the 1-sided NTA and CDC constants, the ellipticity constants of $L_0$ and $L$, and $p$) such that if one has $\vertiii{\varrho(A, A_0)} \leq\varepsilon_p$, then $\omega_L\in RH_p(\pom,\omega_{L_0})$ (cf. Definition \ref{d:RHp}). Here, $[\omega_{L}]_{RH_p(\pom,\omega_0)}$  (cf. Definition \ref{d:RHp}) depends only on dimension, the 1-sided NTA and CDC constants,  the ellipticity constants of $L_0$ and $L$, and $p$.
\end{list}
\end{theorem}

\begin{remark}
Let us make a few remarks regarding the expression in \eqref{discrepancia}. First, the collection of $B'$ in the second sup is chosen so that $X_\Delta\notin  4B'$,  hence the Green function is not singular in the domain of integration. But even if the domain of integration contained $X_\Delta$ this would not cause any problem, since the corresponding estimate near $X_\Delta$ 
\begin{align*}
&\frac{1}{\omega_{L_0}^{X_{\Delta}} (\Delta')} \iint_{B(X_{\Delta},{\delta(X_{\Delta})}/{2})}\varrho(A,A_0)(X)^2\frac{G_{L_0}(X_{\Delta},X)}{\delta(X)^2}\,dX
\\ \nonumber
&\quad\quad\lesssim
(\|A-A_0\|_{L^\infty(B(X_{\Delta},{\delta(X_{\Delta})}/{2}))})^2\,\frac1{\delta(X_{\Delta})^2}\iint_{B(X_\Delta,{\delta(X_\Delta)}/{2})\cap\Omega} |X-X_{\Delta}|^{1-n}\,dX
\\ \nonumber
&\quad\quad\lesssim
(\|A-A_0\|_{L^\infty(B(X_{\Delta},{\delta(X_{\Delta})}/{2}))})^2.
\end{align*}

Second, at a first glance \eqref{discrepancia} seems different than \eqref{FKP-condition}, the condition imposed by Fefferman, Kenig, and Pipher in \cite{FKP}, which in the current case and if $\Omega$ is \textbf{bounded} (avoiding the pole as just mentioned) would read as
\begin{equation}\label{def-varrho:*}
\vertiii{\varrho(A,A_0)}_*
:=
\sup_{B'} \frac{1}{\omega^{X_{\Omega}} (\Delta')}
\iint_{B'\cap\Omega}\varrho(A,A_0)(X)^2\frac{G_{L_0}(X_{\Omega}, X)}{\delta(X)^2}\,dX,
\end{equation}
where $X_\Omega\in\Omega$ is a ``center'' of $\Omega$ (say,  $X_\Omega$ is the Corkscrew point associated with the surface ball $\Delta(x_0,\diam(\pom)/2)$ for some fixed $x_0\in\Omega$) so that $\delta(X_\Omega)\approx \diam(\pom)$; $\Delta'=B'\cap\pom$ and the sup is taken over all balls $B'=B(x',r')$ with $x'\in\pom$ and $0<r<\diam(\pom) c_0/4$. We can easily see that $\vertiii{\varrho(A,A_0)}\approx \vertiii{\varrho(A,A_0)}_*$. First, using  Lemma \ref{lemma:proppde} below and possibly Harnack's inequality, one can see that for $B=B(x,r)$ and $B'=B(x',r')$ as in \eqref{def-varrho} if $X\in B'\cap\Omega$ then $ \frac{G_{L_0}(X_\Delta, X)}{\omega_{L_0}^{X_\Delta}(\Delta')}\approx \frac{G_{L_0}(X_\Omega,X)}{\omega_{L_0}^{X_\Omega}(\Delta')}$. Thus, $\vertiii{\varrho(A,A_0)}\lesssim \vertiii{\varrho(A,A_0)}_*$. To obtain the converse inequality, let $B'=B(x',r')$ with $x'\in\pom$ and $0<r'<\diam(\pom)c_0/4$. Pick $\max\{\frac12,4\,r'/( \diam(\pom)c_0)\}<\theta<1$ and write $r=\theta\,\diam(\pom)$ so that $\diam(\pom)/2<r<\diam(\pom)$ and $r'<rc_0/4$. Set $B=B(x',r)$ and note that the Harnack chain condition  and Harnack's inequality easily yield $\omega^{X_{\Omega}} (\Delta')\approx \omega^{X_{\Delta}} (\Delta')$, and also
$G_{L_0}(X_{\Omega},X)\approx G_{L_0}(X_{\Delta},X)$ for every $X\in B'\cap\Omega$, where $\Delta=B\cap\pom$ and $\Delta'=B'\cap\pom$. All these give at once that $\vertiii{\varrho(A,A_0)}_*\lesssim \vertiii{\varrho(A,A_0)}$. Hence, $\vertiii{\varrho(A,A_0)}\approx \vertiii{\varrho(A,A_0)}_*$ when $\Omega$ is \textbf{bounded}. 

In the \textbf{unbounded} case, one could use a similar argument working with a pole at infinity, which would require to normalize appropriately $\omega_{L_0}$ and $G_{L_0}$;  here we will simply work with the scale-invariant expression \eqref{def-varrho} to avoid that issue. 
\end{remark}

Finally, we also have a generalization of a result of \cite{F, FKP, MPT}:
\begin{theorem}\label{thm:main-SF} 
Let $\Omega\subset\mathbb{R}^{n+1}$, $n\ge 2$, be a 1-sided NTA domain  (cf. Definition \ref{def1.1nta}) satisfying the capacity density condition (cf. Definition \ref{def-CDC}), and let $Lu=-\div(A\nabla u)$ and $L_0u=-\div(A_0\nabla u)$ be real (non-necessarily symmetric) elliptic operators.  Given $\alpha>0$, set
\begin{align}\label{SF-def}
\mathcal{A}_\alpha(\varrho(A,A_0))(x):=\left(\iint_{\Gamma_\alpha(x)} \frac{\varrho(A,A_0)(X)^2}{\delta(X)^{n+1}} dX\right)^{\frac12},
\qquad x\in\pom,
\end{align}
where $\Gamma_\alpha(x)=\{Y\in\Omega:|Y-x|<(1+\alpha)\delta(Y)\}$.

\begin{list}{$(\theenumi)$}{\usecounter{enumi}\leftmargin=1cm \labelwidth=1cm \itemsep=0.2cm \topsep=.2cm \renewcommand{\theenumi}{\alph{enumi}}}
	
\item If $\mathcal{A}_\alpha(\varrho(A,A_0))\in L^\infty(\omega_{L_0})$, then $\omega_L\in A_\infty(\pom,\omega_{L_0})$ (cf. Definition \ref{d:RHp}). More precisely, there exists $1<q<\infty$ such that $\omega_L\in RH_q(\pom, \omega_{L_0})$ (cf. Definition \ref{d:RHp}). Here, $q$ and $[\omega_{L}]_{RH_q(\pom,\omega_0)}$ (cf. Definition \ref{d:RHp}) depend only on dimension, the 1-sided NTA and CDC constants, the ellipticity constants of $L_0$ and $L$, $\alpha$, and $\|\mathcal{A}_\alpha(\varrho(A,A_0))\|_{L^\infty(\omega_{L_0})}$.

\item Given $p$, $1<p<\infty$, there exists $\varepsilon_p>0$ (depending only on $p$, dimension, the 1-sided NTA and CDC constants, the ellipticity constants of $L_0$ and $L$, and $\alpha$) such that if 
$\mathcal{A}_\alpha(\varrho(A,A_0))\in L^\infty(\omega_{L_0})$ with $\|\mathcal{A}_\alpha(\varrho(A,A_0))\|_{L^\infty(\omega_{L_0})}\le \varepsilon_p$, then $\omega_L\in RH_p(\pom,\omega_{L_0})$  (cf. Definition \ref{d:RHp}). Here $[\omega_{L}]_{RH_p(\pom,\omega_0)}$ (cf. Definition \ref{d:RHp}) depends only on dimension, the 1-sided NTA and CDC constants, the ellipticity constants of $L_0$ and $L$, $\alpha$, and $p$.

\end{list}	
\end{theorem}

\begin{remark}\label{remark:ambiguity}
	Note that in the previous result we are not specifying the pole for the elliptic measure $\omega_{L_0}$. However there is no ambiguity since, as a matter of fact, for any given $X$, $Y\in\Omega$ one has that $\omega_{L_0}^X$ and $\omega_{L_0}^Y$ are mutually absolutely continuous, thus $L^\infty(\pom,\omega_{L_0}^X)=L^\infty(\pom,\omega_{L_0}^Y)$ with $\|\cdot\|_{L^\infty(\pom,\omega_{L_0}^X)}=\|\cdot \|_{L^\infty(\pom,\omega_{L_0}^Y)}$. 
\end{remark}

The plan of this paper is as follows. Section \ref{section:prelim}
contains some of the preliminaries, definitions, and tools which will be used throughout the paper. Section \ref{section:main-proof} is devoted to proving our main results. As a matter of fact Theorem \ref{thm:main} follows from a local version, interesting in 
its own right, which is valid on bounded domains, see Proposition \ref{PROP:LOCAL-VERSION}. The proof of Theorem \ref{thm:main-SF} is also in Section \ref{section:main-proof}.  The proof of Proposition \ref{PROP:LOCAL-VERSION} is in Sections \ref{subsection:proof-a} and \ref{subsection:proof-b} which respectively handle the large and small constant cases. The proof of the large constant case is based on the extrapolation of Carleson measure technique mentioned above. Finally, in Section \ref{appendix-CAD}
we apply our main results to consider the case of  1-sided CAD (cf. Definition \ref{defi:CAD}) ---hence the domain is 1-sided NTA and satisfies the CDC condition--- and show in Corollaries \ref{corol:main} and \ref{corol:main-SF} that one can immediately recover some results from \cite{CHM, CHMT} (see also \cite{D, F, FKP, MPT}) as well as give new extensions.

We would like to mention that after an initial version of this work was posted on arXiv \cite{AHMT-full}, Feneuil and Poggi in \cite{FP} obtained results related to ours, compare for instance Theorem \ref{thm:main}  part $(a)$ with \cite[Theorem 1.27]{FP}, or Corollary \ref{corol:main} part $(a)$ with \cite[Corollary 1.32]{FP}. Also, the recent work \cite{CDMT} complements this paper and its companion \cite{AHMT-I}, see for instance \cite[Theorem 1.2, Corollary 1.6]{CDMT}.

\section{Preliminaries}
\label{section:prelim}

\subsection{Notation and conventions}

\begin{list}{$\bullet$}{\leftmargin=0.4cm  \itemsep=0.2cm}
	
	\item We use the letters $c,C$ to denote harmless positive constants, not necessarily the same at each occurrence, which depend only on dimension and the
	constants appearing in the hypotheses of the theorems (which we refer to as the ``allowable parameters'').  We shall also sometimes write $a\lesssim b$ and $a \approx b$ to mean, respectively, that $a \leq C b$ and $0< c \leq a/b\leq C$, where the constants $c$ and $C$ are as above, unless
	explicitly noted to the contrary.   Unless otherwise specified upper case constants are greater than $1$  and lower case constants are smaller than $1$. In some occasions it is important to keep track of the dependence on a given parameter $\gamma$, in that case we write $a\lesssim_\gamma b$ or $a\approx_\gamma b$ to emphasize  that the implicit constants in the inequalities depend on $\gamma$.
	
	\item  Our ambient space is $\ree$, $n\ge 2$. 
	
	\item Given $E\subset\ree$ we write $\diam(E)=\sup_{x,y\in E}|x-y|$ to denote its diameter.
	
	\item Given a domain $\Omega \subset \ree$, we shall
	use lower case letters $x,y,z$, etc., to denote points on $\partial \Omega$, and capital letters
	$X,Y,Z$, etc., to denote generic points in $\ree$ (especially those in $\ree\setminus \partial\Omega$).
	
	\item The open $(n+1)$-dimensional Euclidean ball of radius $r$ will be denoted
	$B(x,r)$ when the center $x$ lies on $\partial \Omega$, or $B(X,r)$ when the center
	$X \in \ree\setminus \partial\Omega$.  A {\it surface ball} is denoted
	$\Delta(x,r):= B(x,r) \cap\partial\Omega$, and unless otherwise specified it is implicitly assumed that $x\in\pom$.
	
	\item If $\pom$ is bounded, it is always understood (unless otherwise specified) that all surface balls have radii controlled by the diameter of $\pom$, that is, if $\Delta=\Delta(x,r)$ then $r\lesssim \diam(\pom)$. Note that in this way $\Delta=\pom$ if $\diam(\pom)<r\lesssim \diam(\pom)$.
	
	

	\item For $X \in \ree$, we set $\delta(X):= \dist(X,\partial\Omega)$.
	
	\item We let $\mathcal{H}^n$ denote the $n$-dimensional Hausdorff measure. 
	
	\item For a Borel set $A\subset \ree$, we let $\mathbf{1}_A$ denote the usual
	indicator function of $A$, i.e. $\mathbf{1}_A(X) = 1$ if $X\in A$, and $\mathbf{1}_A(X)= 0$ if $X\notin A$.

	
	
%
	
	\item We shall use the letter $I$ (and sometimes $J$)
	to denote a closed $(n+1)$-dimensional Euclidean cube with sides
	parallel to the coordinate axes, and we let $\ell(I)$ denote the side length of $I$.
	We use $Q$ to denote  dyadic ``cubes''
	on $E$ or $\partial \Omega$.  The
	latter exist as a consequence of Lemma \ref{lemma:dyadiccubes} below.
	
\end{list}

\subsection{Some definitions}\label{ssdefs} 

\begin{definition}[\bf Corkscrew condition]\label{def1.cork}
	Following \cite{JK}, we say that a domain $\Omega\subset \ree$
	satisfies the {\it Corkscrew condition} if for some uniform constant $0<c_0<1$ and
	for every $x\in \partial\Omega$ and $0<r<\diam(\partial\Omega)$, if we write $\Delta:=\Delta(x,r)$, there is a ball
	$B(X_\Delta,c_0r)\subset B(x,r)\cap\Omega$.  The point $X_\Delta\subset \Omega$ is called
	a {\it Corkscrew point relative to} $\Delta$ (or, relative to $B$). We note that  we may allow
	$r<C\diam(\pom)$ for any fixed $C$, simply by adjusting the constant $c_0$.
\end{definition}


\begin{definition}[\bf Harnack Chain condition]\label{def1.hc}
Again following \cite{JK}, we say
	that $\Omega$ satisfies the {\it Harnack Chain condition} if there are uniform constants $C_1,C_2>1$ such that for every pair of points $X, X'\in \Omega$
	there is a chain of balls $B_1, B_2, \dots, B_N\subset \Omega$ with $N \leq  C_1(2+\log_2^+ \Pi)$ 
	where
	\begin{equation}\label{cond:Lambda}
	\Pi:=\frac{|X-X'|}{\min\{\delta(X), \delta(X')\}}.
	\end{equation} 
such that $X\in B_1$, $X'\in B_N$, $B_k\cap B_{k+1}\neq\emptyset$ and for every $1\le k\le N$
	\begin{equation}\label{preHarnackball}
	C_2^{-1} \diam(B_k) \leq \dist(B_k,\partial\Omega) \leq C_2 \diam(B_k).
	\end{equation}
	The chain of balls is called a {\it Harnack Chain}.
\end{definition}

We note that in the context of the previous definition if $\Pi\le 1$ we can trivially form the Harnack chain $B_1=B(X,3\delta(X)/5)$ and $B_2=B(X', 3\delta(X')/5)$ where \eqref{preHarnackball} holds with $C_2=3$. Hence the Harnack chain condition is non-trivial only when $\Pi> 1$.

\begin{definition}[\bf 1-sided NTA and NTA]\label{def1.1nta}
	We say that a domain $\Omega$ is a {\it 1-sided non-tangentially accessible domain} (1-sided NTA)  if it satisfies both the Corkscrew and Harnack Chain conditions.
	Furthermore, we say that $\Omega$ is a {\it non-tangentially accessible domain}	(NTA  domain)	if it is a 1-sided NTA domain and if, in addition, $\Omega_{\rm ext}:= \ree\setminus \overline{\Omega}$ also satisfies the Corkscrew condition.
\end{definition}
\begin{remark} 
	In the literature, 1-sided NTA domains are also called \textit{uniform domains}. We remark that the 1-sided NTA condition is a quantitative form of path connectedness.
\end{remark}

\begin{definition}[\bf Ahlfors  regular]\label{def1.ADR}
	We say that a closed set $E \subset \ree$ is {\it $n$-dimensional Ahlfors regular} (AR for short) if
	there is some uniform constant $C_1>1$ such that
	\begin{equation} \label{eq1.ADR}
	C_1^{-1}\, r^n \leq \mathcal{H}^n(E\cap B(x,r)) \leq C_1\, r^n,\qquad x\in E, \quad 0<r<\diam(E).
	\end{equation}
\end{definition}

\begin{definition}[\bf 1-sided CAD and CAD]\label{defi:CAD}
	A \emph{1-sided chord-arc domain} (1-sided CAD) is a 1-sided NTA domain with AR boundary.
	A \emph{chord-arc domain} (CAD) is an NTA domain with AR boundary.
\end{definition}

We next recall the definition of the capacity of a set.  Given an open set $D\subset \ree$ (where we recall that we always assume that $n\ge 2$) and a compact set $K\subset D$  we define the capacity of $K$ relative to $D$ as 
\[
\Cap(K, D)=\inf\left\{\iint_{D} |\nabla v(X)|^2 dX:\, \, v\in C^{\infty}_{0}(D),\, v(x)\geq  1 \mbox{ in }K\right\}.
\]

\begin{definition}[\textbf{Capacity density condition}]\label{def-CDC}
	An open set $\Omega$ is said to satisfy the \textit{capacity density condition} (CDC for short) if there exists a uniform constant $c_1>0$ such that
\begin{equation}\label{eqn:CDC}
	\frac{\Cap(\overline{B(x,r)}\setminus \Omega, B(x,2 r))}{\Cap(\overline{B(x,r)}, B(x,2 r))} \geq c_1
\end{equation}
	for all $x\in \partial\Omega$ and $0<r<\diam(\pom)$.
\end{definition}

The CDC is also known as the uniform 2-fatness as studied by Lewis in \cite{L88}. Using \cite[Example 2.12]{HKM} one has that 
\begin{equation}\label{cap-Ball}
\Cap(\overline{B(x,r)}, B(x,2 r))\approx r^{n-1}, \qquad \mbox{for all $x\in\ree$ and $r>0$},
\end{equation}
and hence the CDC is a quantitative version of the Wiener regularity, in particular every $x\in\pom$ is Wiener regular. It is easy to see that the exterior Corkscrew condition implies CDC. Also, it was proved in \cite[Section 3]{Zhao} and \cite[Lemma 3.27]{HLMN} that a set with Ahlfors regular boundary satisfies the capacity density condition with constant $c_1$ depending only on $n$ and the Ahlfors regular constant.

\begin{remark}\label{remark:diam-radius}
	Given $\Omega$, a 1-sided NTA domain satisfying the CDC, as shown in \cite[Remark 2.56]{AHMT-I} if $\Delta=\Delta(x,r)$ with $x\in\pom$ and $0<r<\diam(\pom)$ then $\diam(\Delta)\approx r$. 
\end{remark}

\subsection{Dyadic analysis}\label{subsection:dyadic-analysis}

Throughout this section we will work with $E\subset\re^{n+1}$ and a countable collection of Borel sets $\mathbb{D}=\{Q\}_{Q\in\mathbb{D}}$ which is a dyadic grid on $E$, whose elements will be called ``cubes''. This means that $\mathbb{D}=\bigcup_{k\in\mathbb{Z}}\mathbb{D}_k$ (with $\dd_k\neq\emptyset$ for each $k\in\ZZ$) and the following properties hold:

\begin{list}{$\bullet$}{\leftmargin=0.8cm  \itemsep=0.2cm}

\item $E=\bigcup_{Q\in\mathbb{D}_k} Q$ for every $k\in\mathbb{Z}$ with the union comprising 
pairwise disjoint sets.

\item If $Q\in\mathbb{D}_k$ and $Q'\in\mathbb{D}_j$ with $k\ge j$ then either $Q\subset Q'$ or $Q\cap Q'=\emptyset$.

\item If for every $k>j$ and  $Q\in\mathbb{D}_k$ there exists (a unique) $Q'\in\mathbb{D}_j$ such that $Q\subset Q'$.

\end{list}

See Section \ref{ss-dyadic} below (and the references \cite{C}, and \cite{HK1,HK2}) 
for a discussion of the existence of such a dyadic system, as 
well as its additional properties.

Note that by assumption,  within the same generation (that is, within each $\dd_k$) the cubes are pairwise disjoint (hence, there are no repetitions). On the other hand, we allow repetitions in the different generations, that is, we could have that $Q\in\dd_k$ and $Q'\in\dd_{k-1}$ agree. Then, although $Q$ and $Q'$ are the same set,  as cubes we consider that they are different. In short, it is then understood that $\dd$ is an indexed collection of sets where repetitions of sets are allowed in the different generations but not within the same generation. With this in mind, we can give a proper definition of the ``length'' of a cube (this concept has no geometric meaning in this context). For every $Q\in\mathbb{D}_k$, we set $\ell(Q)=2^{-k}$, which is called the ``length'' of $Q$. Note that the ``length'' is well defined when considered on $\dd$, but it is not well-defined on the family of sets induced by $\dd$. It is important to observe that the ``length'' refers to the way the cubes are organized in the dyadic grid and in general may not have a geometrical meaning (see the examples below).

\begin{remark}\label{remark:trunc-generations}
	We would like to observe that in our notion of dyadic grid the generations run for all $k\in\ZZ$. However, as we are about to see, sometimes it is natural to truncate the generations (from above or from below). For instance, it could be that $E=Q_0$ for some $Q_0\in\dd_{k_0}$ and $k_0\in\ZZ$, hence $\dd_k=\{Q_0\}$ for all $k\le k_0$. In that scenario it is convenient to ignore those $k\in\ZZ$ with $k<k_0$ and work with $\mathbb{D}=\bigcup_{k\ge k_0 }\mathbb{D}_k$. We will actually use this convention throughout this paper and, more specifically, when $E$ is bounded we will be working with the generations $k\in\ZZ$ so that $2^{-k}\lesssim \diam(E)$. In any case, the results and proofs in this section remain valid with or without the truncation of generations. 
\end{remark}

It is interesting to introduce some examples. In $\re^n$ we can consider the collection of classical dyadic cubes. Note that here there are no repetitions at all, $E=\re^n$, and that if we let $\dd_k$ be the collection of those dyadic cubes with side length $2^{-k}$, then the ``length'' is indeed the side length. Analogously, with $E=\re^n$ we can let $\dd_{2\,k}$ be the collection of those dyadic cubes with side length $2^{-k}$ and $\dd_{2\,k+1}=\dd_{2\,k}$. Hence there are repetitions of cubes in $\dd$ and ``length'' is comparable to the square root of the side length. 

Another example is the collection of dyadic subcubes of the unit cube $Q_0=[0,1)^n$. To frame this in the previous definition (without truncating the generations),  we let $\dd_k$ be the collection of dyadic subcubes of $Q_0$ if $k\ge 0$ and $\dd_k=\{Q_0\}$ for $k\le 0$. In this scenario $E=Q_0$ and all the dyadic ancestors of $Q_0$ are indeed $Q_0$, hence there are repetitions in $\dd$. Observe that the ``length'' agrees with the side length in $\dd_k$ for $k\ge 0$. On the other hand, for $Q_k\in \dd_k$ with $k\le 0$ we have that $\ell(Q_k)=2^{-k}$ (note that $Q_k$ and $Q_0$ are the same set but as dyadic cubes they are distinct). In this case, it may be convenient and more natural to truncate the generations and just work with $\dd_k$, $k\ge 0$, in which case the ``length'' agrees with the side length.

We can also consider  all classical dyadic cubes with side length at least $1$. In this scenario, let $\dd_k$ be the set of classical dyadic cubes with side length $2^{-k}$ for $k\le 0$, and $\dd_k$ the collection of classical dyadic cubes with side length $1$ for $k\ge 0$. In this scenario, $E=\re^n$ and all the dyadic descendants of any cube $Q$ with length side equal 1 are indeed $Q$, hence there are repetitions in $\dd$. Note that ``length'' agrees with the side length in $\dd_k$ for $k\le 0$, however in $\dd_k$ for $k\ge 0$ the ``length'' is $2^{-k}$ although the cubes comprising that family have side length $1$. Again, in this example, it may be more natural to truncate the generations and work with $\dd_k$, $k\le 0$, so that ``length'' and side length agree. 

Our last example is that of dyadic subcubes of the unit cube $Q_0=[0,1)^n$ with side length at least $2^{-N}$ with $N\in\N$ fixed. We let $\dd_k$ be the collection of dyadic subcubes of $Q_0$ if $0\le  k\le N$, $\dd_k=\{Q_0\}$ for $k\le 0$, and $\dd_k$, $k\ge N$, is the collection of all dyadic subcubes of $Q_0$ of side length $2^{-N}$. In this case, $E=Q_0$, all the dyadic ancestors of $Q_0$ are indeed $Q_0$, 
and all the dyadic descendants of any cube $Q$ with length side equal $2^{-N}$ are indeed $Q$. We have infinitely many cubes but only a finite number of different sets. Here the reasonable thing is to truncate the generations and just work with $\dd_k$, $0\le k\le N$.

We next introduce the ``discretized Carleson region'' relative to $Q$, $\mathbb{D}_{Q}=\{Q'\in\dd:Q'\subset Q\}$. Let $\mathcal{F}=\{Q_i\}\subset\mathbb{D}$ be a family of pairwise disjoint cubes. The ``global discretized sawtooth'' relative to $\mathcal{F}$ is the collection of cubes $Q\in\mathbb{D}$ that are not contained in any $Q_i\in\mathcal{F}$, and for a given $Q\in\mathbb{D}$, the ``local discretized sawtooth'' relative to $\mathcal{F}$ is the collection of cubes in $\mathbb{D}_Q$ in $\mathbb{D}_\mathcal{F}$. These are respectively
\[
\mathbb{D}_\mathcal{F}:=\mathbb{D}\setminus\bigcup_{Q_i\in\mathcal{F}}\mathbb{D}_{Q_i}, 
\qquad
\mathbb{D}_{\mathcal{F},Q}:=\mathbb{D}_{Q}\setminus\bigcup_{Q_i\in\mathcal{F}}\mathbb{D}_{Q_i}=\mathbb{D}_\mathcal{F}\cap\mathbb{D}_Q.
\]
We also allow $\F$ to be the null set in which case $\mathbb{D}_{\tinyemptyset}=\dd$ and $\mathbb{D}_{\tinyemptyset,Q}=\dd_Q$.

With a slight abuse of notation,  let $Q^0$ be either $E$, and in that case $\dd_{Q^0}:=\dd$, or a fixed cube in $\dd$, hence $\dd_{Q^0}$ is the family of dyadic subcubes of $Q^0$. Let $\mu$ be a non-negative Borel measure on $Q^0$ so that $0<\mu(Q)<\infty$ for every $Q\in\mathbb{D}_{Q^0}$. Consider the operators $\mathcal{A}_{Q^0}$, $\mathcal{B}_{Q^0}$ defined by
\begin{equation}\label{definicionA-C}
\mathcal{A}_{Q^0}^\mu\alpha(x):=\bigg(\sum_{x\in Q\in\mathbb{D}_{Q^0}}\frac{1}{\mu(Q)}\,\alpha_{Q}^2\bigg)^{\frac12},
\quad
\mathcal{B}_{Q^0}^\mu \alpha(x):=\sup_{x\in Q\in\mathbb{D}_{Q^0}}\bigg(\frac{1}{\mu(Q)}\sum_{Q'\in\mathbb{D}_{Q}}\alpha_{Q'}^2\bigg)^{\frac12},
\end{equation}
where $\alpha=\{\alpha_Q\}_{Q\in\mathbb{D}_{Q^0}}$ is a sequence of real numbers. Note that these operators are discrete analogues of those used in \cite{CMS} to develop the theory of tent spaces. Sometimes, we use a truncated version of $\mathcal{A}_{Q^0}^\mu $, which is denoted as $\mathcal{A}^{\mu,k}_{Q^0}\alpha$, $k\geq 0$, and where the sum runs over $x\in Q\in \dd_Q^k:=\{Q'\in\dd_Q: \ell(Q')\le 2^{-k}\ell(Q)\}$.

The following lemma is a discrete version of \cite[Theorem 1]{CMS} and extends \cite[Lemma 3.8]{CHM}:

\begin{lemma}\label{lemma:tentspaces}
	Under the previous considerations, given $Q^0$ as above, and $\alpha=\{\alpha_{Q}\}_{Q\in\mathbb{D}_{Q^0}}$, $\beta=\{\beta_{Q}\}_{Q\in\mathbb{D}_{Q^0}}$ sequences of real numbers, we have that
	\begin{equation}\label{cotatentspaces}
	\sum_{Q\in\mathbb{D}_{Q^0}}\,|\alpha_{Q}\beta_{Q}|
	\leq 
	4
	\int_{Q^0}\mathcal{A}_{Q^0}^\mu \alpha(x)\mathcal{B}_{Q^0}^\mu\beta(x)\,d\mu(x).
	\end{equation}
\end{lemma}

\begin{proof}
	The proof follows the argument in \cite[Lemma 3.8]{CHM} which in turn is based on \cite[Theorem 1]{CMS}. We first claim that it suffices to assume that $Q^0\in\dd$. Indeed, if $Q^0=E$ we have
	\begin{multline*}
	\sum_{Q\in\mathbb{D}_{Q^0}}\,|\alpha_{Q}\beta_{Q}|
	=
	\sum_{Q\in\mathbb{D}}\,|\alpha_{Q}\beta_{Q}|
	=
	\sup_N\sum_{Q\in \dd_{-N}}\sum_{Q'\in\dd_Q}|\alpha_{Q'}\beta_{Q'}|
	\\
	\le
	4
	\sup_N\sum_{Q\in \dd_{-N}}
	\int_{Q}\mathcal{A}_{Q}^\mu \alpha(x)\mathcal{B}_{Q}^\mu\beta(x)\,d\mu(x)
	\le 4
	\int_{E}\mathcal{A}_{Q^0}^\mu \alpha(x)\mathcal{B}_{Q^0}^\mu\beta(x)\,d\mu(x),
	\end{multline*}
	where in the first estimate we have used our claim for $Q$, which has finite length, and in the second one the fact that the cubes in $\dd_{-N}$ are pairwise disjoint. 
		
	From now on we assume $Q^0\in\dd$, hence $\ell(Q^0)<\infty$. Recall $\dd$ that is countable collection of cubes and then we can find $\dd^1\subset\dd^2\subset \dots\subset \dd^N\subset \dots\subset \dd$ with $\dd=\bigcup_{N\ge 1} \dd^N$ and $\#\dd^N\le N$.  Given $N\ge 1$,  let $\beta^N=\{\beta_Q^N\}_{Q\in\mathbb{D}_{Q^0}}$ where $\beta_Q^N=\beta_Q$ if $Q\in \dd^N$ and  $\beta_Q^N=0$ otherwise. With this notation in mind, if we show \eqref{cotatentspaces} for $\beta^N$ then observing that $\mathcal{B}_{Q^0}^\mu\beta^N\le \mathcal{B}_{Q^0}^\mu \beta$ we just need to let $N\to\infty$ and the desired estimate follows at once. 
	
	Thus from now on we work with $\beta^N$. To simplify the presentation we drop the exponent and keep in mind that $\beta_Q=0$ for every $Q\not\in\dd^N$. For $Q\in\mathbb{D}_{Q^0}$, let $k_Q\geq 0$ be so that $\ell(Q)=2^{-k_Q}\ell(Q^0)$. Suppose that $Q'\in\mathbb{D}_{Q^0}$ satisfies $\ell(Q')\leq 2^{-k_Q}\ell(Q^0)=\ell(Q)$ and $Q'\cap Q\neq\emptyset$, then necessarily $Q'\in\mathbb{D}_{Q}$ and for every $x\in Q$
	\begin{multline}\label{acotacionk0}
	\xi_Q:=\aver{Q} \big(\mathcal{A}_{Q^0}^{\mu, k_Q}\beta(y)\big)^2\,d\mu(y)
	=
	\aver {Q}\sum_{Q'\in\mathbb{D}_{Q}}\mathbf{1}_{Q'}(y)\frac{1}{\mu(Q')}\,\beta_{Q'}^2\,d\mu(y)
	\\
	=\frac1{\mu(Q)}
	\sum_{Q'\in\mathbb{D}_{Q}}\,\beta_{Q'}^2
	\le
	\big(\mathcal{B}_{Q^0}^\mu\beta(x)\big)^2
	.
	\end{multline}
	Since $\beta_Q=0$ for $Q\not\in\dd^N$ and $\#\dd^N\le N$,  we have $\mathcal{A}_{Q^0}^\mu \beta(x)\leq C_N<\infty$ for every $x\in Q^0$ and hence $\xi_Q\leq C_N^2<\infty$.  Now, define
	$$
	F_0:=\big\{x\in Q^0:\:\mathcal{A}_{Q^0}^{\mu, k}\beta(x)> 2\,\mathcal {B}_{Q^0}^\mu \beta(x),\;\forall k\geq 0\big\}.
	$$
	In particular, using \eqref{acotacionk0}, we have $\mathcal{A}_{Q^0}^{\mu, k_Q}\beta(x)>2\,\xi_Q^{\frac12}$ for each $x\in Q\cap F_0$. We claim that $4\mu(Q\cap F_0)\leq\mu(Q)$.  Indeed, if $\xi_Q=0$ then one can see that $\mathcal{A}_{Q^0}^{\mu, k_Q}\beta(y)=0$ for every $y\in Q$ and hence $Q\cap F_0=\emptyset$, which trivially gives that $4\mu(Q\cap F_0)\leq\mu(Q)$.  On the other hand, if $\xi_Q>0$, we have
	$$
	4\,\xi_Q\,\mu(Q\cap F_0)\leq\int_{Q\cap F_0}\big(\mathcal{A}_{Q^0}^{\mu,k_Q}\beta(y)\big)^2\,d\mu(y)\leq\xi_Q\,\mu(Q),
	$$
	and the desired estimate follows since $0<\xi_Q<\infty$. Let us now consider
	\begin{equation}\label{defkx}
	k(x):=\min\big\{k\geq 0:\ \mathcal{A}_{Q^0}^{\mu,k} \beta(x)\leq 2\,\mathcal{B}_{Q^0}^\mu \beta(x)\big\},\qquad x\in Q^0\setminus F_0.
	\end{equation}
	Setting $F_{1,Q}:=\{x\in Q\setminus F_0:\,k(x)>k_Q\}$ and using \eqref{acotacionk0} we obtain
	$$
	F_{1,Q}\subset\{x\in Q\setminus F_0:\:\mathcal{A}_{Q^0}^{\mu, k_Q}\beta(x)>2\, \xi_Q^{\frac12}\big\}.
	$$
	Applying Chebyshev's inequality, it follows that
	$$
	\mu(F_{1,Q})\leq\frac{1}{4\,\xi_Q}\int_{Q\setminus F_0}\big(\mathcal{A}_{Q^0}^{\mu, k_Q} \beta(y)\big)^2\,d\mu(y)
	\leq
	\frac{1}{4}\mu(Q).
	$$
	Setting $F_{2,Q}:=\{x\in Q\setminus F_0:\,k(x)\leq k_Q\}$, and gathering the above estimates, we have
	$$
	\mu(F_{2,Q})=\mu(Q)-\mu(Q\cap F_0)-\mu(F_{1,Q})\geq\frac{1}{2}\mu(Q).
	$$
	Hence, Cauchy-Schwarz's inequality and \eqref{defkx} yield
	\begin{align*}
	\sum_{Q\in\mathbb{D}_{Q^0}}|\alpha_Q\beta_Q|
	&\le 
	2
	\sum_{Q\in\mathbb{D}_{Q^0}}\mu(F_{2,Q})\frac{|\alpha_Q\beta_Q|}{\mu(Q)}
	\leq
	2
	\int_{Q^0\setminus F_0}\sum_{Q\in\mathbb{D}_{Q^0}}\frac{|\alpha_Q\beta_Q|}{\mu(Q)}\mathbf{1}_{F_{2,Q}}(x)\,d\mu(x)
	\\
	&\le
	2 \int_{Q^0\setminus F_0}\mathcal{A}_{Q^0}^\mu\alpha(x)\bigg(\sum_{Q\in\mathbb{D}_{Q^0}}\frac{1}{\mu(Q)}\,\beta_{Q}^2
	\mathbf{1}_{F_{2,Q}}(x)\bigg)^{\frac12}
	\,d\mu(x)
	\\
	&\le 
	2
	\int_{Q^0\setminus F_0}\mathcal{A}_{Q^0}^\mu\alpha(x)\mathcal{A}_{Q^0}^{\mu, k(x)}\beta(x)\,d\mu(x)
	\\
	&\le
	4
	\int_{Q^0}\mathcal{A}_{Q^0}^\mu\alpha(x)\mathcal{B}_{Q^0}^\mu\beta(x)\,d\mu(x),
	\end{align*}
	where we have used that $Q\in\mathbb{D}_{Q^0}^{k(x)}$ for each $x\in F_{2,Q}$. This completes the proof of \eqref{cotatentspaces}.
\end{proof}

\begin{lemma}\label{lemma:Carleson-mu-nu}
	Under the previous considerations, given $Q^0$ as above, let $\mu$ and $\nu$ be two non-negative Borel measures on $Q^0$ so that $0<\mu(Q),\nu(Q)<\infty$ for every $Q\in\mathbb{D}_{Q^0}$. Assume that 
	there exist $\alpha, \beta\in (0,1)$ such that
	\begin{equation}\label{cond-Ainfty-dyadic:carleson}
	F\subset Q\in\dd_{Q^0},\ \frac{\mu(F)}{\mu(Q)}>\alpha
	\qquad\implies\qquad
	\frac{\nu(F)}{\nu(Q)}\ge \beta.
	\end{equation}
	Given $\gamma=\{\gamma_{Q}\}_{Q\in\mathbb{D}_{Q^0}}$, a sequence of non-negative real numbers, if we set
	\[
	\vertiii{\gamma}_{\nu}:=\sup_{Q\in\dd_{Q^0}} \frac1{\nu(Q)}\sum_{Q'\in\dd_Q}\gamma_{Q'}\, \nu(Q'),
	\qquad
	\vertiii{\gamma}_{\mu}:=\sup_{Q\in\dd_{Q^0}} \frac1{\mu(Q)}\sum_{Q'\in\dd_Q}\gamma_{Q'}\, \mu(Q').
	\]
	then, 
	\begin{equation}\label{Carleson-nu-mu}
	(1-\alpha)\,\beta\,
	\vertiii{\gamma}_{\mu}
	\le 
	\vertiii{\gamma}_{\nu}
	\le 
	\frac1{(1-\alpha)\,\beta}\,
	\vertiii{\gamma}_{\mu}.
	\end{equation}
\end{lemma}

Let us observe that when $\mu$ is dyadically doubling (that is, there exists $C_\mu$ such that $\mu(Q)\le C_\mu \mu(Q')$ for every $Q, Q'\in\dd_{Q^0}$ with $\ell(Q)=2\ell(Q')$), the assumption \eqref{cond-Ainfty-dyadic:carleson} means exactly $\nu\in A_\infty^{\rm dyadic}(Q^0,\mu)$ (see Definition \ref{def:Ainfty-dyadic} below).

\begin{proof}
	We first consider the case on which $\#\{Q\in\dd_{Q^0}:\gamma_Q\neq 0\}<\infty$ so that 
	$\vertiii{\gamma}_{\nu}$, $\vertiii{\gamma}_{\mu}<\infty$ (albeit with constants depending on the set $\{Q:\gamma_Q\neq 0\}$), condition which will be used qualitatively. We will eventually see how to pass to the general case.

	Fix $Q_0\in\dd_{Q^0}$. Let $\F=\{Q_j\}_j$ be the collection of dyadic cubes contained in $Q_0$
	that are maximal with respect to the inclusion, and therefore pairwise disjoint,  with respect to the property that 
	\begin{equation}\label{eqn:stop}
	\frac{\nu(Q)}{\mu(Q)}> \frac1{1-\alpha}\,\frac{\nu(Q_0)}{\mu(Q_0)}
	\end{equation}
	Note that $\F\subset\dd_{Q_0}\setminus \{Q_{0}\}$ since $(1-\alpha)^{-1}>1$. Also, the maximality of the cubes in $\F$ immediately gives
	\begin{equation}\label{est-above}
	\frac{\nu(Q)}{\mu(Q)} \le \frac1{1-\alpha}\,\frac{\nu(Q_0)}{\mu(Q_0)},\qquad \forall\, Q\in\dd_{\F, Q_0}.
	\end{equation}
	Set $	E_0= \bigcup_{Q_j\in\F} Q_j$
	and note that if $\F$ is the null set then we understand that $E_0$ is also empty. The definition of the family $\F$ gives
	\[
	\frac{\mu(E_0)}{\mu (Q_0)}
	=
	\sum_{Q_j\in \F} \frac{\mu(Q_j)}{\mu (Q_0)}
	<
	(1-\alpha)\,\sum_{Q_j\in \F} \frac{\nu(Q_j)}{\nu (Q_0)}
	=
	(1-\alpha)\, \frac{\nu(E_0)}{\nu (Q_0)}
	\le
	1-\alpha. 
	\]
	Applying \eqref{cond-Ainfty-dyadic:carleson} to $F=Q_0\setminus E_0$ which satisfies $\mu(Q_0\setminus E_0)>\alpha\,\mu(Q_0)$ we obtain $\nu(Q_0\setminus E_0)\ge \beta\,\nu(Q_0)$ and eventually $\nu(E_0)\le (1-\beta)\,\nu(Q_0)$. Therefore,
	\begin{multline*}
	\sum_{Q\in \dd_{Q_0}\setminus\dd_{\F, Q_0}} \gamma_Q\, \nu(Q)
	=
	\sum_{Q_j\in\F} \sum_{Q\in\dd_{Q_j}} \gamma_Q\, \nu(Q)
	\le
	\vertiii{\gamma}_{\nu}\, \sum_{Q_j\in\F} \nu(Q_j)
	\\
	=
	\vertiii{\gamma}_{\nu}\, \nu\bigg(\bigcup_{Q_j\in\F} Q_j\bigg)
	=
	\vertiii{\gamma}_{\nu}\, \nu(E_0)
	\le
	(1-\beta)\,\vertiii{\gamma}_{\nu}\,\nu(Q_0).
	\end{multline*}
	On the other hand, invoking \eqref{est-above},
	\begin{multline*}
	\frac1{\nu(Q_0)}\,\sum_{Q\in \dd_{\F, Q_0}} \gamma_Q\, \nu(Q)
	\le
	\frac1{1-\alpha}\,
	\frac1{\mu(Q_0)}\,\sum_{Q\in \dd_{\F, Q_0}} \gamma_Q\, \mu(Q)
	\\
	\le 
	\frac1{1-\alpha}\,
	\frac1{\mu(Q_0)}\,\sum_{Q\in \dd_{Q_0}} \gamma_Q\, \mu(Q)
	\le
	\frac1{1-\alpha}\,
	\vertiii{\gamma}_{\mu}.
	\end{multline*}
	Combining the previous estimates we arrive that
	\begin{multline*}
	\frac1{\nu(Q_0)}\,\sum_{Q\in \dd_{Q_0}} \gamma_Q\, \nu(Q)
	=
	\frac1{\nu(Q_0)}\,\Big(
	\sum_{Q\in \dd_{Q_0}\setminus\dd_{\F, Q_0}} \gamma_Q\, \nu(Q)+
	\sum_{Q\in \dd_{\F, Q_0}} \gamma_Q\, \nu(Q)
	\Big)
	\\
	\le 
	(1-\beta)\,\vertiii{\gamma}_{\nu}+ \frac1{1-\alpha}\,
	\vertiii{\gamma}_{\mu}.
	\end{multline*}
	We next take the supremum over all $Q_0\in\dd_{Q^0}$ to conclude
	\[
	\vertiii{\gamma}_{\nu}
	\le
	(1-\beta)\,\vertiii{\gamma}_{\nu}+ \frac1{1-\alpha}\,
	\vertiii{\gamma}_{\mu}.
	\]
	Recalling that in the current case $\vertiii{\gamma}_{\nu}<\infty$ (and this is used qualitatively) the first term in the right hand side can be absorbed and we eventually obtain the second estimate in \eqref{Carleson-nu-mu}.

	Let us now remove the assumption $\#\{Q:\gamma_Q\neq 0\}<\infty$. Much as in the proof of Lemma \ref{lemma:tentspaces} we can find $\dd^1\subset\dd^2\subset \dots\subset \dd^N\subset \dots\subset \dd$ with $\dd=\bigcup_{N\ge 1} \dd^N$ and $\#\dd^N\le N$. Given $N\ge 1$,  let $\gamma^N=\{\gamma_Q^N\}_{Q\in\mathbb{D}_{Q^0}}$ where $\gamma_Q^N=\beta_Q$ if $Q\in \dd^N$ and  $\gamma_Q^N=0$ otherwise. Note that 
	$\#\{Q:\gamma_Q^N\neq 0\}\le N<\infty$ hence the previous estimate applies to $\gamma^N$. Thus, for every $Q_0\in\dd_{Q^0}$
	\begin{multline*}
	\frac1{\nu(Q_0)}\,\sum_{Q\in \dd_{Q_0}} \gamma_Q\, \nu(Q)
	=
	\sup_{N\ge 1}
	\frac1{\nu(Q_0)}\,\sum_{Q\in \dd_{Q_0}\cap\dd_N} \gamma_Q\, \nu(Q)
	\\
	=
	\sup_{N\ge 1}
	\frac1{\nu(Q_0)}\,\sum_{Q\in \dd_{Q_0}} \gamma_Q^N\, \nu(Q)
	\le
	\sup_{N\ge 1}
	\frac1{(1-\alpha)\,\beta}\,\vertiii{\gamma^N}_{\mu}
	\le
	\frac1{(1-\alpha)\,\beta}\,\vertiii{\gamma}_{\mu}.
	\end{multline*}
	Taking now the supremum over all $Q_0\in\dd_{Q^0}$ we conclude the second estimate in \eqref{Carleson-nu-mu}.
	
	Obtaining the first estimate in \eqref{Carleson-nu-mu} is now easy.  Set $\widetilde{\alpha}=1-\beta$ and $\widetilde{\beta}=1-\alpha$, and note that for any $F\subset Q\in\dd_{Q^0}$, applying  the contrapositive of \eqref{cond-Ainfty-dyadic:carleson} to $Q\setminus F$ we obtain
	\[
	\frac{\nu(F)}{\nu(Q)}>\widetilde{\alpha}
	\ \implies\ 
	\frac{\nu(Q\setminus F)}{\nu(Q)}<\beta
	\ \implies\ 
	\frac{\mu(Q\setminus F)}{\mu(Q)}\le \alpha
	\ \implies\ 
	\frac{\mu(F)}{\mu(Q)}\ge \widetilde{\beta}.
	\]
	That is, in \eqref{cond-Ainfty-dyadic:carleson} holds with $\mu$ and $\nu$ swapped, and with $\widetilde{\alpha}$,  and $\widetilde{\beta}$. Hence, the second estimate in \eqref{Carleson-nu-mu} with $\mu$ and $\nu$ swapped yields
	\[
	\vertiii{\gamma}_{\mu}
	\le
	\frac1{(1-\widetilde{\alpha})\,\widetilde{\beta}}\,\vertiii{\gamma}_{\nu}
	=
	\frac1{(1-\alpha)\,\beta}\,\vertiii{\gamma}_{\nu},
	\]
	which is the first estimate in \eqref{Carleson-nu-mu}. This completes the proof. 
\end{proof}

As above, $Q^0$ is either $E$, and in which case $\dd_{Q^0}:=\dd$, or a fixed cube in $\dd$, hence $\dd_{Q^0}$ is the family of dyadic subcubes of $Q^0$. For the rest of the section we will be working with $\mu$ which is dyadically doubling in $Q^0$. This means that there exists $C_\mu$ such that
$\mu(Q)\le C_\mu \mu(Q')$ for every $Q, Q'\in\dd_{Q^0}$ with $\ell(Q)=2\ell(Q')$.

\begin{definition}[$A_{\infty}^{\rm dyadic}$]\label{def:Ainfty-dyadic}
	
	Given $Q^0$ and $\mu$, a non-negative dyadically doubling measure in $Q^0$, a non-negative Borel measure $\nu$ defined on $Q^0$ is said to belong to $A_\infty^{\rm dyadic}(Q^0,\mu)$ if there exist constants $0<\alpha,\beta<1$ such that for every $Q\in\mathbb{D}_{Q^0}$ and for every Borel set $F\subset Q$, we have that
	\begin{equation}\label{cond-Ainfty-dyadic}
	\frac{\mu(F)}{\mu(Q)}>\alpha
	\qquad\implies\qquad
	\frac{\nu(F)}{\nu(Q)}>\beta.
	\end{equation}
\end{definition}

It is well known (see \cite{CF-1974,GR}) that since $\mu$ is a dyadically doubling measure in $Q^0$, $\nu\in A_\infty^{\rm dyadic}(Q^0,\mu)$ if and only if $\nu\ll\mu$ in $Q^0$  and there exists $1<p<\infty$ such that $\nu\in RH_p^{\rm dyadic}(Q^0,\mu)$, that is, there is a constant $C\ge 1$ such that
 $$
 \bigg(\aver{Q} k(x)^{p}\,d\mu (x)\bigg)^{\frac{1}{p}}
 \leq
 C\aver{Q} k(x)\,d\mu(x)
 = 
 C\,
 \frac{\nu(Q)}{\mu(Q)},
 $$
 for every $Q\in\dd_{Q^0}$, and where $k=d\nu/d\mu$ is the Radon-Nikodym derivative.

For each $\mathcal{F}=\{Q_i\}\subset\mathbb{D}_{Q^0}$, a family of pairwise disjoint dyadic cubes, and each $f\in L^1_{\rm loc}(\mu)$, we define the projection operator
$$
\mathcal{P}_{\mathcal{F}}^\mu f(x)
=
f(x)\mathbf{1}_{E\setminus(\bigcup_{Q_i\in\mathcal{F}} Q_i)}(x)+\sum_{Q_i\in\mathcal{F}}\Big(\aver{Q_i}f(y)\,d\mu(y)\Big)\mathbf{1}_{Q_i}(x).
$$
If $\nu$ is a non-negative Borel measure on $Q^0$, we may naturally then define the measure $\mathcal{P}_\mathcal{F}^\mu\nu$ as $\mathcal{P}_{\mathcal{F}}^\mu\nu(F)=\int_{E}\mathcal{P}_{\mathcal{F}}^\mu\mathbf{1}_F\,d\nu$, that is,
\begin{equation}\label{defprojection}
\mathcal{P}_{\mathcal{F}}^\mu\nu(F)=\nu\Big(F\setminus\bigcup_{Q_i\in\mathcal{F}}Q_i\Big)+\sum_{Q_i\in\mathcal{F}}\frac{\mu(F\cap Q_i)}{\mu(Q_i)}\nu(Q_i),
\end{equation}
for each Borel set $F\subset Q^0$.

The next result follows easily by adapting the arguments in  \cite[Lemma B.1]{HM1} and \cite[Lemma 4.1]{HM-note} to the current scenario.

\begin{lemma}\label{lemm_w-Pw-:properties}
Given $Q^0$, let $\mu$ be a non-negative dyadically  doubling measure in $Q^0$, and let $\nu$ be a non-negative Borel measure in $Q^0$.	
	\begin{list}{$(\theenumi)$}{\usecounter{enumi}\leftmargin=1cm \labelwidth=1cm \itemsep=0.1cm \topsep=.2cm \renewcommand{\theenumi}{\alph{enumi}}}
		
		\item If $\nu$ is dyadically doubling on $Q^0$ then $\mathcal{P}_{\mathcal{F}}^\mu\nu$  is dyadically doubling on $Q^0$.
		
		\item If  $\nu\in A_\infty^{\rm dyadic}(Q^0,\mu)$  then $\mathcal{P}_{\mathcal{F}}^\mu \nu\in A_\infty^{\rm dyadic}(Q^0,\mu)$.
	\end{list}
	
\end{lemma}

Let $\gamma=\{\gamma_Q\}_{Q\in\dd_{Q^0}}$ be a sequence of non-negative numbers. For any collection $\dd'\subset\dd_{Q^0}$, we define an associated ``discrete measure''
\begin{equation}
\mut_\gamma(\dd'):= \sum_{Q\in\dd'}\gamma_{Q}.
\label{eq:mut-defi}
\end{equation}
We say that $\mut_\gamma$ is a ``discrete Carleson measure'' (with respect to $\mu$) in $Q^0$,  if
\begin{equation}\label{eq6.0}
\|\mut_\gamma\|_{\C(Q^0,\mu)}
:= 
\sup_{Q\in\dd_{Q^0}} \frac{\mut_\gamma(\dd_{Q})}{\mu(Q)} 
=
\sup_{Q\in\dd_{Q^0}} \frac1{\mu(Q)}\sum_{Q'\in\dd_Q}\gamma_{Q'}<\infty.%
\end{equation}
For simplicity, when $Q^0=E$ we simply write $\|\mut_\gamma\|_{\C(\mu)}$. 

Given $\mathcal{F}=\{Q_i\}\subset\mathbb{D}_{Q^0}$, a (possibly empty) family of pairwise disjoint dyadic cubes, we define $\mathfrak{m}_{\gamma,\mathcal{F}}$ by
\begin{equation}\label{def-car-F}
\mathfrak{m}_{\gamma,\mathcal{F}}(\mathbb{D}')
=
\mathfrak{m}_\gamma(\mathbb{D}'\cap\mathbb{D}_\mathcal{F})
=
\sum_{Q\in\mathbb{D}'\cap\mathbb{D}_\mathcal{F}} \gamma_Q,\qquad \mathbb{D}'\subset\mathbb{D}_{Q^0}.
\end{equation}
Equivalently, $\mathfrak{m}_{\gamma, \mathcal{F}}=\mathfrak{m}_{\gamma_\F}$ where $\gamma_\F=\{\gamma_{\F,Q}\}_{Q\in \dd_{Q^0}}$ is given by 
\begin{equation}\label{gammauxiliar}
\gamma_{\mathcal{F},Q}=\left\{
\begin{array}{ll}
\gamma_Q & \hbox{ if $Q\in\mathbb{D}_{\mathcal{F},Q^0}$,} 
\\[5pt]
0 & \hbox{ if $Q\in\mathbb{D}_{Q_0}\setminus\mathbb{D}_{\mathcal{F},Q^0}$.}
\end{array}
\right.
\end{equation}
Note that if $\F=\emptyset$, then $\gamma_\F=\gamma$ and hence $\mathfrak{m}_{\gamma, \tinyemptyset}=\mathfrak{m}_{\gamma}$.

The following result was proved in \cite[Lemma 8.5]{HM1} under the additional assumption that $\pom$ is AR, however a careful inspection of the proof shows that the same argument can be carried out under the current assumption. 
We note that \cite[Lemma 8.5]{HM1} was formulated and proved in the  
case that $Q^0\in\dd$, but clearly that implies the case $Q^0=E$. We caution the reader to beware of the
distinction between sub- and super-script, $Q_0$ vs. $Q^0$, in the statement of the following lemma.

\begin{lemma}[{\cite[Lemma 8.5]{HM1}}]\label{lemma:extrapolation}
Given $Q^0$, let $\mu$, $\nu $ be a pair of non-negative dyadically doubling Borel measures on $Q^0$, and let $\mathfrak{m}_\gamma$ be a discrete Carleson measure with respect to $\mu$, with
	$$
	\|\mathfrak{m}_\gamma\|_{\mathcal{C}(Q^0,\mu)}\leq M_0.
	$$
	Suppose that there exists $\varepsilon$ such that for every $Q_0\in\mathbb{D}_{Q^0}$ and every family of pairwise disjoint dyadic cubes $\mathcal{F}=\{Q_i\}\subset\mathbb{D}_{Q_0}$ verifying
	$$
	\|\mathfrak{m}_{\gamma, \mathcal{F}}\|_{\mathcal{C}(Q_0,\mu)}
	=
	\sup_{Q\in\mathbb{D}_{Q_0}}\frac{\mathfrak{m}_\gamma (\mathbb{D}_{\mathcal{F},Q})}{\mu(Q)} \leq \varepsilon,
	$$
	we have that $\mathcal{P}_{\mathcal{F}}^\mu\nu$ satisfies the following property:
	\[
	\forall\zeta\in(0,1),\quad\exists\,C_\zeta>1\text{ such that }
	\Big(F\subset  Q_0,\quad\frac{\mu(F)}{\mu(Q_0)}\geq\zeta
	\ \implies\
	\frac{\mathcal{P}_{\mathcal{F}}^\mu\nu(F)}{\mathcal{P}_{\mathcal{F}}^\mu\nu(Q_0)}\geq\frac{1}{C_\zeta} \Big).
	\]
	Then, there exist $\eta_0\in(0,1)$ and $1<C_0<\infty$ such that, for every $Q_0\in\mathbb{D}_{Q^0}$
	$$
	F\subset Q_0,\quad\frac{\mu(F)}{\mu(Q_0)}\geq 1-\eta_0
	\quad\implies\quad
	\frac{\nu(F)}{\nu(Q_0)}\geq\frac{1}{C_0}.
	$$
	In other words, $\nu\in A_{\infty}^{\rm dyadic}(Q^0,\mu)$.
\end{lemma}

\subsection{Existence of a dyadic grid}\label{ss-dyadic}

In this section we introduce  a dyadic grid along the lines of that obtained in \cite{C}. More precisely,  we will use the dyadic structure from \cite{HK1, HK2}, with a modification from \cite[Proof of Proposition 2.12]{HMMM}:

\begin{lemma}[\textbf{Existence and properties of the ``dyadic grid''}]\label{lemma:dyadiccubes}
Let $E\subset\re^{n+1}$ be a closed set. Then there exists a constant $C\ge 1$ depending just on $n$ such that for each $k\in\mathbb{Z}$ there is a collection of Borel sets  (called ``cubes'')
	$$
	\mathbb{D}_k:=\big\{Q_j^k\subset E:\ j\in\mathfrak{J}_k\big\},
	$$
	where $\mathfrak{J}_k$ denotes some (possibly finite) index set depending on $k$ satisfying:
	\begin{list}{$(\theenumi)$}{\usecounter{enumi}\leftmargin=1cm \labelwidth=1cm \itemsep=0.2cm \topsep=.2cm \renewcommand{\theenumi}{\alph{enumi}}}
		\item $E=\bigcup_{j\in \mathfrak{J}_k}Q_j^k$ for each $k\in\mathbb{Z}$.
		\item If $m\le k$ then either $Q_j^k \subset Q_i^m$ or $Q_i^m\cap Q_j^k=\emptyset$.
		\item For each $k\in\mathbb{Z}$, $j\in\mathfrak{J}_k$, and $m<k$, there is a unique $i\in\mathfrak{J}_m $ such that $Q_j^k\subset Q_i^m$.
		\item For each  $k\in\mathbb{Z}$, $j\in\mathfrak{J}_k$ there is $x_j^k\in E$ such that
		\[B(x_j^k, C^{-1}2^{-k})\cap E\subset Q_j^k \subset B(x_j^k, C 2^{-k})\cap E.\]

	\end{list}

Moreover, assume that there is a Borel measure $\mu$ which is doubling, that is, there exists $C_\mu\ge 1$ such that $\mu(\Delta(x,2 r))\le C_\mu  \mu(\Delta(x,r))$ for every $x\in E$ and $r>0$. Then $\mu(\partial Q)=0$ for every $Q\in\dd_k$, $k\in\ZZ$. Furthermore, there exist $0<\tau_0<1$, $C_1$, and $\eta>0$ depending only on dimension and $C_\mu$ such that for every $\tau\in(0,\tau_0)$ and $Q\in\dd_k$, $k\in\ZZ$,
\begin{equation}\label{eqn:thin-boundary}
	\mu \big(\big\{x\in Q:\,\dist(x,E\setminus Q)\leq\tau \ell(Q)\big\}\big)
	\leq C_1
	\tau^\eta \mu(Q).
\end{equation}
\end{lemma}

In what follows given $B=B(x,r)$ with $x\in E$ we will denote $\Delta=\Delta(x,r)=B\cap E$. A few remarks are in order concerning this lemma.   We first observe that if $E$ is bounded and $k\in\ZZ$ is such that $\diam(E)<C^{-1}2^{-k}$, then there cannot be two distinct cubes in $\dd_k$. Thus, $\dd_k=\{Q^k\}$ with $Q^k=E$. Therefore, as explained in Remark \ref{remark:trunc-generations} we are going to ignore those $k\in\mathbb{Z}$ such that $2^{-k}\gtrsim\diam(E)$. Hence, we shall denote by $\mathbb{D}(E)$ the collection of all relevant $Q_j^k$, i.e., $\mathbb{D}(E):=\bigcup_k\mathbb{D}_k$, where, if $\diam(E)$ is finite, the union runs over those $k\in\mathbb{Z}$ such that $2^{-k}\lesssim\diam(E)$. For a dyadic cube $Q\in\mathbb{D}_k$, as explained above we shall set $\ell(Q)=2^{-k}$, and we shall refer to this quantity as the ``length'' of $Q$. It is clear from $(d)$ that $\diam(Q)\lesssim \ell(Q)$ (we will see below that in our setting the converse hold, see Remark  \ref{remark:diam-radius}). We write $\Xi=2C^2$, with $C$ being the constant in Lemma \ref{lemma:dyadiccubes}, which is a purely dimensional. For $Q\in\mathbb{D}(E)$ we will set $k(Q)=k$ if $Q\in\mathbb{D}_k$. Property $(d)$ implies that for each cube $Q\in\mathbb{D}$, there exist $x_Q\in E$ and $r_Q$, with $\Xi^{-1}\ell(Q)\leq r_Q\leq\ell(Q)$ (indeed $r_Q= (2C)^{-1}\ell(Q)$), such that
    \begin{equation}\label{deltaQ}
    \Delta(x_Q,2 r_Q)\subset Q\subset\Delta(x_Q,\Xi r_Q).
    \end{equation} 
    We shall denote these balls and surface balls by
    \begin{equation}\label{deltaQ2}
    B_Q:=B(x_Q,r_Q),\qquad\Delta_Q:=\Delta(x_Q,r_Q),
    \end{equation}
    \begin{equation}\label{deltaQ3}
    \widetilde{B}_Q:=B(x_Q,\Xi r_Q),\qquad\widetilde{\Delta}_Q:=\Delta(x_Q,\Xi r_Q),
    \end{equation}
    and we shall refer to the point $x_Q$ as the ``center'' of $Q$.
    
Let $Q\in\dd_k$ and consider  the family of its dyadic children $\{Q'\in \dd_{k+1}: Q'\subset Q\}$. Note that for any two distinct children $Q', Q''$, one has $|x_{Q'}-x_{Q''}|\ge r_{Q'}=r_{Q''}=r_Q/2$, otherwise $x_{Q''}\in Q''\cap \Delta_{Q'}\subset Q''\cap Q'$, contradicting the fact that $Q'$ and $Q''$ are disjoint. Also $x_{Q'}, x_{Q''}\in Q\subset \Delta(x_Q,r_Q)$, hence by the geometric doubling property we have a purely dimensional bound for the number of such $x_{Q'}$ and hence the number of dyadic children of a given dyadic cube is uniformly bounded.

\subsection{Sawtooth domains}\label{subsection:sawtooth}

In the sequel, $\Omega\subset\re^{n+1}$, $n\geq 2$, will be a 1-sided NTA domain satisfying the CDC. Write $\dd=\dd(\pom)$ for the dyadic grid obtained from Lemma \ref{lemma:dyadiccubes} with $E=\pom$. By Remark \ref{remark:diam-radius} and under the present assumptions one has that $\diam(\Delta)\approx r_{\Delta}$ for every surface ball $\Delta$. In particular $\diam(Q)\approx\ell(Q)$ for every $Q\in\dd$ in view of \eqref{deltaQ}. Given $Q\in\mathbb{D}$ we define the ``Corkscrew point relative to $Q$'' as $X_Q:=X_{\Delta_Q}$. We note that
    $$
    \delta(X_Q)\approx\dist(X_Q,Q)\approx\diam(Q).
    $$

Much as we did in Section \ref{subsection:dyadic-analysis} of, given $Q\in\mathbb{D}$ and $\F$ a possibly empty family of pairwise disjoint dyadic cubes, we can define $\mathbb{D}_Q$, the ``discretized Carleson region''; $\mathbb{D}_\mathcal{F}$, the  ``global discretized sawtooth'' relative to $\mathcal{F}$; and $\mathbb{D}_{\mathcal{F},Q}$, the ``local discretized sawtooth'' relative to $\mathcal{F}$. Note that if $\F$ to be the null set in which case $\mathbb{D}_{\tinyemptyset}=\dd$ and $\mathbb{D}_{\tinyemptyset,Q}=\dd_Q$.

We also introduce the ``geometric'' Carleson regions and sawtooths.  Given $Q\in\mathbb{D}$ we want to define some associated regions which inherit the good properties of $\Omega$. Let $\mathcal{W}=\mathcal{W}(\Omega)$ denote a collection of (closed) dyadic Whitney cubes of $\Omega\subset\re^{n+1}$, so that the cubes in $\mathcal{W}$  
form a covering of $\Omega$ with non-overlapping interiors, and satisfy
\begin{equation}\label{constwhitney}
4\diam(I)\leq\dist(4I,\partial\Omega)\leq\dist(I,\partial\Omega)\leq 40\diam(I),\qquad\forall I\in\mathcal{W},
\end{equation}
and
$$
\diam(I_1)\approx\diam(I_2),\,\text{ whenever }I_1\text{ and }I_2\text{ touch}.
$$
Let $X(I)$ denote the center of $I$, let $\ell(I)$ denote the side length of $I$, and write $k=k_I$ if $\ell(I)=2^{-k}$.

Given $0<\lambda<1$ and $I\in\mathcal{W}$ we write $I^*=(1+\lambda)I$ for the ``fattening'' of $I$. By taking $\lambda$ small enough, we can arrange matters, so that, first, $\dist(I^*,J^*)\approx\dist(I,J)$ for every $I,J\in\mathcal{W}$. Secondly, $I^*$ meets $J^*$ if and only if $\partial I$ meets $\partial J$ (the fattening thus ensures overlap of $I^*$ and $J^*$ for any pair $I,J\in\mathcal{W}$ whose boundaries touch, so that the Harnack Chain property then holds locally in $I^*\cup J^*$, with constants depending upon $\lambda$). By picking $\lambda$ sufficiently small, say $0<\lambda<\lambda_0$, we may also suppose that there is $\tau\in(\frac12,1)$ such that for distinct $I,J\in\mathcal{W}$, we have that $\tau J\cap I^*=\emptyset$. In what follows we will need to work with dilations $I^{**}=(1+2\lambda)I$ or $I^{***}=(1+4\lambda)I$, and in order to ensure that the same properties hold we further assume that $0<\lambda<\lambda_0/4$.

For every $Q\in\mathbb{D}$ we can construct a family $\mathcal{W}_Q^*\subset\mathcal{W}(\Omega)$, and define
$$
U_Q:=\bigcup_{I\in\mathcal{W}_Q^*}I^*,
$$
satisfying the following properties: $X_Q\in U_Q$ and there are uniform constants $k^*$ and $K_0$ such that
\begin{align}
\label{kstar_K0}
\begin{split}
k(Q)-k^*\leq k_I\leq k(Q)+k^*,\quad\forall I\in\mathcal{W}_Q^*,
\\[4pt]
X(I)\rightarrow_{U_Q} X_Q,\quad\forall I\in\mathcal{W}_Q^*,
\\[4pt]
\dist(I,Q)\leq K_0 2^{-k(Q)},\quad\forall I\in\mathcal{W}_Q^*.
\end{split}
\end{align}
Here, $X(I)\rightarrow_{U_Q} X_Q$ means that the interior of $U_Q$ contains all balls in a Harnack Chain (in $\Omega$) connecting $X(I)$ to $X_Q$, and moreover, for any point $Z$ contained in any ball in the Harnack Chain, we have $\dist(Z,\partial\Omega)\approx\dist(Z,\Omega\setminus U_Q)$ with uniform control of the implicit constants. The constants $k^*, K_0$ and the implicit constants in the condition $X(I)\rightarrow_{U_Q} X_Q$, depend on the allowable parameters and on $\lambda$. Moreover, given $I\in\mathcal{W}(\Omega)$ we have that $I\in\mathcal{W}_{Q_I}^*$, where $Q_I\in\dd$ satisfies $\ell(Q_I)=\ell(I)$, and contains any fixed $\widehat{y}\in\partial\Omega$ such that $\dist(I,\partial\Omega)=\dist(I,\widehat{y})$. The reader is referred to \cite{HM1, HMT1} for full details.

For a given $Q\in\mathbb{D}$, the ``Carleson box'' relative to $Q$ is defined by
$$
T_Q:=\interior\bigg(\bigcup_{Q'\in\mathbb{D}_Q}U_{Q'}\bigg).
$$
For a given family $\mathcal{F}=\{Q_i\}\subset\dd$ of pairwise disjoint cubes and a given $Q\in\mathbb{D}$, we define the ``local sawtooth region'' relative to $\mathcal{F}$ by
\begin{equation}
\label{defomegafq}
\Omega_{\mathcal{F},Q}=\interior\bigg(\bigcup_{Q'\in\mathbb{D}_{\mathcal{F},Q}}U_{Q'}\bigg)=\interior\bigg(\bigcup_{I\in\mathcal{W}_{\mathcal{F},Q}}I^*\bigg),
\end{equation}
where $\mathcal{W}_{\mathcal{F},Q}:=\bigcup_{Q'\in\mathbb{D}_{\mathcal{F},Q}}\mathcal{W}_Q^*$. Note that in the previous definition we may allow $\F$ to be empty in which case clearly $\Omega_{\tinyemptyset ,Q}=T_Q$. Similarly, the ``global sawtooth region'' relative to $\mathcal{F}$ is defined as
\begin{equation}
\label{defomegafq-global}
\Omega_{\mathcal{F}}=\interior\bigg(\bigcup_{Q'\in\mathbb{D}_{\mathcal{F}}}U_{Q'}\bigg)=\interior\bigg(\bigcup_{I\in\mathcal{W}_{\mathcal{F}}}I^*\bigg),
\end{equation}
where $\mathcal{W}_{\mathcal{F}}:=\bigcup_{Q'\in\mathbb{D}_{\mathcal{F}}}\mathcal{W}_Q^*$. If $\F$ is the empty set clearly $\Omega_{\tinyemptyset}=\Omega$.
For a given $Q\in\dd$  and $x\in \pom$ let us introduce the ``truncated dyadic cone'' 
\[
\Gamma_{Q}(x) := \bigcup_{x\in Q'\in\mathbb{D}_{Q}}  U_{Q'},
\]
where it is understood that $\Gamma_{Q}(x)=\emptyset$ if $x\notin Q$. 
Analogously, we can slightly fatten the Whitney boxes and use $I^{**}$ to define new fattened Whitney regions and sawtooth domains. More precisely, for every $Q\in\dd$,
\[
T_Q^*:=\interior\bigg(\bigcup_{Q'\in\mathbb{D}_Q}U_{Q'}^*\bigg),\quad\Omega^*_{\mathcal{F},Q}:=\interior\bigg(\bigcup_{Q'\in\mathbb{D}_{\F,Q}}U_{Q'}^*\bigg), \quad
\Gamma^*_{Q}(x) := \bigcup_{x\in Q'\in\mathbb{D}_{Q_0}}  U_{Q'}^*
\]
where $U_{Q}^*:=\bigcup_{I\in\mathcal{W}_Q^*}I^{**}$.
Similarly, we can define $T_Q^{**}$, $\Omega^{**}_{\mathcal{F},Q}$, $\Gamma_Q^{**}(x)$, and $U^{**}_{Q}$ by using $I^{***}$ in place of $I^{**}$.

Given $Q$ we next define the ``localized dyadic non-tangential maximal function''
\begin{equation}\label{def:NT}
\mathcal{N}_{Q}u(x) 
: = 
\sup_{Y\in \Gamma^*_{Q}(x)} |u(Y)|,
\qquad x\in \pom,
\end{equation}
for every $u\in C(T_{Q}^*)$, where it is understood that $\mathcal{N}_{Q}u(x)= 0$ for every $x\in\pom\setminus Q$ (since $\Gamma_Q^*(x)=\emptyset$ in such a case). 
Finally, let us introduce the ``localized  dyadic conical square function''
\begin{equation}\label{def:SF}
\mathcal{S}_{Q}u(x):=\bigg(\iint_{\Gamma_{Q}(x)}|\nabla u(Y)|^2\delta(Y)^{1-n}\,dY\bigg)^{\frac12}, \qquad x\in \pom,
\end{equation}
for every $u\in W^{1,2}_{\rm loc}  (T_{Q_0})$. Note that again $\mathcal{S}_{Q}u(x)=0$ for every $x\in\pom\setminus Q$.

To define  the ``Carleson box'' $T_\Delta$ associated with a surface ball $\Delta=\Delta(x,r)$, let $k(\Delta)$ denote the unique $k\in\mathbb{Z}$ such that $2^{-k-1}<200r\leq 2^{-k}$, and set
\begin{equation}\label{D-delta}
\mathbb{D}^{\Delta}:=\big\{Q\in\mathbb{D}_{k(\Delta)}:\:Q\cap 2\Delta\neq\emptyset\big\}.
\end{equation}
We then define
\begin{equation}
\label{def:T-Delta}
T_{\Delta}:=\interior\bigg(\bigcup_{Q\in\mathbb{D}^\Delta}\overline{T_Q}\bigg).
\end{equation}
We can also consider fattened versions of $T_\Delta$ given by
$$
T_{\Delta}^*:=\interior\bigg(\bigcup_{Q\in\mathbb{D}^\Delta}\overline{T_Q^*}\bigg),\qquad T_{\Delta}^{**}:=\interior\bigg(\bigcup_{Q\in\mathbb{D}^\Delta}\overline{T_Q^{**}}\bigg).
$$

Following \cite{HM1, HMT1}, one can easily see that there exist constants $0<\kappa_1<1$ and $\kappa_0\geq 16\Xi$ (with $\Xi$ the constant in \eqref{deltaQ}), depending only on the allowable parameters, so that
\begin{gather}\label{definicionkappa12}
\kappa_1B_Q\cap\Omega\subset T_Q\subset T_Q^*\subset T_Q^{**}\subset \overline{T_Q^{**}}\subset\kappa_0B_Q\cap\overline{\Omega}=:\tfrac{1}{2}B_Q^*\cap\overline{\Omega},
\\[6pt]
\label{definicionkappa0}
\tfrac{5}{4}B_\Delta\cap\Omega\subset T_\Delta\subset T_\Delta^*\subset T_\Delta^{**}\subset\overline{T_\Delta^{**}}\subset\kappa_0B_\Delta\cap\overline{\Omega}=:\tfrac{1}{2}B_\Delta^*\cap\overline{\Omega},
\end{gather}
and also
\begin{equation}\label{propQ0}
Q\subset\kappa_0B_\Delta\cap\partial\Omega=\tfrac{1}{2}B_\Delta^*\cap\partial\Omega=:\tfrac{1}{2}\Delta^*,\qquad\forall\,Q\in\mathbb{D}^{\Delta},
\end{equation}
where $B_Q$ is defined as in \eqref{deltaQ2}, $\Delta=\Delta(x,r)$ with $x\in\partial\Omega$, $0<r<\diam(\partial \Omega)$, and $B_{\Delta}=B(x,r)$ is so that $\Delta=B_\Delta\cap\partial\Omega$. From our choice of the parameters one also has that $B_Q^*\subset B_{Q'}^*$ whenever $Q\subset Q'$.

In the remainder of this section we show that if $\Omega$ is a 1-sided NTA domain satisfying the CDC then Carleson boxes and local and global sawtooth domains are also 1-sided NTA domains satisfying the CDC. We next present some of the properties of the capacity which will be used in our proofs. From the definition of capacity one can easily see that given a ball $B$ and compact sets $F_1\subset F_2\subset \overline{B}$ then
\begin{equation}\label{cap:prop1}
\Cap(F_1, 2B)\le \Cap (F_2, 2B).
\end{equation}
Also, given two balls $B_1\subset B_2$ and a compact set $F\subset \overline{B_1}$ then 
\begin{equation}\label{cap:prop2}
\Cap(F, 2B_2)\le \Cap (F, 2B_1).
\end{equation}
On the other hand, \cite[Lemma 2.16]{HKM} gives that if $F$ is a compact with $F\subset \overline{B}$ then there is a dimensional constant $C_n$ such that
\begin{equation}\label{cap:prop3}
C_n^{-1}\Cap(F, 2B)\le \Cap (F, 4B)\le \Cap (F, 2B).
\end{equation}

\begin{lemma}\label{lemma:CDC-inherit}
	Let $\Omega\subset\mathbb{R}^{n+1}$, $n\ge 2$, be a 1-sided NTA domain satisfying the CDC. Then all of its Carleson boxes $T_Q$ and $T_\Delta$, and sawtooth regions $\Omega_\F$, and $\Omega_{\F,Q}$ are 1-sided NTA domains and satisfy the CDC with uniform implicit constants depending only on dimension and on the corresponding
	constants for $\Omega$.
\end{lemma}

\subsection[Elliptic operators, elliptic measure, and the Green function]{Uniformly elliptic operators, elliptic measure, and the Green function}
Next, we recall several facts concerning elliptic measure and the Green functions. To set the stage let $\Omega\subset\re^{n+1}$ be an open set. Throughout we consider  
elliptic operators $L$ of the form $Lu=-\div(A\nabla u)$ with $A(X)=(a_{i,j}(X))_{i,j=1}^{n+1}$ being a real (non-necessarily symmetric) matrix such that $a_{i,j}\in L^{\infty}(\Omega)$ and there exists $\Lambda\geq 1$ such that the following uniform ellipticity condition holds 
\begin{align}
\label{e:elliptic}
\Lambda^{-1} |\xi|^{2} \leq A(X) \xi \cdot \xi,
\qquad\qquad
|A(X) \xi \cdot\eta|\leq \Lambda |\xi|\,|\eta| 
\end{align}
for all $\xi,\eta \in\mathbb{R}^{n+1}$ and for almost every $X\in\Omega$. We write $L^\top$ to denote the transpose of $L$, or, in other words, $L^\top u = -\div(A^\top
\nabla u)$ with $A^\top$ being the transpose matrix of $A$.

We say that $u$ is a weak solution to $Lu=0$ in $\Omega$ provided that $u\in W_{\rm loc}^{1,2}(\Omega)$ satisfies
\[
\iint A(X)\nabla u(X)\cdot \nabla\phi(X) dX=0  \quad\mbox{whenever}\,\, \phi\in C^{\infty}_{0}(\Omega).
\]
Associated with $L$ one can construct an elliptic measure $\{\omega_L^X\}_{X\in\Omega}$ and a Green function $G_L$ (see \cite{HMT1} for full details). Sometimes, in order to emphasize the dependence on $\Omega$, we will write $\omega_{L,\Omega}$ and $G_{L,\Omega}$. If $\Omega$ satisfies the CDC then it follows that all boundary points are Wiener regular and hence for a given $f\in C_c(\partial\Omega)$ we can define
\[
u(X)=\int_{\partial\Omega} f(z)d\omega^{X}_{L}(z), \quad \mbox{whenever}\, \, X\in\Omega,
\]
so that  $u\in W^{1,2}_{\rm loc}(\Omega)\cap C(\overline{\Omega})$ satisfying $u=f$ on $\partial\Omega$ and $Lu=0$ in the weak sense. Moreover, if $f\in \Lip(\Omega)$ then $u\in W^{1,2}(\Omega)$.

We first define the reverse Hölder class and the $A_\infty$ classes with respect to fixed elliptic measure in $\Omega$.  One reason we take this approach is that we do not know whether $\mathcal{H}^{n}|_{\partial\Omega}$ is well-defined since we do not assume any Ahlfors regularity in Theorem \ref{thm:main}. Hence we have to develop these notions in terms of elliptic measures. To this end, let $\Omega$ satisfy the CDC and let $L_0$ and $L$ be two real (non-necessarily symmetric) elliptic operators associated with $L_0u=-\div(A_0\nabla u)$ and $L u=-\div(A\nabla u)$ where $A$ and $A_0$ satisfy \eqref{e:elliptic}. Let $\omega^{X}_{0}$ and $\omega_{L}^{X}$ be the elliptic measures of $\Omega$ associated with the operators $L_0$ and $L$ respectively with pole at $X\in\Omega$. Note that if we further assume that $\Omega$ is connected then $\omega_{L}^{X}\ll\omega_L^{Y}$ on $\pom$ for every $X,Y\in\Omega$. Hence if $\omega_L^{X_0}\ll\omega_{L_0}^{Y_0}$ on $\pom$  for some $X_0,Y_0\in\Omega$ then $\omega_L^{X}\ll\omega_{L_0}^{Y}$ on $\pom$ for every $X,Y\in\Omega$ and thus  we can simply write $\omega_{L}\ll \omega_{L_0}$ on $\pom$. In the latter case we will use the notation
\begin{equation}\label{def-RN}
h(\cdot\,;L, L_0, X)=\frac{d\omega_L^{X}}{d\omega_{L_0}^{X}}
\end{equation}
to denote the Radon-Nikodym derivative of $\omega_{L}^{X}$ with respect to $\omega_{L_0}^{X}$,
which is a well-defined function $\omega_{L_0}^{X}$-almost everywhere on $\pom$.

\begin{definition}[Reverse Hölder and $A_\infty$ classes]\label{d:RHp}
	Fix $\Delta_0=B_0\cap \pom$ where $B_0=B(x_0,r_0)$ with $x_0\in\pom$ and $0<r_0<\diam(\pom)$. Given $p$, $1<p<\infty$, we say that $\omega_L\in RH_p(\Delta_0,\omega_{L_0})$, provided that $\omega_L\ll \omega_{L_0}$ on $\Delta_0$, and there exists $C\geq 1$ such that 
\begin{align*}
	\left(\aver{\Delta}h(y;L,L_0,X_{\Delta_0} )^p d \omega_{L_0}^{X_{\Delta_0}}(y)\right)^{\frac1p} 
	\leq 
	C 
	\aver{\Delta} h(y;L,L_0,X_{\Delta_0} ) d \omega_{L_0}^{X_{\Delta_0}}(y)
	=
	C\frac{\omega_L^{X_{\Delta_0}}(\Delta)}{\omega_{L_0}^{X_{\Delta_0}}(\Delta)},
\end{align*}
	for every $\Delta=B\cap \partial\Omega$ where $B\subset B(x_0,r_0)$, $B=B(x,r)$ with  $x\in \partial\Omega$, $0<r<\diam(\partial\Omega)$. The infimum of the constants $C$ as above is denoted by $[\omega_{L}]_{RH_p(\Delta_0,\omega_{L_0})}$. 
	
	Similarly, we say that $\omega_L\in RH_p(\pom,\omega_{L_0})$ provided that for every $\Delta_0=\Delta(x_0,r_0)$ with $x_0\in\pom$ and $0<r_0<\diam(\pom)$ one has $\omega_L\in RH_p(\Delta_0,\omega_{L_0})$ uniformly on $\Delta_0$, that is, 
	\[
	[\omega_{L}]_{RH_p(\pom,\omega_{L_0})}
	:=\sup_{\Delta_0} [\omega_{L}]_{RH_p(\Delta_0,\omega_{L_0})}<\infty.
	\]

	Finally,
	\[
	A_\infty(\Delta_0,\omega_{L_0})=\bigcup_{p>1} RH_p(\Delta_0,\omega_{L_0})
	\quad\mbox{and}\quad
	A_\infty(\partial\Omega,\omega_{L_0})=\bigcup_{p>1} RH_p(\partial\Omega,\omega_{L_0})
	.\]
\end{definition}

The following lemmas state some properties for the Green functions and elliptic measures, proofs may be found in \cite{HMT1}. 

\begin{lemma}\label{lemma:Greensf}
	Suppose that $\Omega\subset\re^{n+1}$, $n\ge 2$, is an open set satisfying the CDC. Given a real (non-necessarily symmetric) elliptic operator $L=-\div(A\nabla)$, there exist $C>1$ (depending only on dimension and on the ellipticity constant of $L$) such that $G_L$,  the Green function associated with $L$, satisfies
\begin{gather}\label{sizestimate}
	0\le G_L(X,Y)\leq C|X-Y|^{1-n},\quad\forall X,Y\in\Omega,\quad X\neq Y;
	\\[0.15cm] 
	G_L(\cdot,Y)\in  W_{\rm loc}^{1,2}(\Omega\setminus\{Y\})\cap C\big(\overline{\Omega}\setminus\{Y\}\big)\quad\text{and}\quad G_L(\cdot,Y)|_{\partial\Omega}\equiv 0\quad\forall Y\in\Omega;
	\\[0.15cm]\label{G-G-top}
	G_L(X,Y)=G_{L^\top}(Y,X),\quad\forall X,Y\in\Omega,\quad X\neq Y;
	\\[0.15cm]
	\label{eq:G-delta}
	\iint_{\Omega}A(X)\nabla_X G_L(X,Y)\cdot\nabla\varphi(X)\,dX=\varphi(Y),\qquad\forall\, \varphi\in  C_c^{\infty}(\Omega).
\end{gather}
\end{lemma}
\medskip

\begin{remark}\label{rem:GF}
	If we also assume that $\Omega$ is bounded, following \cite{HMT1} we know that the Green function $G_L$ coincides with the one constructed in \cite{gruterwidman}. Consequently, $G_L(\cdot,Y)\in W^{1,2}(\Omega\setminus B(Y,r))\cap W_0^{1,1}(\Omega)$  
	Moreover, for every $\varphi\in C_c^{\infty}(\Omega)$ such that $0\le \varphi\le1 $ and $\varphi\equiv 1$ in $B(Y,r)$ with $0<r<\delta(Y)$, we have that $	(1-\varphi)G_L(\cdot,Y)\in W_0^{1,2}(\Omega)$.
\end{remark}
\medskip

The following result lists some properties which will be used throughout the paper: 

\begin{lemma}\label{lemma:proppde}
	Suppose that $\Omega\subset\re^{n+1}$, $n\ge 2$, is a 1-sided NTA domain satisfying the CDC. Let $L_0=-\div(A_0\nabla)$ and $L=-\div(A\nabla)$ be two real (non-necessarily symmetric) elliptic operators, there exist $C_1\ge 1$, $\rho\in (0,1)$ (depending only on dimension, the 1-sided NTA constants, the CDC constant, and the ellipticity of $L$) and $C_2\ge 1$ (depending on the same parameters and on the ellipticity of $L_0$), such that for every $B_0=B(x_0,r_0)$ with $x_0\in\partial\Omega$, $0<r_0<\diam(\partial\Omega)$, and $\Delta_0=B_0\cap\partial\Omega$ we have the following properties:
	\begin{list}{$(\theenumi)$}{\usecounter{enumi}\leftmargin=1cm \labelwidth=1cm \itemsep=0.1cm \topsep=.2cm \renewcommand{\theenumi}{\alph{enumi}}}
		
		\item $\omega_L^Y(\Delta_0)\geq C_1^{-1}$ for every $Y\in C_1^{-1}B_0\cap\Omega$ and $\omega_L^{X_{\Delta_0}}(\Delta_0)\ge C_1^{-1}$.

			\item If $B=B(x,r)$ with $x\in\partial\Omega$ and $\Delta=B\cap\partial\Omega$ is such that $2B\subset B_0$, then for all $X\in\Omega\setminus B_0$ we have that $
		{C_1^{-1}}\omega_L^X(\Delta)\leq r^{n-1} G_L(X,X_\Delta)\leq C_1\omega_L^X(\Delta)$.
		
		\item  If $X\in\Omega\setminus 4B_0$, then $\omega_{L}^X(2\Delta_0)\leq C_1\omega_{L}^X(\Delta_0)$.

		\item   For every $X\in\Omega\setminus 2\kappa_0B_0$ with $\kappa_0$ as in \eqref{definicionkappa0}, we have that
		\[
		\frac1C_1 \frac1{\omega_L^X(\Delta_0)}\le  \frac{d\omega_L^{{X_{\Delta_0}}}}{d\omega_L^X}(y)\le C_1 \frac1{\omega_L^X(\Delta_0)},
		\qquad\mbox{for $\omega_L^X$-a.e. $y\in\Delta_0$}.
		\]

		\item  If $B=B(x,r)$ with $x\in\Delta_0$, $0<r<r_0/4$ and $\Delta=B\cap\partial\Omega$, then 
			$$
		\frac{1}{C_1}\omega_{L,\Omega}^{X_{\Delta}}(F)\leq\omega_{L,T_{\Delta_0}}^{X_{\Delta}}(F)\leq C_1\omega_{L,\Omega}^{X_{\Delta}}(F),\quad\text{ for every Borel set }F\subset\Delta.
		$$

		\item  If $L\equiv L_0$ in $B(x_0,2\kappa_0r_0)\cap\Omega$ with $\kappa_0$ as in \eqref{definicionkappa0}, then
		$$
		\frac{1}{C_2}\omega_{L_0}^{X_{\Delta_0}}(F)\leq\omega_{L}^{X_{\Delta_0}}(F)\leq C_2\omega_{L_0}^{X_{\Delta_0}}(F),\quad\text{ for every Borel set }F\subset\Delta_0.
		$$

%
	\end{list}
\end{lemma}

\medskip

\begin{remark}\label{remark:chop-dyadic}
	We note that from $(d)$ in the previous result, Harnack's inequality, and \eqref{deltaQ} one can easily see that 
	\begin{equation}\label{chop-dyadic:densities}
		\frac{d\omega_L^{X_{Q'}}}{d\omega_L^{X_{Q''}}}(y)
		\approx 
		\frac1{\omega_L^{X_{Q''}}(Q')},
		\qquad
		\mbox{ for $\omega_L^{X_{Q''}}$-a.e. }y\in Q',
		\mbox{whenever }Q'\subset Q''\in\dd.
	\end{equation}		
	Observe that since $\omega_L^{X_{Q''}}\ll \omega_L^{X_{Q'}}$  an analogous inequality for the reciprocal of the Radon-Nikodym derivative follows immediately. 
\end{remark}

We close this section by stating a dyadic versions of the main lemma in \cite{DJK}. To set the stage we first quote some auxiliary result:

\begin{proposition}[{\cite[Proposition 6.7]{HM1}, \cite[Proposition 3.1]{AHMT-I}}]\label{prop:Pi-proj}
	Let $\Omega\subset\ree$, $n\ge 2$, be a 1-sided NTA domain satisfying the CDC. Fix $Q_0\in \dd $ and let $\mathcal{F}=\{Q_k\}_k \subset \mathbb{D}_{Q_0}$ be a family of pairwise disjoint dyadic cubes. There exists $Y_{Q_0}\in \Omega\cap \Omega_{\F,Q_0}\cap \Omega_{\F,Q_0}^*$ so that
	\begin{equation}\label{eq:common-cks}
		\dist(Y_{Q_0},\pom)\approx \dist(Y_{Q_0},\pom_{\F,Q_0})\approx \dist(Y_{Q_0},\pom_{\F,Q_0}^*)\approx \ell(Q_0),
	\end{equation}
	where the implicit constants depend only on dimension, the 1-sided NTA constants, the CDC constant, and is  independent of $Q_0$ and $\F$. 
	Additionally, for each $Q_j\in\mathcal{F}$, there is an $n$-dimensional cube $P_j\subset\partial\Omega_{\mathcal{F},Q_0}$, which is contained in a face of $I^*$ for some $I\in\mathcal{W}$, and  which satisfies
	\begin{equation}\label{props:Pj}
		\ell(P_j)\approx \dist(P_j,Q_j)\approx \dist(P_j,\partial\Omega)\approx \ell(I)\approx \ell(Q_j),
	\end{equation}
	and $\sum_{j} 1_{P_j} \lesssim 1$, where the implicit constants depend on allowable parameters.   
\end{proposition}

We are now ready to state the a version of \cite[Lemma 6.15]{HM1} (see also \cite{DJK}) valid in our setting:

\begin{lemma}[Discrete sawtooth lemma for projections, {\cite[Lemma 3.5]{AHMT-I}}]\label{lemma:DJK-sawtooth}
	Suppose that $\Omega\subset\re^{n+1}$, $n\ge 2$, is a \textbf{bounded} 1-sided NTA domain satisfying the CDC. Let $Q_0\in\mathbb{D}$, let $\mathcal{F}=\{Q_i\}\subset\mathbb{D}_{Q_0}$ be a family of pairwise disjoint dyadic cubes, and let $\mu$ be a dyadically doubling measure in $Q_0$. Given two real (non-necessarily symmetric) elliptic $L_0$, $L$, we write  $\omega_0^{Y_{Q_0}}=\omega_{L_0,\Omega}^{Y_{Q_0}}$,  $\omega_L^{Y_{Q_0}}=\omega_{L,\Omega}^{Y_{Q_0}}$ for the elliptic measures associated with $L_0$ and $L$ for the domain $\Omega$  with fixed pole at $Y_{Q_0}\in\Omega_{\mathcal{F},Q_0}\cap\Omega$ (cf.~Proposition~\ref{prop:Pi-proj}). Let $\omega_{L,*}^{Y_{Q_0}}=\omega_{L,\Omega_{\mathcal{F},Q_0}}^{Y_{Q_0}}$ be the elliptic measure associated with $L$ for the domain $\Omega_{\mathcal{F},Q_0}$ with fixed pole at  $Y_{Q_0}\in\Omega_{\mathcal{F},Q_0}\cap\Omega$. Consider $\nu_L^{Y_{Q_0}}$ the measure defined by
	\begin{equation}\label{eq:def-nu}
		\nu_L^{Y_{Q_0}}(F)=\omega_{L,*}^{Y_{Q_0}}\Big(F\setminus\bigcup_{Q_i\in\mathcal{F}}Q_i\Big)+\sum_{Q_i\in\mathcal{F}}\frac{\omega_L^{Y_{Q_0}}(F\cap Q_i)}{\omega_L^{Y_{Q_0}}(Q_i)}\omega_{L,*}^{Y_{Q_0}}(P_i),\qquad F\subset Q_0,
	\end{equation}
	where $P_i$ is the cube produced in Proposition \ref{prop:Pi-proj}. Then $\mathcal{P}_{\mathcal{F}}^{\mu}\nu_L^{Y_{Q_0}}$ (see \eqref{defprojection}) depends only on $\omega_0^{Y_{Q_0}}$ and $\omega_{L,*}^{Y_{Q_0}}$, but not on $\omega_L^{Y_{Q_0}}$. More precisely,
	\begin{equation}\label{eq:def-nu:P}
		\mathcal{P}_{\mathcal{F}}^{\mu}\nu_L^{Y_{Q_0}}(F)
		=
		\omega_{L,*}^{Y_{Q_0}}\Big(F\setminus\bigcup_{Q_i\in\mathcal{F}}Q_i\Big)+\sum_{Q_i\in\mathcal{F}}\frac{\mu(F\cap Q_i)}{\mu(Q_i)}\omega_{L,*}^{Y_{Q_0}}(P_i),\qquad F\subset Q_0.
	\end{equation}
	Moreover, there exists $\theta>0$ such that for all $Q\in\mathbb{D}_{Q_0}$ and all $F\subset Q$, we have
	\begin{equation}\label{ainfsawtooth}
		\bigg(\frac{\mathcal{P}_{\mathcal{F}}^{\mu}\omega_L^{Y_{Q_0}}(F)}{\mathcal{P}_{\mathcal{F}}^{\mu}\omega_L^{Y_{Q_0}}(Q)}\bigg)^\theta
		\lesssim
		\frac{\mathcal{P}_{\mathcal{F}}^{\mu}\nu_L^{Y_{Q_0}}(F)}{\mathcal{P}_{\mathcal{F}}^{\mu}\nu_L^{Y_{Q_0}}(Q)}
		\lesssim
		\frac{\mathcal{P}_{\mathcal{F}}^{\mu}\omega_L^{Y_{Q_0}}(F)}{\mathcal{P}_{\mathcal{F}}^{\mu}\omega_L^{Y_{Q_0}}(Q)}.
	\end{equation}
\end{lemma}

\section{Proofs of the main results}\label{section:main-proof}

In order to prove Theorem \ref{thm:main} we are going to obtain a local version valid for bounded domains, interesting on its own right, which in turn will imply the desired results.

\begin{proposition}\label{PROP:LOCAL-VERSION}
Let $\Omega\subset\mathbb{R}^{n+1}$, $n\ge 2$, be a \textbf{bounded} 1-sided NTA domain satisfying the CDC. 
Let $Lu=-\div(A\nabla u)$ and $L_0u=-\div(A_0\nabla u)$ be two real (non-necessarily symmetric) elliptic operators. Fix $x_0\in\pom$ and $0<r_0<\diam(\pom)$ and let $
B_0=B(x_0,r_0)$, $\Delta_0=B_0\cap\pom$. Set
\begin{equation}\label{def-varrho-local}
\vertiii{\varrho(A,A_0)}_{B_0}
:=
\sup_{B}
\frac{1}{\omega_{L_0}^{X_{\Delta_0}} (\Delta)}
\iint_{B\cap\Omega}\varrho(A,A_0)(X)^2\frac{G_{L_0}(X_{\Delta_0},X)}{\delta(X)^2}\,dX,
\end{equation}
where $\varrho(A, A_0)$ was defined in \eqref{discrepancia}, $\Delta=B\cap\pom$, and the sup is taken over all balls $B=B(x,r)$ with $x\in 2\Delta_0$ and $0<r<r_0c_0/4$ \textup{(}$c_0$ is the Corkscrew constant\textup{)}.

\begin{list}{$(\theenumi)$}{\usecounter{enumi}\leftmargin=1cm \labelwidth=1cm \itemsep=0.1cm \topsep=.2cm \renewcommand{\theenumi}{\alph{enumi}}}
	
	\item If $\vertiii{\varrho(A, A_0)}_{B_0}<\infty$, then $\omega_L\in A_\infty(\Delta_0,\omega_{L_0})$, that is, there exists $1<q<\infty$ such that $\omega_L\in RH_q(\Delta_0, \omega_{L_0})$. Here, $q$ and the implicit constant depend only on dimension, the 1-sided NTA and CDC constants, the ellipticity constants of $L_0$ and $L$, and $\vertiii{\varrho(A, A_0)}_{B_0}$.

	\item Given $1<p<\infty$, there exists $\varepsilon_p>0$ (depending only on $p$, dimension, the 1-sided NTA and CDC constants and the ellipticity constants of $L_0$ and $L$) such that if one has $\vertiii{\varrho(A, A_0)}_{B_0} \leq\varepsilon_p$, then $\omega_L\in RH_p(\Delta_0,\omega_{L_0})$, with the implicit constant depending only on $p$, dimension, the 1-sided NTA and CDC constants, and the ellipticity constant of $L_0$ and $L$.
\end{list}
\end{proposition}

Assuming this result momentarily we can prove Theorem \ref{thm:main}:

\begin{proof}[Proof of Theorem \ref{thm:main}, part $(a)$]

\
	
\noindent\textbf{Case 1:} $\Omega$ \textbf{bounded}. 

For every ball $B_0=B(x_0,r_0)$ with $x_0\in\pom$ and $0<r_0<\diam(\pom)$, we clearly have $\vertiii{\varrho(A, A_0)}_{B_0}\le \vertiii{\varrho(A, A_0)}<\infty$. We can then invoke Proposition \ref{PROP:LOCAL-VERSION} part $(a)$ to find $q$, $1<q<\infty$, such that $\omega_L\in RH_q(\Delta_0, \omega_{L_0})$. Moreover, since  $\sup_{B_0}\vertiii{\varrho(A, A_0)}_{B_0}\le \vertiii{\varrho(A, A_0)}$ then the same $q$ is valid for every $B_0$ and also 
$\sup_{\Delta_0} [\omega_{L}]_{RH_q(\Delta_0,\omega_{L_0})}<\infty$. This means that 
$\omega_L\in RH_q(\pom,\omega_{L_0})$ and hence $\omega_L\in A_\infty(\pom,\omega_{L_0})$.

\medskip

\noindent\textbf{Case 2:}  $\Omega$ \textbf{unbounded}. 

Fix $B_0=B(x_0,r_0)$ with $x_0\in\pom$ and $0<r_0<\diam(\pom)$. From Lemma \ref{lemma:CDC-inherit}, we know that every $T_\Delta$ is a 1-sided NTA domain satisfying the CDC and moreover all the implicit constants depend on the corresponding ones for $\Omega$. Write $c_0^\star$ for the  associated Corkscrew constant (which is independent of $\Delta$), set $K=\max\{1, c_0^\star/c_0\}$ and fix $M>16 K\ge 16$. We have two sub-cases:

\medskip

\noindent\textbf{Case 2a:}  $0<r_0<\diam(\pom)/(2M)$. 

Set $\widehat{B}_0=M B_0$, so that $r_{\widehat{B}_0}<\diam(\pom)/2$, and let $\widehat{\Delta}_0 = \widehat{B}_0\cap\pom$. Define $\Omega_\star= T_{\widehat{\Delta}_0}\subset\Omega$, and our goal is to apply Proposition \ref{PROP:LOCAL-VERSION} in this bounded domain. From Lemma \ref{lemma:CDC-inherit}, it follows that $\Omega_\star$ is  a 1-sided NTA domain satisfying the CDC and moreover all the implicit constants depend on the corresponding ones for $\Omega$ but are uniform on $M$. In particular, the interior Corkscrew condition holds with $c_0^\star$ (which does not depend on $M$). 

Write $\widetilde{B}_0=B(x_0, \widetilde{r}_0)=B(x_0, Kr_0)$ so that $8B_0\subset 8\widetilde{B}_0\subset \widehat{B}_0$, and set $\widetilde{\Delta}_0=\widetilde{B}_0\cap\pom$, $\widetilde{\Delta}_0^\star=\widetilde{B}_0\cap\pom_\star$, and $\Delta_0^\star:=B_0\cap\pom_\star$. Note that by \eqref{definicionkappa0} we have $8\widetilde{B}_0\cap\Omega\subset \widehat{B}_0\cap\Omega\subset T_{\widehat{\Delta}_0}=\Omega_\star$ and hence $8\widetilde{\Delta}_0=8\widetilde{\Delta}_0^\star$.
Moreover, one can also see that for every $X\in 4\widetilde{B}_0\cap \Omega= 4\widetilde{B}_0\cap \Omega_\star$ then $\delta(X)=\dist(X,\pom_\star)=:\delta_\star(X)$. Consequently, if $X_{\Delta_0^\star}$ denotes the Corkscrew point relative to $\Delta_0^\star$ for the domain $\Omega_\star$ and $X_{\widetilde{\Delta}_0}$ denotes the Corkscrew point relative to $\widetilde{\Delta}_0$ for the domain $\Omega$ we have
\[
c_0^\star r_0\le \delta_\star(X_{\Delta_0^\star})=\delta(X_{\Delta_0^\star})\le r_0,
\quad
c_0 r_0\le \delta(X_{\widetilde{\Delta}_0})=\delta_\star(X_{\widetilde{\Delta}_0})\le r_0,
\]
and $|X_{\Delta_0^\star}-X_{\widetilde{\Delta}_0}|\le (1+K)r_0$.

Fix $x\in 2\Delta_0$, $0<r<r_0 c_0^\star/4$, write $B=B(x,r)$, $\Delta=B\cap\pom$, $\Delta^\star=B\cap\pom_\star$, and note that from the above observations $\Delta=\Delta^\star$. 
Invoking Lemma \ref{lemma:proppde} part $(e)$, the Harnack chain condition for $\Omega_\star$ allows us to obtain
\[
\omega_{L_0,\Omega_\star}^{X_{\Delta_0^\star}} (\Delta^\star)
\approx
\omega_{L_0,\Omega_\star}^{X_{\widetilde{\Delta}_0}}(\Delta)
\approx
\omega_{L_0,\Omega}^{X_{\widetilde{\Delta}_0}}(\Delta).
\]
On the other hand if $Y\in B\cap\Omega_\star=B\cap\Omega$ and we pick $y\in\pom$ so that $|Y-y|=\delta(Y)=\delta_\star(Y)<r_0$.
Write $B_Y=B(y,2\delta(Y))$ which satisfies $B_Y\subset 5 B_0$ and hence $\Delta_Y:=B_Y\cap\pom=B_Y\cap\pom_\star=:\Delta_Y^\star$. Then
if $X_{\Delta_Y}$ (respectively $X_{\Delta_Y^\star})$ stands for the Corkscrew point relative to $\Delta_Y$ (respectively $\Delta_Y^\star$) with respect to $\Omega$ (respectively $\Omega_\star$) we observe that
\begin{multline*}
G_{L_0,\Omega_\star}(X_{\Delta_0^\star},Y)
\approx
G_{L_0,\Omega_\star}(X_{\Delta_0^\star}, X_{\Delta_Y^\star})
\approx
\delta(Y)^{1-n}\omega_{L_0,\Omega_\star}^{X_{\Delta_0^\star}}(\Delta_Y^\star)
\approx
\delta(Y)^{1-n}\omega_{L_0,\Omega_\star}^{^{X_{\widetilde{\Delta}_0}}}(\Delta_Y)
\\
\approx
\delta(Y)^{1-n}\omega_{L_0,\Omega}^{X_{\widetilde{\Delta}_0}}(\Delta_Y)
\approx
G_{L_0,\Omega}(X_{\widetilde{\Delta}_0},X_{\Delta_Y})
\approx
G_{L_0,\Omega}(X_{\widetilde{\Delta}_0},Y),
\end{multline*}
where we have used the Harnack chain condition in both $\Omega$ and $\Omega_\star$, Harnack's inequality, and Lemma \ref{lemma:proppde} parts $(b)$ and $(e)$. Finally, 
\[
\varrho_\star(A, A_0)(Y)
:=
\|A-A_0\|_{B(Y,\delta_\star(Y)/2)}
=
\|A-A_0\|_{B(Y,\delta(Y)/2)}
=
\varrho(A, A_0)(Y)
\]
since  $Y\in B\cap\Omega\subset 4\widetilde{B}_0\cap\Omega=4\widetilde{B}_0\cap\Omega_\star$ and hence $\delta(Y)=\delta_\star(Y)$.

At this point we collect the previous estimates to obtain that
\begin{align*}
&\vertiii{\varrho(A,A_0)}_{B_0,\Omega_\star}
\\
&\qquad:=
\sup_{
	\substack{B=B(x,r)\\ x\in\Delta_0^\star, 0<r<r_0 c_0^\star/4}}
\frac{1}{\omega_{L_0,\Omega_\star}^{X_{\Delta_0^\star}} (\Delta^\star)}
\iint_{B\cap\Omega_\star}\varrho_\star(A,A_0)(X)^2\frac{G_{L_0,\Omega_\star}(X_{\Delta_0^\star},X)}{\delta_\star(X)^2}\,dX
\\
&\qquad\,\lesssim
\sup_{
	\substack{B=B(x,r)\\ x\in\widetilde{\Delta}_0, 0<r<\widetilde{r}_0 c_0/4}}
\frac{1}{\omega_{L_0,\Omega}^{X_{\widetilde{\Delta}_0}}(\Delta)}
\iint_{B\cap\Omega}\varrho(A,A_0)(X)^2\frac{G_{L_0,\Omega}(X_{\widetilde{\Delta}_0},Y)}{\delta(X)^2}\,dX
\\
&\qquad\,\le 
\vertiii{\varrho(A,A_0)}<\infty,
\end{align*}
where all the implicit constants are independent of $M$ and uniform in $B_0$. We can then invoke Proposition \ref{PROP:LOCAL-VERSION} part $(a)$ (since $\Omega_\star$ is bounded) to find $q$, $1<q<\infty$, such that $\omega_{L,\Omega_\star}\in RH_q(\Delta_0, \omega_{L_0,\Omega_\star})$. On the other hand, by Lemma \ref{lemma:proppde} part $(e)$ we have that $\omega_{L,\Omega_\star}$ and $\omega_{L,\Omega}$ are comparable in $\Delta_0$ and so are $\omega_{L_0,\Omega_\star}$ and $\omega_{L_0,\Omega}$. Thus eventually, $\omega_{L,\Omega}\in RH_q(\Delta_0, \omega_{L_0,\Omega})$.
Moreover, the previous estimate is independent of $B_0$ and the same $q$ is valid for every $B_0$ as in the present case. 

\medskip

\noindent\textbf{Case 2b:}  $\diam(\pom)/(2M)<r_0<\diam(\pom)$. 
 
Note first that this case is vacuous if $\pom$ is unbounded. Hence we may assume that $\pom$ is bounded.  We first find a finite maximal collection of points  $\{x_j\}_{j=1}^J\in \Delta_0$ with $1\le J\le (1+20M)^{n+1}$ such that $|x_j-x_k|\ge \diam(\pom)/(10 M)$ for $1\le j<k\le J$. For any of the balls $B_j=B(x_j,\diam(\pom)/(10M))$ by \textbf{Case 2a} we have that $\omega_{L}\in RH_q(3\Delta_j, \omega_{L_0})$ where the implicit constants do not depend on $j$, and we have written $\omega_{L_0}=\omega_{L_0,\Omega}$ and $\omega_{L}=\omega_{L,\Omega}$.

To show that $\omega_{L}\in RH_q(\Delta_0, \omega_{L_0})$, let $B=B(x,r)\subset B_0$ with $x\in\pom$ and $\Delta=B\cap\pom$. If $\Delta\cap \Delta_j\neq\emptyset$ and $0<r<\diam(\pom)/(10 M)$ we note that  $\Delta\cap \Delta_j\subset \Delta\subset 3\Delta_j$  and thus
\begin{multline*}
\left(\frac1{\omega_{L_0}^{X_{\Delta_0}}(\Delta)}\int_{\Delta\cap \Delta_j}h(y;L,L_0,X_{\Delta_0} )^q d \omega_{L_0}^{X_{\Delta_0}}(y)\right)^{\frac1q} 
\\
\lesssim
\left(\aver{\Delta}h(y;L,L_0,X_{3\Delta_j} )^q d \omega_{L_0}^{X_{3\Delta_j}}(y)\right)^{\frac1q} 
\lesssim
\frac{\omega_L^{X_{3\Delta_j}}(\Delta)}{\omega_{L_0}^{X_{3\Delta_j}}(\Delta)}
\approx
\frac{\omega_L^{X_{\Delta_0}}(\Delta)}{\omega_{L_0}^{X_{\Delta_0}}(\Delta)},
\end{multline*}
where we have used Harnack's inequality and that $\omega_{L}\in RH_q(3\Delta_j, \omega_{L_0})$.
On the other hand, if  $\Delta\cap \Delta_j\neq\emptyset$ and $\diam(\pom)/(10 M)<r<r_0$ we have that $r\approx r_0\approx\diam(\pom)$. Thus, by Lemma \ref{lemma:proppde} parts $(a)$, $(b)$, and $(c)$,
$\omega_{L_0}^{X_{\Delta_0}}(\Delta)\approx \omega_{L_0}^{X_{\Delta_j}}(\Delta_j)\approx 1$ and the same occurs for $\omega_L$. These yield
\begin{multline*}
\left(\frac1{\omega_{L_0}^{X_{\Delta_0}}(\Delta)}\int_{\Delta\cap \Delta_j}h(y;L,L_0,X_{\Delta_0} )^q d \omega_{L_0}^{X_{\Delta_0}}(y)\right)^{\frac1q} 
\\
\lesssim
\left(\aver{\Delta_j}h(y;L,L_0,X_{\Delta_j} )^q d \omega_{L_0}^{X_{\Delta_j}}(y)\right)^{\frac1q} 
\lesssim
\frac{\omega_L^{X_{\Delta_j}}(\Delta_j)}{\omega_{L_0}^{X_{\Delta_j}}(\Delta_j)}
\approx 
1
\approx
\frac{\omega_L^{X_{\Delta_0}}(\Delta)}{\omega_{L_0}^{X_{\Delta_0}}(\Delta)},
\end{multline*}
where we have used Harnack's inequality and the fact that $\omega_{L}\in RH_q(3\Delta_j, \omega_{L_0})$. 
All these, the fact $\Delta\subset \bigcup_j\Delta_j\cap\Delta$, and the bound $J\le (1+2M)^{n+1}$ imply
\begin{multline*}
\left(\aver{\Delta}h(y;L,L_0,X_{\Delta_0} )^q d \omega_{L_0}^{X_{\Delta_0}}(y)\right)^{\frac1q} 
\\
\le
\left(\sum_{j=1}^J\frac1{\omega_{L_0}^{X_{\Delta_0}}(\Delta)}\int_{\Delta\cap \Delta_j}h(y;L,L_0,X_{\Delta_0} )^q d \omega_{L_0}^{X_{\Delta_0}}(y)\right)^{\frac1q} 
\lesssim
\frac{\omega_L^{X_{\Delta_0}}(\Delta)}{\omega_{L_0}^{X_{\Delta_0}}(\Delta)},
\end{multline*}
which eventually shows $\omega_{L,\Omega}\in RH_q(\Delta_0, \omega_{L_0,\Omega})$ in the current case.

\medskip

Collecting \textbf{Case 2a} and \textbf{Case 2b} we have shown that $\omega_{L,\Omega}\in RH_q(\Delta_0, \omega_{L_0,\Omega})$ uniformly on $\Delta_0$ which eventually means that $\omega_{L,\Omega}\in RH_q(\pom, \omega_{L_0,\Omega})$ and hence $\omega_{L,\Omega}\in A_\infty(\pom, \omega_{L_0,\Omega})$. This completes the proof.

\end{proof}

\begin{proof}[Proof of Theorem \ref{thm:main}, part $(b)$]
We follow the same argument as in the previous proof using part $(b)$ in place of part $(a)$ in Proposition \ref{PROP:LOCAL-VERSION}. Further details are left to the interested reader. 
\end{proof}

\begin{proof}[Proof of Theorem \ref{thm:main-SF}]
Fix $\alpha>0$. It is immediate to see that parts $(a)$ and $(b)$ follow respectively from parts $(a)$ and $(b)$ in Theorem \ref{thm:main} and the following estimate:
\begin{equation}\label{Car:SF}
\vertiii{\varrho(A,A_0)}\lesssim_\alpha \|\mathcal{A}_\alpha(\varrho(A,A_0))\|_{L^\infty(\omega_{L_0})}^2,
\end{equation}
where, as explained in Remark \ref{remark:ambiguity}, the pole for $\omega_{L_0}$ needs not to be specified. Hence everything reduces to obtaining such estimate. With this goal in mind, fix $\Delta_0=B_0\cap\pom$ with $B_0=B(x_0,r_0)$, $x_0\in\pom$, and $0<r_0<\diam(\pom)$. Let $\Delta=B\cap\pom$ with $B=B(x,r)$, $x\in2\Delta_0$, and $0<r<r_0 c_0/4$, here $c_0$ is the Corkscrew constant. Write $X_0=X_{\Delta_0}$ and $\omega_0=\omega_{L_0}^{X_{0}}$. Note that this choice guarantees that $X_0\notin 4B$. Define
\[
\W_B=\{I\in\W: I\cap B\neq\emptyset\}
\]
and for every $I\in\W_B$ let $X_I\in I\cap B$ so that $4\diam(I)\le \dist(I,\pom)\le\delta(X_I)<r$ and hence $I\subset \frac54B$. 
Pick $x_I\in\pom$ such that $|X_I-x_I|=\delta(X_I)\le\diam(I)+\dist(I,\pom)$ and let $Q_I\in\dd$ be such that $x_I\in Q_I$ and $\ell(I)=\ell(Q_I)$. By Lemma \ref{lemma:proppde} parts $(a)$--$(c)$ and Harnack's inequality one can show that 
\[
\omega_0(Q_I)\approx \ell(I)^{n-1} G_{L_0}(X_0,X_I)\approx \delta(Y)^{n-1} G_{L_0}(X_0,Y), \quad\text{ for every $Y\in I$}.
\]
Then,
\begin{multline*}
\mathcal{I}_B
:=
\iint_{B\cap\Omega}\varrho(A,A_0)(Y)^2\frac{G_{L_0}(X_{0},Y)}{\delta(Y)^2}\,dY
\lesssim
\sum_{I\in\W_B} \iint_{B\cap I}\frac{\varrho(A,A_0)(Y)^2}{\delta(Y)^{n+1}}\,dY\,\omega_0(Q_I)
\\
=
\iint_{B\cap\Omega} \frac{\varrho(A,A_0)(Y)^2}{\delta(Y)^{n+1}}\sum_{I\in\W_B} \mathbf{1}_I (Y)\,\omega_0(Q_I)\,dY.
\end{multline*}
Fix $Y\in B$ and note that by the nature of the Whitney cubes one has $\#\{I\in\W_B:I\ni Y\}\le C_n$ for some dimensional constant (indeed the $I$'s have non-overlapping interiors and hence for a.e. $Y\in\Omega$, there is just one $I_Y$ containing $Y$). Pick $y\in\pom$ such that $|Y-y|=\delta(Y)$. Let $z\in Q_I$, then by \eqref{deltaQ} and \eqref{constwhitney}
\begin{multline*}
|z-y|
\le 
|z-x_I|+|x_I-X_I|+|X_I-Y|+|Y-y|
\\
\le
\Xi\ell(Q_I)+\delta(X_I)+\diam(I)+\delta(Y)
<3\Xi\delta(Y)
\end{multline*}
and therefore  $Q_I\subset \Delta(y,3\Xi\delta(Y))$. Note also that 
\[
\Delta(y,\alpha\delta(Y))\subset B(Y,(1+\alpha)\delta(Y))\cap\pom\subset (2+\alpha)\Delta.
\] 
Then using Lemma \ref{lemma:proppde} parts $(a)$ and $(c)$
\begin{multline*}
\sum_{I\in\W_B} \mathbf{1}_I (Y)\,\omega_0(Q_I)
\le
C_n \omega_0\big(\Delta(y,3\Xi\delta(Y))\big)
\\
\lesssim_\alpha \omega_0\big(\Delta(y,\alpha\delta(Y))\big)\le
\omega_0\big(B(Y,(1+\alpha)\delta(Y))\cap\pom\big).
\end{multline*}
Hence, using again Lemma \ref{lemma:proppde} parts $(a)$ and $(c)$, and Harnack's inequality we conclude:
\begin{align*}
\mathcal{I}_B
&\lesssim_\alpha
\iint_{B\cap\Omega} \frac{\varrho(A,A_0)(Y)^2}{\delta(Y)^{n+1}}\omega_0\big(B(Y,(1+\alpha)\delta(Y))\cap\pom\big)\,dY
\\
&=
\int_{(2+\alpha)\Delta} \iint_{B\cap\Omega} \frac{\varrho(A,A_0)(Y)^2}{\delta(Y)^{n+1}} \mathbf{1}_{B(Y,(1+\alpha)\delta(Y))\cap\pom}(z)\,dY\,d \omega_0(z)
\\
&\le
\int_{(2+\alpha)\Delta} \iint_{\Gamma_\alpha(z)} \frac{\varrho(A,A_0)(Y)^2}{\delta(Y)^{n+1}}\,d \omega_0(z)
\\
&=
\int_{(2+\alpha)\Delta} \mathcal{A}_\alpha(\varrho(A,A_0))(z)^2\,d \omega_0(z)
\\
&\lesssim
\|\mathcal{A}_\alpha(\varrho(A,A_0))\|_{L^\infty(\omega_0)}^2\omega_0((2+\alpha)\Delta)
\\
&\lesssim_\alpha 
\|\mathcal{A}_\alpha(\varrho(A,A_0))\|_{L^\infty(\omega_0)}^2\omega_0(\Delta).
\end{align*}
This eventually shows \eqref{Car:SF} and this completes the proof of Theorem \ref{thm:main-SF}.
\end{proof}

\subsection{Auxiliary results}\label{subsection:aux}

We next state some auxiliary lemmas which will be needed for our arguments. 

\begin{lemma}
	\label{lemma:lema272'}
	Let $\Omega$ be a \textbf{bounded} 1-sided NTA domain satisfying the CDC. Consider $L_0=-\div(A_0\nabla)$ and  $L=-\div(A\nabla)$ two real (non-necessarily symmetric) elliptic operators, and let $u_0\in W^{1,2}(\Omega)$ be a weak solution to $L_0 u_0=0$ in $\Omega$. Then,
	\begin{equation}
	\iint_{\Omega} A_0^\top(Y)\nabla_YG_{L^\top}(Y,X)\cdot \nabla u_0(Y) \,dY=0,
	\qquad \mbox{for a.e.~}X\in\Omega.
	\end{equation} 
\end{lemma}

\begin{proof}
	We follow the argument in \cite[Lemma 3.12]{CHM} where it was assumed that $\pom$ is AR and the operators were symmetric. 
	Pick $\varphi\in C_{0}^{\infty}(\mathbb{R})$ with $\mathbf{1}_{(0,1)}\leq \varphi \leq \mathbf{1}_{(0,2)}$. Fix $X_0\in\Omega$, for each $0<\varepsilon<\delta(X_0)/16$ we set $\varphi_\varepsilon(X)=\varphi(|X-X_0|/\varepsilon)$ and $\psi_\varepsilon=1-\varphi_\varepsilon$. By Remark~\ref{rem:GF}, one has that  $G_{L^\top}(\cdot,X_0)\psi_\varepsilon\in W_0^{1,2}(\Omega)$, which together with the assumption that $u_0\in W^{1,2}(\Omega)$ is a weak solution to $L_0u_0=0$ in $\Omega$, allows us to see that
	\[
	\iint_{\Omega}A_0^\top(Y)\nabla \big(G_{L^\top}(\cdot,X_0)\psi_\varepsilon\big)(Y)\cdot\nabla u_0(Y)\,dY
	=
	0.
	\]
	As a consequence,
	\begin{multline*}
	\iint_{\Omega}A_0^\top\nabla G_{L^\top}(\cdot,X_0)\cdot\nabla u_0\,dY
	=
	\iint_{\Omega}A_0^\top\nabla\big(G_{L^\top}(\cdot,X_0)\varphi_\varepsilon\big)\cdot\nabla u_0\,dY
	\\
	=\iint_{\Omega}A_0^\top\nabla G_{L^\top}(\cdot,X_0)\cdot\nabla u_0\,\varphi_\varepsilon\,dY
	+\iint_{\Omega}A_0^\top\nabla\varphi_\varepsilon\cdot\nabla u_0\,G_{L^\top}(\cdot,X_0)\,dY
	=:\mathcal{I}_\varepsilon+\mathcal{II}_\varepsilon.
	\end{multline*}
	We use \eqref{uniformlyelliptic}, Cauchy-Schwarz's inequality, Caccioppoli's inequality for $G_{L^\top}(\cdot,X_0)$ (which satisfies $L^\top G_{L^\top}(\cdot,X_0)=0$ in the weak sense in $\Omega\setminus\{X_0\}$), and \eqref{sizestimate}
	\begin{align*}
	&|\mathcal{I}_\varepsilon|
	\lesssim
	\iint_{B(X_0,2\varepsilon)}|\nabla G_{L^\top}(\cdot,X_0)|\,|\nabla u_0|\,dY
	\\
	&\ \lesssim
	\sum_{j=0}^{\infty}\bigg(\iint_{2^{-j}\varepsilon\le |Y-X_0|<2^{-j+1}\varepsilon}|\nabla_Y G_{L^\top}(Y,X_0)|^2\,dY\bigg)^{\frac12}
	\bigg(\iint_{B(X_0,2^{-j+1}\varepsilon)}|\nabla u_0|^2\,dY\bigg)^{\frac12}
	\\
	&\ \lesssim
	M_2(|\nabla u_0|\mathbf{1}_{\Omega})(X_0)
	\sum_{j=1}^{\infty}\big(2^{-j}\varepsilon\big)^{\frac{n-1}{2}}
	\bigg(\iint_{2^{-j-1}\varepsilon\le |Y-X_0|<2^{-j+2}\varepsilon}|G_{L^\top}(Y,X_0)|^2\,dY\bigg)^{\frac12}
	\\
	&\ 
	\lesssim\varepsilon M_2(|\nabla u_0|\mathbf{1}_{\Omega})(X_0),
	\end{align*}
	where  $M_2f:=M(|f|^2)^{\frac12}$, with $M$ being the Hardy-Littlewood maximal operator on $\re^{n+1}$. For the second term, we invoke again \eqref{sizestimate} and Jensen's inequality:
	\begin{multline*}
	|\mathcal{II}_\varepsilon|\lesssim\varepsilon^{-1}\iint_{\varepsilon\le |Y-X_0|<2\varepsilon}|G_{L^\top}(Y,X_0)|\,|\nabla u_0(Y)|\,dY
	\\
	\lesssim\varepsilon^{-n}\iint_{B(X_0,2\varepsilon)}|\nabla u_0(Y)|\,dY\lesssim \varepsilon M_2(|\nabla u_0|\mathbf{1}_{\Omega})(X_0).
	\end{multline*}
	Combining the obtained estimates we have shown that, for every $X_0\in\Omega$ and for every $0<\varepsilon<\delta(X_0)/16$,
	\begin{equation}\label{cotaM2}
	\bigg|\iint_{\Omega}A_0^\top\nabla G_{L^\top}(\cdot,X_0)\cdot\nabla u_0\,dY\bigg|\lesssim\varepsilon M_2(|\nabla u_0|\mathbf{1}_{\Omega})(X_0).
	\end{equation}
	Since  $u_0\in W^{1,2}(\Omega)$ it the follows that $M_2(|\nabla u_0|\mathbf{1}_{\Omega})\in L^{2,\infty}(\ree)$, and as a result $M_2(|\nabla u_0|\mathbf{1}_{\Omega})$ is finite almost everywhere in $\ree$. Thus, we can let $\varepsilon\rightarrow 0^+$ in \eqref{cotaM2} to obtain the desired equality.
\end{proof}

\begin{lemma}
	\label{lemma:u0-u:bounded}
	Let $\Omega$ be a \textbf{bounded} 1-sided NTA domain satisfying the CDC, and let $L_0=-\div(A_0\nabla)$ and $L=-\div(A\nabla)$ be two real (non-necessarily symmetric) elliptic operators. Given $g\in \Lip(\partial\Omega)$, consider the solutions $u_0$ and $u$ given by
	\[
	u_0(X)=\int_{\partial\Omega}g(y)\,d\omega_{L_0}^X(y),\qquad u(X)=\int_{\partial\Omega}g(y)\,d\omega_{L}^X(y),
	\qquad X\in\Omega.
	\]
	Then,
	\begin{equation}
	\label{u_0-u:bounded}
	u(X)-u_0(X)=\iint_{\Omega}(A_0-A)^\top(Y)\nabla_Y G_{L^\top}(Y,X)\cdot \nabla u_0(Y) \,dY
	\end{equation}
	for almost every $X\in\Omega$. 
\end{lemma}

\begin{proof}
	We again follow the argument in \cite[Lemma 3.18]{CHM}  with some appropriate changes. Following \cite{HMT1} we know that $u_0=\widetilde{g} - v_0$ and $u=\widetilde{g}-v$, where $\widetilde{g}\in \Lip_c(\ree)$ is a Lipschitz extension of $g$, and $v_0,v\in W_0^{1,2}(\Omega)$ are the Lax-Milgram solutions of $L_0v_0=L_0\widetilde{g}$ and $Lv=L\widetilde{g}$ respectively. Hence, we have that $u-u_0=v_0-v\in W_0^{1,2}(\Omega)$, and following again \cite{HMT1} one can extend \eqref{eq:G-delta} so that 
	\[
	(u-u_0)(X)=\iint_{\Omega}A^\top(Y)\nabla_YG_{L^\top}(Y,X)\cdot\nabla (u-u_0)(Y)\,dY,\quad\text{for a.e. }X\in\Omega.
	\]
	For almost every $X\in\Omega$ we then have that
	\begin{multline*}
	(u-u_0)(X)-\iint_{\Omega}(A_0-A)^\top(Y)\nabla_Y G_{L^\top}(Y,X)\cdot\nabla u_0(Y)\,dY
	\\
	=\iint_{\Omega}A^\top(Y)\nabla_YG_{L^\top}(Y,X)\cdot\nabla u(Y)\,dY
	-\iint_{\Omega}A_0^\top(Y)\nabla_Y G_{L^\top}(Y,X)\cdot\nabla u_0(Y)\,dY.
	\end{multline*}
	Using Lemma \ref{lemma:lema272'} for both terms the right side of the above equality vanishes almost everywhere, and this proves \eqref{u_0-u:bounded}.
\end{proof}

For the following result, we recall the definition of the localized  dyadic conical square function in \eqref{def:SF}. Also, if  $\mu$ is a non-negative Borel measure on $Q_0$ so that $0<\mu(Q)<\infty$ for every $Q\in\mathbb{D}_{Q_0}$, we define the localized dyadic maximal function with respect to $\mu$ as
\[
M^{\mathbf{d}}_{Q_0,\mu} \nu (x)
:= 
\sup_{x\in Q\in\mathbb{D}_{Q_0}} \frac{\nu(Q)}{\mu(Q)}, 
\]
where $\nu$ is a non-negative Borel measure on $Q_0$.

\begin{lemma}\label{prop:step1-estimate}
		Let $\Omega$ be a 1-sided NTA domain satisfying the CDC and let $L_0=-\div(A_0\nabla)$ and $L=-\div(A\nabla)$ be two real (non-necessarily symmetric) elliptic operators. 	Let $Q_0\in\mathbb{D}$ and let $\F=\{Q_j\}_j\subset\dd_{Q_0}$ be a \textup{(}possibly empty\textup{)} family of pairwise disjoint dyadic cubes. 	 Let $u_0\in W^{1,2}_{\rm loc}(\Omega)$, and let $0\le H\in L^\infty(\Omega)$. Let $Y_0\in \Omega\setminus B_{Q_0}^*$ \textup{(}see \eqref{definicionkappa12}\textup{)} and define $\gamma_{Y_0}=\{\gamma_{Y_0,Q}\}_{Q\in\dd_{Q_0}}$ where
	\[
	\gamma_{Y_0,Q}:=\omega_{L_0}^{Y_0}(Q)\sum_{I\in\mathcal{W}_{Q}^{*}} \|H\|_{L^\infty(I^*)}^2,
	\qquad Q\in\mathbb{D}_{Q_0}.
	\]
	Then,
	\begin{multline}\label{prop:main-estimate}	
	\iint_{\Omega_{\F,Q_0}} H(Y)|\nabla_Y G_{L^\top}(Y,Y_0)| \, |\nabla u_0(Y)| dY
	\\
	\lesssim
	\|\mut_{\gamma_{Y_0},\F}\|_{\C(Q_0,\omega_{L_0}^{Y_0})}^\frac12
	\int_{Q_0} M^{\mathbf{d}}_{Q_0,\omega_{L_0}^{Y_0}}(\omega_{L}^{Y_0})(x) \mathcal{S}_{Q_0}u_0(x) d\omega_{L_0}^{Y_0}(x).
	\end{multline}
\end{lemma}

\begin{proof}
	
	To ease the notation let us write $\omega_{0}:=\omega_{L_0}^{Y_0}$, $\omega:=\omega_{L}^{Y_0}$, $\gamma_{Y_0,Q}=\gamma_Q$, and $\gamma_{Y_0}=\gamma$. From the definition of $\Omega_{\F,Q_0}$; Cauchy-Schwarz's, Caccioppoli's and Harnack's inequalities (applied to $G_{L^\top}(\cdot, Y_0)$ which satisfies $L^\top G_{L^\top}(\cdot, Y_0)=0$ in the weak sense in $\Omega\setminus\{Y_0\}$); the fact that $\ell(I)\approx\ell(Q)\approx \delta(Y)$ for every $Y\in I^*\in\mathcal{W}^{*}_{Q}$; \eqref{G-G-top}; and  Lemma \ref{lemma:proppde} part $(b)$ in conjunction with \eqref{definicionkappa12}, we clearly have
	\begin{align*}
	&\mathcal{I}_0
	:=
	\iint_{\Omega_{\F,Q_0}} H(Y)|\nabla_Y G_{L^\top}(Y,Y_0)| \, |\nabla u_0(Y)| dY
	\\
	& \le
	\sum_{Q\in \mathbb{D}_{\mathcal{F},Q_0}} 
	\sum_{I\in\mathcal{W}^{*}_{Q}} 
	\|H\|_{L^\infty(I^*)}
	\iint_{I^*} |\nabla_Y G_{L^\top}(Y,Y_0)| \, |\nabla u_0(Y)| dY
	\\
	&
	\leq 
	\sum_{Q\in \mathbb{D}_{\mathcal{F},Q_0}} 
	\sum_{I\in\mathcal{W}^{*}_{Q}} 
	\|H\|_{L^\infty(I^*)}
	\left(\iint_{I^*} \, |\nabla_Y G_{L^\top}(Y,Y_0)|^{2} dY\right)^{\frac12} 
	\left(\iint_{I^*} |\nabla u_0(Y)|^{2} dY\right)^{\frac12}
	\\
	&
	\lesssim  
	\sum_{Q\in \mathbb{D}_{\mathcal{F},Q_0}} 
	\sum_{I\in\mathcal{W}^{*}_{Q}} 
	\|H\|_{L^\infty(I^*)} \ell(I)^n \, \frac{G_{L^\top}(X_Q,Y_0)}{\delta(X_{Q})} 
	\left(\iint_{I^*} |\nabla u_0(Y)|^2 \delta(Y)^{1-n}dY\right)^{\frac12} 
	\\
	&
	\lesssim 
	\sum_{Q\in \mathbb{D}_{\mathcal{F},Q_0}} 
	\sum_{I\in\mathcal{W}^{*}_{Q}} 
	\|H\|_{L^\infty(I^*)}
	\omega(Q) 
	\left(\iint_{I^*} |\nabla u_0(Y)|^2 \delta(Y)^{1-n}dY\right)^{\frac12} 
	\\
	&
	\le 
	\sum_{Q\in \mathbb{D}_{Q_0}}  
	\left(\omega_{0}(Q) \left( \frac{\omega(Q)}{\omega_0(Q)}\right)^{2} 
	\iint_{U_Q}|\nabla u_0(Y)|^2 \delta(Y)^{1-n} dY\right)^{\frac12} 
	\gamma_{\F,Q}^{\frac12},
	\end{align*}
	where in the last estimate we have used that the family $\{I^*\}_{I\in \W_Q^*}$ has bounded overlap.
	If we now set $\alpha=\{\alpha_Q\}_{Q\in\dd_{Q_0}}$ with
	\[
	\alpha_{Q}
	:=
	\left(\omega_{0}(Q) \left( \frac{\omega(Q)}{\omega_0(Q)}\right)^{2}  \iint_{U_Q}|\nabla u_0(Y)|^2 \delta(Y)^{1-n} dY\right)^{\frac12},
	\qquad Q\in\dd_{Q_0},
	\]
	we obtain by invoking Lemma \ref{lemma:tentspaces} with $\mu=\omega_{0}$
	\[
	\mathcal{I}_0
	\lesssim
	\sum_{Q\in \mathbb{D}_{Q_0}} \alpha_Q \gamma_{\F,Q}^{\frac12}
	\le
	4\,\int_{Q_0}\mathcal{A}_{Q_0}^{\omega_0} \alpha(x)\mathcal{B}_{Q_0}^{\omega_0} \big(\{\gamma_{\F,Q}^{\frac12}\}_{Q\in\dd_{Q_0}}\big)(x)\,d\omega_0(x).
	\]
	Note that for every $x\in Q_0$
	\begin{align*}
	\mathcal{A}_{Q_0}^{\omega_0} \alpha(x)
	=
	\left(
	\sum_{x\in Q\in\mathbb{D}_{Q_0}} \left(\frac{\omega(Q)}{\omega_0(Q)}\right)^{2} \iint_{U_Q}|\nabla u_0|^2 \delta(\cdot)^{1-n} dY
	\right)^\frac12
	\lesssim
	M^{\mathbf{d}}_{Q_0,\omega_0} \omega(x)\,\mathcal{S}_{Q_0}u_0(x),
	\end{align*}
	where we have used that the family $\{U_Q\}_{Q\in\dd_{Q_0}}$ has finite overlap. Besides, if $x\in Q_0$
	\[
	\mathcal{B}_{Q_0}^{\omega_0}\big(\{\gamma_{\F,Q}^{\frac12}\}_{Q\in\dd_{Q_0}}\big)(x)
	=
	\sup_{x\in Q\in\mathbb{D}_{Q_0}}\bigg(\frac{1}{\omega_0(Q)}\sum_{Q'\in\mathbb{D}_{Q}}\gamma_{\F,Q'}\bigg)^{\frac12}
	\le
	\|\mut_{\gamma,\F}\|_{\C(Q_0,\omega_0)}^\frac12.
	\]
	Collecting all the obtained estimates completes the proof of \eqref{prop:main-estimate}. 
\end{proof}


Throughout the rest of this section we will always assume that $\Omega$ is a 1-sided NTA domain satisfying the CDC, hence $\pom$ is also bounded. We fix $\dd=\dd(\pom)$ the dyadic grid for Lemma \ref{lemma:dyadiccubes} with $E=\pom$. Let $Lu=-\div(A\nabla u)$ and $L_0u=-\div(A_0\nabla u)$ be two real (non-necessarily symmetric) elliptic operators. Fix $x_0\in\pom$ and $0<r_0<\diam(\pom)$ and let $B_0=B(x_0,r_0)$, $\Delta_0=B_0\cap\pom$. From now on $X_0:=X_{\Delta_0}$, $\omega_0:=\omega_{L_0}^{X_0}$ and $\omega:=\omega_{L}^{X_0}$.

We further assume that $0<r_0<\diam(\pom)/2$. In particular $r_{2\Delta_0}<\diam(\pom)$. We introduce the following notation (which should not be confused with the one introduced in \eqref{D-delta}): 
\begin{equation}\label{D-delta-new}
\dd^{\Delta_0}_*
=
\Big\{
Q\in\dd:\ Q\cap \tfrac32\Delta_0\neq\emptyset, \ \tfrac{c_0}{16\kappa_0}r_0\le \ell(Q)<\tfrac{c_0}{8\kappa_0}r_0
\Big\}.
\end{equation}

Fixed $\varphi\in C^{\infty}(0,\infty)$ with $\mathbf{1}_{(0,1)}\leq \varphi \leq \mathbf{1}_{(0,2)}$, we define
\begin{align}
\label{defgt}
P_t g(x):=\int_{\partial\Omega}\varphi_t(x,y)g(y)\,d\omega_0(y)\qquad \mbox{whenever} \, \, x\in \partial\Omega,
\end{align}
where
\begin{align}\label{def:phit}
\varphi_t(x,y):=\frac{\varphi\left(\frac{|x-y|}{t}\right)}{\int_{\partial\Omega}\varphi\left(\frac{|x-z|}{t}\right)\,d\omega_0(z)}\qquad \mbox{whenever}\, \, x,y\in \partial\Omega.
\end{align}
A variant of the following lemma was shown in \cite[Lemma 3.5]{CHM}.

\begin{lemma}\label{lemma:Pt}
	Let $\Omega\subset\mathbb{R}^{n+1}$ be a 1-sided NTA domain satisfying the CDC. Let $L_0u=-\div(A_0\nabla u)$ be a 
	real (non-necessarily symmetric) elliptic operator. Fix $\varphi\in C^{\infty}(0,\infty)$ with $\mathbf{1}_{(0,1)}\leq \varphi \leq \mathbf{1}_{(0,2)}$. There exists $C$ depending only on dimension $n$, the 1-sided NTA constants, the CDC constant, the ellipticity constant of $L_0$, and $\varphi$ (and independent of $\Delta_0$), such that for every $Q\in\dd_{Q^0}$ with $Q^0\in\dd^{\Delta_0}_*$, and with $P_t$ as above then the following statements are true:
	
	\begin{list}{$(\theenumi)$}{\usecounter{enumi}\leftmargin=1cm \labelwidth=1cm \itemsep=0.2cm \topsep=.2cm \renewcommand{\theenumi}{\alph{enumi}}}
		
		\item If $g\in L^q(\pom,\omega_0)$, $1\le q\le \infty$, then 
		\[
		\sup_{0<t<\ell(Q)}\|P_t g\|_{L^q(2\widetilde{\Delta}_{Q},\omega_0)}\le C \|g\|_{L^q( 3\widetilde{\Delta}_{Q},\omega_0)}.
		\]

		\item If $g\in L^q(\pom, \omega_0)$, $1\le q\le \infty$, and $0<t<\ell(Q)$ then $P_t (g\mathbf{1}_{Q})\in \Lip(\pom)\cap L^{\infty}(\pom, \omega_0)$.

		\item If $g\in L^q(\pom, \omega_0)$, $1\le q<\infty$, then $P_t g \longrightarrow g$ in $L^q(2\widetilde{\Delta}_{Q},\omega_0)$ as $t\to 0^+$.

		\item If $g\in C(\pom)$ then $P_t g(x)\longrightarrow  g(x)$ as $t\to 0^+$ for every $x\in 2\widetilde{\Delta}_{Q}$.
		
		\item If $\supp(g)\subset \overline{\Delta(x, r)}$ then $\supp(P_t g)\subset \overline{\Delta(x, r+2t)}$.
	\end{list}
\end{lemma}

\begin{proof}
	We start with some preliminaries. Fix $Q\in\dd_{Q^0}$ with $Q^0\in\dd^{\Delta_0}_*$. 	
	Set 
	\[
	H(x):=\int_{\partial\Omega}\varphi\left(\frac{|x-z|}{t}\right)\,d\omega_0(z),\qquad x\in\pom
	\]
	and observe that $\omega_0(\Delta(x,t))\le H(x)\le \omega_0(\Delta(x,2t))$. 
	Hence if $x,y\in\ \pom$
	\begin{equation}\label{kernel-Pt}
	\frac{\mathbf{1}_{\Delta(x,t)}(y)}{\omega_0(\Delta(x,2t))}
	\le
	\varphi_t(x,y)
	\le 
	\frac{\mathbf{1}_{\Delta(x,2t)}(y)}{\omega_0(\Delta(x,t))}.
	\end{equation}
	This easily implies $(e)$ and also, recalling the notation in  \eqref{deltaQ}--\eqref{deltaQ3}, 
	\begin{equation}\label{kernel-Pt:2}
	\frac{\mathbf{1}_{\Delta(x,t)}(y)}{\omega_0(\Delta(x,t))}
	\lesssim
	\varphi_t(x,y)
	\lesssim 
	\frac{\mathbf{1}_{\Delta(x,2t)}(y)}{\omega_0(\Delta(x,2t))},
	\qquad
	0<t<\ell(Q^0),\quad x\in 4\widetilde{\Delta}_{Q^0},
	\end{equation}
	by Lemma \ref{lemma:proppde} part $(c)$, and the implicit constant does not depend on $t$. Moreover, 
	for every $x\in 4\widetilde{\Delta}_{Q}$ 
	\begin{equation}\label{Pt-M}
	\sup_{0<t<\ell(Q)}|P_tg(x)|
	\le
	C\sup_{0<t<2\ell(Q)} \aver{\Delta(x,t)} |g(y)|\,d\omega_0(y).
	\end{equation}

	Note also that fixed $0<t<\ell(Q)\le\ell(Q^0)<r_0$ for every $x\in 4\widetilde{\Delta}_{Q}$ we have
	$\delta(X_{\Delta(x,2t)})\ge 2c_0t$ and since $Q^0\in \dd^{\Delta_0}_*$
	\begin{multline*}
	|X_{\Delta(x,2t)}-X_{\Delta_0}|
	\le
	|X_{\Delta(x,2t)}-x|+|x-x_{Q}|+|x_Q-x_{Q^0}|+|x_{Q^0}-x_0|+|x_0-X_{\Delta_0}|
	\\
	\le 
	2t+6\Xi\ell(Q^0)+3r_0
	\lesssim r_0.
	\end{multline*}
	Hence, the Harnack Chain condition and Harnack's inequality yield
	\begin{equation}\label{ulowebound}
	\omega_0(\Delta(x, 2t)) 
	\approx_t
	\omega_{L_0}^{X_{\Delta(x,2t)}}(\Delta(x, 2t)) 
	\approx 1
	\end{equation}
	where the last estimate follows from Lemma \ref{lemma:proppde} part $(a)$ and the implicit constants depend on $t$ but are uniform in $x\in 4\widetilde{\Delta}_{Q}$.

	To show $(a)$, note first $(P_t g)\mathbf{1}_{2\widetilde{\Delta}_{Q}}=(P_t (g \mathbf{1}_{3\widetilde{\Delta}_{Q}}))\mathbf{1}_{2\widetilde{\Delta}_{Q}}$ whenever $0<t<\ell(Q)$. This,  Fubini's theorem  and \eqref{Pt-M} yield
	\[
	\|P_t g\|_{L^1(2\widetilde{\Delta}_{Q},\omega_0)}\le \|g\|_{L^1(3\widetilde{\Delta}_{Q},\omega_0)}
	\qquad\mbox{and}\qquad
	\|P_t g\|_{L^\infty(2\widetilde{\Delta}_{Q},\omega_0)}\le C \|g\|_{L^\infty(3\widetilde{\Delta}_{Q},\omega_0)}.
	\]
	Thus, $(a)$ follows easily from  Marcinkiewicz's interpolation theorem.
	
	To obtain $(b)$ we first observe that  $(e)$ yields $\supp(P_t(g\mathbf{1}_{Q}))\subset 3\widetilde{\Delta}_{Q}$. This, \eqref{kernel-Pt:2}, Hölder's inequality, and \eqref{ulowebound} give
	for every $x\in  3\widetilde{\Delta}_{Q}$ 
	\[
	|P_t (g\mathbf{1}_{Q})(x)|
	\lesssim
	\aver{\Delta(x,2t)} |g(y)| \mathbf{1}_{Q}(y)\,d\omega_0(y)
	\lesssim_t
	\|g\|_{L^q(Q,\omega_0)}.
	\]
	Thus, $P_t (g\mathbf{1}_{Q})\in L^{\infty}(\pom,\omega_0)$.
	
	We next see that $P_t (g\mathbf{1}_{Q})\in\Lip(\pom)$. Using what we have proved so far it is trivial to see that it suffices to consider the case on which $|x-x'|<\ell(Q)$ and both $x,x'\in 4\widetilde{\Delta}_{Q}$. Taking such points we  note that
	\[
	|P_t (g\mathbf{1}_{Q})(x)-P_t (g\mathbf{1}_{Q}) (x')|
	\le
	\int_{\pom}|\varphi_t(x,y)-\varphi_t(x',y)|\, |g(y)|\, \mathbf{1}_{Q}(y)\,d\omega_0(y).
	\]
	Note that for every $y\in Q$ we have by the mean value theorem and easy calculations
	\begin{align*}
	|\varphi_t(x,y)-\varphi_t(x',y)|
	&\le 
	\frac1{H(x)}\left|\varphi\left(\frac{|x-y|}{t}\right)-\varphi\left(\frac{|x'-y|}{t}\right)\right|
	\\
	&\qquad\qquad\qquad
	+
	\varphi\left(\frac{|x'-y|}{t}\right)\left|\frac1{H(x)}-\frac1{H(x')}\right|
	\\
	&\lesssim
	\frac{\|\nabla\varphi\|_{L^\infty}}{t\omega_0(\Delta(x,t))}
	\left(
	1+\frac1{\omega_0(\Delta(x',t))}\right)|x-x'|
	\\
	&\lesssim_t
	\|\nabla \varphi\|_{L^\infty}|x-x'|,
	\end{align*}
	where in the last estimate we have used \eqref{ulowebound}. Consequently,
	\begin{multline*}
	|P_t (g\mathbf{1}_{Q})(x)-P_t (g\mathbf{1}_{Q}) (x')|
	\lesssim_t
	\|\nabla \varphi\|_{L^\infty}|x-x'|
	\int_{\pom}|g(y)| \mathbf{1}_{Q}(y)\,d\omega_0(y)
	\\
	\lesssim
	\|\nabla \varphi\|_{L^\infty}\|g\|_{L^q(\omega_0, Q)}
	|x-x'|,
	\end{multline*}
	and this completes the proof of $(b)$.

	Let us now establish $(d)$. Since $g\in C(\pom)$ and $\pom$ is bounded, $g$ is uniformly continuous and hence given $\varepsilon>0$ there exists $\eta>0$ such that
	$|g(y)-g(x)|<\varepsilon$ whenever $|x-y|<\min\{\eta,\ell(Q)\}$. Hence, if $0<t<\eta/2$ and $x\in 4\widetilde{\Delta}_{Q}$ by \eqref{kernel-Pt:2}
	\[
	|P_t g(x)-g(x)|
	\lesssim
	\aver{\Delta(x,2t)}|g(y)-g(x)|d\omega_0
	<\varepsilon
	\]
	and therefore $P_t g(x)\longrightarrow g(x)$ for every $x\in 4\widetilde{\Delta}_{Q}$ (which is indeed stronger than what stated in $(d)$).
	
	Finally, we show $(c)$. To set the stage, fix $\varepsilon>0$ and $g\in L^q(\omega_0, \partial\Omega)$, $1\le q<\infty$.  Pick $h\in C(\pom)$ such that $\|g-h\|_{L^q(\pom,\omega_0)}<\varepsilon$.  Proceeding as in the proof of $(d)$ there exists $\eta>0$ such that
	$|h(y)-h(x)|<\varepsilon$ whenever $|x-y|<\min\{\eta,\ell(Q)\}$. Hence, if $0<t<\eta/2$ and $x\in 2\widetilde{\Delta}_{Q}$ by \eqref{kernel-Pt:2}
	\[
	|P_t h(x)-h(x)|
	\lesssim
	\aver{\Delta(x,2t)}|h(y)-h(x)|d\omega_0
	\le \varepsilon.
	\]
	Using all these we obtain for all $0<t<\eta/2$ 
	\begin{multline*}
	\|P_tg - g\|_{L^q( 2\widetilde{\Delta}_{Q},\omega_0)}
		\le
	\|P_t(g - h)\|_{L^q( 2\widetilde{\Delta}_{Q},\omega_0)}
\\
+	\|P_t h - h \|_{L^q( 2\widetilde{\Delta}_{Q},\omega_0)}
	+
	\|h - g\|_{L^q( 2\widetilde{\Delta}_{Q},\omega_0)}
	\lesssim
	\varepsilon
	\end{multline*}
	where we have used item $(a)$ and the fact that $\omega_0(\pom)\le 1$. This completes the proof.
\end{proof}

\begin{lemma}\label{lemma:carleson-discrete}
	Let $\Omega\subset\mathbb{R}^{n+1}$ be a 1-sided NTA domain satisfying the CDC and adopt the notation introduced above. 
	There exists $\kappa>0$ depending only on dimension $n$, the 1-sided NTA constants, the CDC constant, and the ellipticity constant of $L_0$ (and independent of $\Delta_0$) such that if $Q^0\in\dd^{\Delta_0}_*$  and we set 
	\begin{equation}\label{coef-carleson}
	\gamma_{Q}=\gamma_{X_0,Q}:=\omega_0(Q)\sum_{I\in\mathcal{W}_{Q}^*} \|A-A_0\|_{L^\infty(I^*)}^2,\qquad Q\in\mathbb{D}_{Q^0},
	\end{equation}
	then $\|\mathfrak{m}_{\gamma}\|_{\mathcal{C}(Q^0, \omega_0)}\leq\kappa \vertiii{\varrho(A,A_0)}_{B_0}$.
\end{lemma}

\begin{proof}
	Fix $Q^0\in\dd^{\Delta_0}_*$ and pick $y_0\in Q^0\cap \Delta_0$. Let $Q\in\dd_{Q^0}$ and note that by \eqref{deltaQ} and the fact that $\kappa_0\ge 16\Xi$
\[
|x_Q-x_0|
\le 
|x_Q-y_0|+|y_0-x_0|
<
2\Xi r_{Q^0}+r_0
\le
2\Xi \ell(Q^0)+r_0
\le
\left(\frac{\Xi c_0}{4\kappa_0}+1\right)r_0
<2 r_0.
\]	
Hence $x_Q\in 2\Delta_0$. Note also that $r_{B_Q^*}=2\kappa_0r_Q\le 2\kappa_0\ell(Q^0)<r_0c_0/4$. This means that $B_Q^*$ is one of the balls in the sup in \eqref{def-varrho-local}. Also, $X_{0}\notin 4B_Q^*$ hence if $Q'\in\mathbb{D}_{Q}$ and $Y\in I^*\in \W_{Q'}^*$ we have by Harnack's inequality  and Lemma \ref{lemma:proppde} parts $(a)$--$(c)$, 
\[
\omega_0(Q')
\approx
\omega_0(\Delta_{Q'})
\approx
\ell(Q')^{n-1} G_{L_0}(X_0,X_{Q'})
\approx
\delta(Y)^{n-1}G_{L_0}(X_0,Y).
\]
On the other hand, by \eqref{constwhitney} and recalling that $I^*=(1+\lambda)I$ with $0<\lambda<1$, it follows that  $I^{*}\subset B(Y,\delta(Y)/2)$ and thus $\|A-A_0\|_{L^\infty(I^*)}\leq 	\varrho(A,A_0)(Y)$. All these imply
	\begin{align}
	\label{mCarlesononQ:q}
	\mut_\gamma(\dd_Q)
	&=
	\sum\limits_{Q'\in \dd_Q} \omega_0(Q')\sum\limits_{I\in\mathcal{W}_{Q'}^{*}} \|A-A_0\|_{L^\infty(I^*)}^2
	\\ \nonumber
	&
	\leq  
	\sum\limits_{Q'\in \dd_Q} \omega_0(Q')\sum\limits_{I\in \mathcal{W}^{*}_{Q'}} \iint_{I^*} \frac{\varrho(A,A_0)(Y)^2}{\ell(I)^{n+1}} dY
	\\ \nonumber
	&
	\approx 
	\sum\limits_{Q'\in \dd_Q} \iint_{U_{Q'}}  \varrho(A,A_0)(Y)^2\, \frac{\omega_{0}(Q')}{\delta(Y)^{n+1}}dY
	\\ \nonumber
	&
	\approx
	\sum\limits_{Q'\in \dd_Q} \iint_{U_{Q'}}  \varrho(A,A_0)(Y)^2\, \frac{G_{L_0}(X_0,Y)}{\delta(Y)^2} dY
	\\ \nonumber
	&
	\lesssim 
	\iint_{T_Q}  \varrho(A,A_0)(Y)^2\frac{G_{L_0}(X_0,Y)}{\delta(Y)^2} dY
	\\ \nonumber
	&
	\lesssim 
	\vertiii{\varrho(A,A_0)}_{B_0}\,\omega_0(\Delta_Q^*)
	\\ \nonumber
	&
	\lesssim 
	\vertiii{\varrho(A,A_0)}_{B_0}\, \omega_0(Q),
	\end{align}
	where have used that the families $\{I^*\}_{I\in \W}$ and $\{U_{Q'}\}_{Q'\in\dd_{Q}}$ have bounded overlap, \eqref{definicionkappa12}, and Lemma \ref{lemma:proppde}, parts $(b)$ and $(c)$. This leads to the desired estimate. 
\end{proof}

\medskip

For each $j\in\mathbb{N}$ (large enough), let  (see Figure \ref{figure:Lj})
\begin{equation}\label{def:Aj}
A^{j}(Y)=
\left\{
\begin{array}{ll}
A(Y) &\mbox{if}\ Y\in\Omega\mbox{ and }\delta(Y)\geq 2^{-j};
\\[4pt]
A_0(Y) &\mbox{if}\ Y\in\Omega\mbox{ and }\delta(Y)< 2^{-j},
\end{array}
\right.
\end{equation}
and define $L^ju=-\div(A^j\nabla u)$. Note that the matrix $A^j$ is uniformly elliptic with constant $\Lambda_0=\max\{\Lambda_A,\Lambda_{A_0}\}$, where $\Lambda_A$ and $\Lambda_{A_0}$ are the ellipticity constants of $A$ and $A_0$ respectively. Let $\omega_{L^j}$ be elliptic measure of $\Omega$ associated to the operator $L^j$ with pole at $X_0$.

\begin{figure}[!ht]
	\def\firstcircle{(0,0) circle (4cm)}
	\def\secondcircle{(0,0) circle (1cm)}
	\centering
	\begin{tikzpicture}[scale=.4]
	
	\draw[dashed, yshift=+1cm] plot [smooth] coordinates {(-14,.5)(-12,.2)(-10,.5)(-8,1)(-6,.6) (-4,.3) (-2,-.3) (0,0)(2,.3)(4,-.3)(6,-.6)(8,-1)(10,-.5)(12,.2)(14,-.5)};
	\draw[red] plot [smooth] coordinates {(-14,.5)(-12,.2)(-10,.5)(-8,.98)(-6,.6) (-4,.3) (-2,-.3) (0,0)(2,.3)(4,-.3)(6,-.6)(8,-1)(10,-.5)(12,.2)(14,-.5)};
	\fill[gray, opacity=.1] plot [smooth] coordinates {(-14,.5)(-12,.2)(-10,.5)(-8,.98)(-6,.6) (-4,.3) (-2,-.3) (0,0)(2,.3)(4,-.3)(6,-.6)(8,-1)(10,-.5)(12,.2)(14,-.5)} (14, -.5)--(14,5)--(-14,5)--(-14,.5);
	
	\node[left, red] at (-14,.5) {$\partial\Omega$};

	\begin{scope}
	\path[clip] plot [smooth] coordinates {(-14,.5)(-12,.2)(-10,.5)(-8,.98)(-6,.6) (-4,.3) (-2,-.3) (0,0)(2,.3)(4,-.3)(6,-.6)(8,-1)(10,-.5)(12,.2)(14,-.5)} (14, -.5)--(14,5)--(-14,5)--(-14,.5);\fill[gray] plot [smooth] coordinates {(-14, -1.5)(-14,1.5)(-12,1.2)(-10,1.5)(-8,2)(-6,1.6) (-4,1.3) (-2,.7) (0,1)(2,1.3)(4,.7)(6,.4)(8,0)(10,.5)(12,1.2)(14,.5)(14, -1.5)(-14,-1.5)};
	
	\draw[<->] (12,.2)--(12,1.2);
	\node at (13,.5) {{\tiny $2^{-j}$}};
	
	\node at (0,3) {$A$};
	\node at (-6,1.1) {{\tiny $A_0$}};
	\node at (8,-.5) {\tiny $A_0$};
	\end{scope}
	
	\end{tikzpicture}
	\caption{Definition of the matrix $A^{j}$ in $\Omega$.}\label{figure:Lj}
\end{figure}
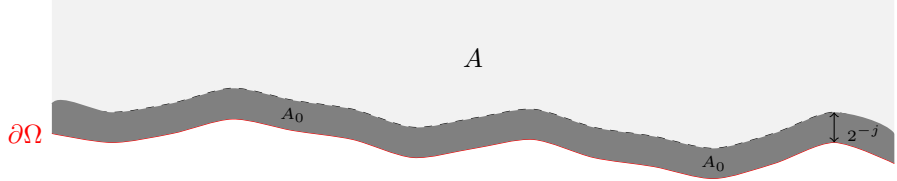

The following result is a version of \cite[Proposition 4.28]{CHM} adapted to our setting. 

\begin{lemma}	\label{lemma:LjtoL}
Let $\Omega\subset\mathbb{R}^{n+1}$ be \textbf{bounded} 1-sided NTA domain satisfying the CDC.	Assume that there exists $q$, $1<q<\infty$, such that $\omega_{L^j}\in RH_q(\frac54\Delta_0,\omega_{0})$ for every $j\geq j_0$ and with implicit constants which are uniform in $j$ and in $\Delta_0$. Then $\omega_L\in RH_q(\Delta_0,\omega_0)$ with 
$[\omega_{L}]_{RH_q(\Delta_0,\omega_{0})}\lesssim \sup_{j\ge j_0}[\omega_{L^j}]_{RH_q(\frac54\Delta_0,\omega_{0})}$, with an implicit constant depending on dimension $n$, the 1-sided NTA constants, the CDC constant, and the ellipticity constants of $L_0$ and $L$ (and independent of $\Delta_0$).
\end{lemma}

\begin{proof}
Set $\Upsilon:=\sup_{j\ge j_0}[\omega_{L^j}]_{RH_q(\frac54\Delta_0,\omega_{0})}$. Consider an arbitrary $\Delta_0'=B_0'\cap\pom$ with $B_0'=B(x_0',r_0')\subset B_0$. Write $X_0'=X_{\Delta_0'}$, $\omega'=\omega_{L}^{X_0'}$, $\omega_0'=\omega_{L_0}^{X_0'}$ (and note that $\omega_0=\omega_{L_0}^{X_0}$ since $X_0=X_{\Delta_0}$). Write $\Delta_1=\frac54\Delta_0'$, let $r_1=\frac54r_0'$ be its radius and set $X_1=X_{\Delta_1}$. By hypotheses $\omega_{L^j}\ll \omega_{0}$ in $\frac54\Delta_0$, hence $h(\cdot\,;L^j,L_0,X)$ is defined $\omega_0$-a.e. in $\frac54\Delta_0$.

If $r_0'<c_0r_0/(3\kappa_0)$ so that $X_{0}\in\Omega\setminus 2\kappa_0 B_1$, by Lemma \ref{lemma:proppde} part $(d)$ applied to $L_j$ and $L_0$ we have
\begin{equation}\label{cop:densities}
h(\cdot\,;L^j,L_0,X_{0})
=
\frac{d\omega_{L^j}^{X_0}}{d\omega_{L_0}^{X_0}}
=
\frac{d\omega_{L^j}^{X_0}}{d\omega_{L^j}^{X_1}}
\frac{d\omega_{L^j}^{X_1}}{d\omega_{L_0}^{X_1}}
\frac{d\omega_{L_0}^{X_1}}{d\omega_{L_0}^{X_0}}
\approx
\frac{\omega_{L^j}^{X_0}(\Delta_1)}{\omega_{L_0}^{X_0}(\Delta_1)}  h(\cdot\,;L^j,L_0,X_{1}),
\end{equation}
$\omega_0$-a.e. in $\Delta_1$. This and Lemma \ref{lemma:proppde} part $(d)$ give
\begin{multline}\label{est:h-Delta1}
\| h(\cdot\,;L^j,L_0,X_1)\|_{L^{q}(\Delta_1,\omega_{L_0}^{X_1})}
\approx
\frac{1}{\omega_{L^j}^{X_0}(\Delta_1)}
\| h(\cdot\,;L^j,L_0,X_0)\|_{L^{q}(\Delta_1,\omega_{0})}
\\
\le
[\omega_{L^j}]_{RH_q(\frac54\Delta_0,\omega_{0})} \omega_{0}(\Delta_1)^{-\frac1{q'}}
\le
\Upsilon \omega_0(\Delta_0')^{-\frac1{q'}}
,
\end{multline}
where the implicit constants are independent of $j$.

For any $f\in C(\partial\Omega)$, we define  $\Phi(f):=
\int_{\partial\Omega} f(y) d\omega'(y)$.
Let $f\in \Lip(\partial\Omega)$ with $\supp(f)\subset \Delta_1$ and consider the following solutions to the Dirichlet problems associated with the operators $L$ and $L^j$ in $\Omega$: 
\[
u(X)=\int_{\partial\Omega} f(y) d\omega^{X}_{L}(y) \quad \mbox{and}\quad u_j(X)=\int_{\partial\Omega} f(y) d\omega^{X}_{L^j}(y),
\qquad X\in\Omega.
\]
Implicit in the way that $\omega_{L^j}$ is defined and since $\Omega$ is bounded one has that $u_j=F-v_j$ where $F$ is a compactly supported Lipschitz extension (e.g., \cite[p. 80]{EG} multiplied  some cut-off function) of $f$ such that $\|F\|_{\Lip (\ree)}\le\|f\|_{\Lip (\pom)}+\|f\|_{L^\infty(\pom)}$  and $v_j\in W_0^{1,2}(\Omega)$ is the unique Lax-Milgram solution to the problem $L^j v_j= L^j F$ in $\Omega$. Also, one has 
\begin{equation}\label{sup-nabla-uj}
\sup_j \|u_j\|_{W^{1,2}(\Omega)}
\le
C_\Omega \|F\|_{W^{1,2}(\Omega)}<\infty
\end{equation}
where the implicit constants depend on $\diam(\pom)$ and $\Lambda_0$. 

Since $f\in \Lip(\partial\Omega)$ it follows that we can use Lemma \ref{lemma:u0-u:bounded}  (slightly moving $X_0'$ if needed) to obtain 
\[
u(X_0')-u_j(X_0')=\int_{\Omega}(A^j-A)^\top(Y)\nabla_{Y} G_{L^\top}(Y,X_0')\cdot \nabla u_j(Y) dY.
\]
We want to estimate the right hand-side of this identity. To this end, if $j>j_0$ is large enough so that $2^{-j}<\delta(X_{\Delta_0'})/2$ then 
$\Sigma_j:=\{Y\in \Omega:\, \,  \delta(Y)<2^{-j}\} \cap B(X_0', \delta(X_0')/2)=\emptyset$. Using \eqref{uniformlyelliptic} and  H\"older's inequality we have
\begin{multline}
	|u(X_0')-u_j(X_0')|
\lesssim 
\int_{\Omega\cap \Sigma_j} |\nabla_{Y} G_{L^\top}(Y,X_0')|\,|\nabla u_j(Y)| dY \\
\\
\lesssim 
\|\nabla_{Y} G_{L^\top}(\cdot,X_0') \, \mathbf{1}_{\Sigma_j}\|_{L^2(\Omega)} \sup_j \|u_j\|_{W^{1,2}(\Omega)}.
\end{multline}
By Remark \ref{rem:GF} and \eqref{sup-nabla-uj} the dominated convergence theorem gives that $u_j(X_0')\longrightarrow  u(X_0')$ as $j\to\infty$. Using this observation, the  definitions of $u$, $u_j$, $\Phi$, and the fact that $\supp(f)\subset\Delta_1$,  we get that for every $f\in \Lip(\pom)$ with  $\supp(f)\subset \Delta_1$
\begin{multline}\label{est-Phi-f-Lip}
|\Phi(f)|
=
|u(X_0')|
=
\lim_{j\to\infty} |u_j(X_0')| 
\lesssim 
\|f\|_{L^{q'}(\Delta_1, \omega_0')} 
\sup_{j\ge j_1} \| h(\cdot\,;L^j,L_0,X_{\Delta_1} )\|_{L^{q}(\Delta_1,\omega_{L_0}^{X_1})} 
\\
\lesssim 
\|f\|_{L^{q'}(\Delta_1, \omega_{0})} \Upsilon\,\omega_0(\Delta_0')^{-\frac1{q'}}.
\end{multline}
Note that in the previous inequalities we have employed that $\Delta_0'\subset\Delta_1$ have comparable radii, Harnack's inequality, and \eqref{est:h-Delta1}.

We next write $\Delta_2=\frac98\Delta_0'$ so that $\Delta_0'\subset\overline{\Delta_0'}\subset \Delta_2\subset\overline{\Delta_2}\subset \Delta_1$ and let $f\in L^{q'}(\Delta_2,\omega_0')$  (where we recall that $\omega_0'=\omega_{L_0}^{X_{\Delta_0'}}$). Abusing the notation we extend $f$ by $0$ in $\pom\setminus\Delta_2$ so that $\supp(f)\subset \overline{\Delta_2}$. By definition of $\dd^{\Delta_0'}_*$, see \eqref{D-delta-new}, we have that $\Delta_0'\subset \Delta_1\subset \bigcup_{Q\in \dd^{\Delta_0'}_*}Q$ where the cubes in $\dd^{\Delta_0'}_*$ are pairwise disjoint. Also, by Harnack's inequality and Lemma \ref{lemma:proppde} parts $(a)$ and $(c)$
\[
\# \dd^{\Delta_0'}_*
\approx
\# \dd^{\Delta_0'}_* \omega_0'(\Delta_0')
\le 
\sum_{Q\in \dd^{\Delta_0'}_*}\omega_0'(Q)
\le
\omega_0 \Big(\bigcup_{Q\in \dd^{\Delta_0'}_*} Q\Big)
\le 1,
\]
hence $\# \dd^{\Delta_0}_*$ is uniformly bounded. This means that by Lemma \ref{lemma:Pt} applied with $\omega_0'$ in place of $\omega_0$
\[
P_t f
=
\sum _{Q\in \dd^{\Delta_0'}_*} P_t(f\mathbf{1}_{Q})\in L^{\infty}(\pom, \omega_0')\cap \Lip(\pom)
\]
provided $0<t<c_0r_0'/(32\kappa_0)=:t_0$. Note that $t_0\le \ell(Q)$ for every $Q\in \dd^{\Delta_0'}_*$. Also Lemma \ref{lemma:Pt} applied with $\omega_0'$ in place of $\omega_0$ implies that 
\[
\supp (P_t f)\subset \overline{\Delta(x_0',\tfrac98 r_0'+2t)}\subset  \Delta_1,
\] 
provided $0<t<r_0'/16$. Consequently, if $0<t<t_0$ we have shown that $P_t f\in \Lip(\pom)$ with $\supp (P_t f)\subset \Delta_1$. We can then invoke \eqref{est-Phi-f-Lip} to see that
\begin{multline*}
\Upsilon^{-1}\,\omega_0(\Delta_0')^{\frac1{q'}} \sup_{0<t<t_0}|\Phi(P_t f)|
\lesssim
\sup_{0<t<t_0}
\|P_t f\|_{L^{q'}(\Delta_1', \omega_{0})}
\\
\le
\sum_{Q\in \dd^{\Delta_0'}_*} \sup_{0<t<\ell(Q)} \|P_t(f\mathbf{1}_{Q}) \|_{L^{q'}(2\widetilde{\Delta}_Q, \omega_{0}')}
\lesssim
\sum _{Q\in \dd^{\Delta_0'}_*} \|f\mathbf{1}_{Q}\|_{L^{q'}(3\widetilde{\Delta}_Q, \omega_{0}')}
\lesssim 
\|f\|_{L^{q'}(\Delta_2, \omega_{0}')},
\end{multline*}
where we have used that $\supp (P_t(f\mathbf{1}_{Q}))\subset \overline{\Delta(x_Q,C r_Q+2t)}\subset 2\widetilde{\Delta}_Q$ for every $Q\in \dd^{\Delta_0}_*$, Lemma \ref{lemma:Pt} applied with $\omega_0'$ in place of $\omega_0$, and that $\# \dd^{\Delta_0}_*$ is uniformly bounded.

On the other hand, if $0<t,s<t_0$ we have that $P_t f-P_s f\in \Lip(\pom)$ with $\supp (P_t f-P_s f) \subset \Delta_1$ and again we can invoke \eqref{est-Phi-f-Lip} to see that a similar computation lead us to
\begin{multline*}
\Upsilon^{-1}\omega_0(\Delta_0')^{\frac1{q'}} |\Phi(P_t f)-\Phi(P_s f)|
=
\Upsilon^{-1}\omega_0(\Delta_0')^{\frac1{q'}} |\Phi(P_t f-P_s f)|
\\
\lesssim
\|P_t f-P_s f\|_{L^{q'}(\Delta_1, \omega_{0}')}
\le
\|P_t f-f\|_{L^{q'}(\Delta_1, \omega_{0}')}
+
\|P_s f-f\|_{L^{q'}(\Delta_1, \omega_{0}')}
\\
\le
\sum _{Q\in \dd^{\Delta_0}_*} \|P_t(f\mathbf{1}_{Q})-f\mathbf{1}_{Q}\|_{L^{q'}(2\widetilde{\Delta}_Q, \omega_{0}')}+ \|P_s(f\mathbf{1}_{Q})-f\mathbf{1}_{Q}\|_{L^{q'}(2\widetilde{\Delta}_Q, \omega_{0}')}.
\end{multline*}
This and Lemma \ref{lemma:Pt}  applied with $\omega_0'$ in place of $\omega_0$ yield that $\{\Phi(P_t f)\}_{0<t<t_0}$ is a Cauchy sequence and we can define $\widetilde{\Phi}(f):=\lim_{t\to 0^+}\Phi(P_t f)$. Clearly, $\widetilde{\Phi}$ is a well-defined linear operator and satisfies
\[
|\widetilde{\Phi}(f)|
=\lim_{t\to 0^+}|\Phi(P_t)|
\le
\sup_{0<t<t_0}|\Phi(P_t f)|
\lesssim
\Upsilon\omega_0(\Delta_0')^{-\frac1{q'}}\|f\|_{L^{q'}(\Delta_2, \omega_{0}')}.
\]
Consequently, there exists $g\in L^q(\Delta_2,\omega_0')$ with $\|g\|_{L^q(\Delta_2,\omega_0')}\lesssim\Upsilon\omega_0(\Delta_0')^{-\frac1{q'}}$ such that 
\begin{equation}\label{e:P0gest}
\widetilde{\Phi}(f)
=
\int_{\Delta_2} f(y)g(y)\,d\omega_0'(y),
\qquad \forall\, f\in L^{q'}(\Delta_2, \omega'_0).
\end{equation}

We now assume that $f\in C(\partial\Omega)$ with $\supp(f)\subset \Delta_2$, thus $f\in L^{q'}(\Delta_2,\omega_0')$ and hence $P_t f\in\Lip(\pom)$. Also, proceeding as above
\begin{multline*}
\sup_{0<t<t_0}|\Phi(P_t f)|
\le
\sum _{Q\in \dd^{\Delta_0'}_*} \sup_{0<t<\ell(Q)} \|P_t(f\mathbf{1}_{Q}) \|_{L^{\infty}(2\widetilde{\Delta}_Q, \omega_{0}')}
\\
\lesssim
\sum _{Q\in \dd^{\Delta_0'}_*} \|f\mathbf{1}_{Q}\|_{L^{\infty}(3\widetilde{\Delta}_Q, \omega_{0}')}
\lesssim 
\|f\|_{L^{\infty}(\pom, \omega_{0}')}.
\end{multline*}
Note also that, as mentioned above, for $t$ small enough  one has $\supp (P_t f)\subset  \Delta_1$ and the cubes in $\dd^{\Delta_0'}_*$ cover $\Delta_1$. Hence by Lemma \ref{lemma:Pt} applied with $\omega_0'$ in place of $\omega_0$ it follows that $P_t f(x)\longrightarrow  f(x)$ as $t\to 0^+$ for every $y\in \Delta_1$. These, the definitions of $\Phi$, $\widetilde{\Phi}$, and the dominated convergence theorem yield for every $f\in C(\partial\Omega)$ with $\supp(f)\subset \Delta_2$
\begin{multline}\label{e:P0gPg}
\widetilde{\Phi} (f) 
= 
\lim_{t\to 0^{+}} \Phi(P_t f) 
=
\lim_{t\to 0^{+}} \int_{\pom} P_t f(y) d\omega'(y)
=
\lim_{t\to 0^{+}} \int_{\Delta_1} P_t f(y) d\omega'(y)
\\
=
\int_{\Delta_1} f(y) d\omega'(y)
=
\int_{\pom } f(y) d\omega'(y)
=
\Phi(f).
\end{multline}

Our next goal is to show that $\omega'= \omega_{L}^{X_0'} \ll \omega_{L_0}^{X_0'}=\omega_0'$ in $\Delta_3=\frac{17}{16}\Delta_0'$. Let $E\subset \Delta_3$ a Borel set. Since both measures are Borel regular, given $\varepsilon>0$ we can find a compact set $K$ and open set $U$ such that $K\subset E\subset U\subset \Delta_2$ satisfying 
$\omega(U\setminus K)+\omega_0(U\setminus K)<\varepsilon$.
Using Urysohn's lemma we construct $f\in C_c(\partial\Omega)$ such that $\mathbf{1}_K\leq f\leq\mathbf{1}_U$ and $\supp(f)\subset \Delta_2$. Thus, combining \eqref{e:P0gest} and \eqref{e:P0gPg}, and using definition of $\Phi$ and $\widetilde{\Phi}$ we have  
\begin{multline*}
\omega'(E)
\leq 
\varepsilon
+ 
\omega'(K) 
\leq 
\varepsilon
+ 
\int_{\pom} f(y) d\omega'(y) 
=
\varepsilon+\Phi(f) 
= 
\varepsilon + \widetilde{\Phi}(f)  
\\
\leq 
\varepsilon 
+ 
\| f\|_{L^{q'}(\Delta_2, \omega_0')} 
\| g\|_{L^{q}(\Delta_2, \omega_0')}
\lesssim 
\varepsilon + [(\varepsilon+\omega_0'(E))^{\frac1{q'}}\Upsilon\omega_0(\Delta_0')^{-\frac1{q'}}].
\end{multline*}
By letting $\varepsilon\to 0$ we see that $\omega'(E)\lesssim \omega_0'(E)^{\frac1{q'}}\Upsilon\omega_0(\Delta_0')^{-\frac1{q'}}$ and consequently $\omega'\ll\omega_0'$ in $\Delta_3$. Thus we can write $h(\cdot):=h(\cdot\,;L, L_0, X_0')=\frac{d\omega_L^{X_0'}}{d\omega_{L_0}^{X_0'}}=\frac{d\omega'}{d\omega_0'}\in L^1(\Delta_3,\omega_0')$ which is well-defined for $\omega_0'$-a.e. point in $\Delta_3$ and if $f\in C(\pom)$ with $\supp f\subset \Delta_3\subset \Delta_2$
\begin{align}
\label{e:h=h}
\int_{\Delta_3} f(y)g(y)d\omega_0'(y) 
= \widetilde{\Phi}(f) 
= 
\Phi(f) 
= 
\int_{\partial\Omega} f(y) d\omega'(y) 
= 
\int_{\Delta_3 } f(y)h(y)d\omega_0'(y).
	\end{align}
Note that $\widetilde{h}=(g-h)\mathbf{1}_{\Delta_3}\in L^1(\pom,\omega_0')$ hence proceeding as above if $0<t<t_0$ Lemma \ref{lemma:Pt} applied with $\omega_0'$ in place of $\omega_0$ gives
\[
\|P_t\widetilde{h}-\widetilde{h}\|_{L^{1}(\Delta_3, \omega_{0}')}
\le
\sum _{Q\in \dd^{\Delta_0'}_*} \|P_t(\widetilde{h}\mathbf{1}_{Q})-\widetilde{h}\mathbf{1}_{Q}\|_{L^{1}(2\widetilde{\Delta}_Q, \omega_{0}')}
\longrightarrow  0,
\qquad\mbox{as }t\to 0^+.
\]
On the other hand, for any $x\in \Delta_0'$ and $0<t<r_0'/32$ if we consider $\varphi_t$ as in \eqref{def:phit} with $\omega_0'$ in place of $\omega_0$ we have 
$\supp(\varphi_t(x,\cdot))\subset \overline{\Delta(x,2t)}\subset \Delta_3$. Thus, we can invoke \eqref{e:h=h} with $f=\varphi_t(x,\cdot)$ to get $P_t \widetilde{h}(x)=0$ for every $x\in\Delta_0'$. Thus, Lemma \ref{lemma:Pt} part $(c)$ applied with $\omega_0'$ allows us to conclude that $\widetilde{h}=0$ $\omega_0'$-a.e. in $\Delta_0'$. Hence $g=h\ge 0$ $\omega_0'$-a.e. in $\Delta_0'$ and using that $\|g\|_{L^q(\Delta_2,\omega_0')}\lesssim \Upsilon\omega_0(\Delta_0')^{-\frac1{q'}}$
\begin{multline}\label{est-RHP-local-aux}
\left(\aver{\Delta_0'} h(y;L, L_0, X_0')^q
d\omega_0'(y)\right)^{\frac1q}
=
\left(
\aver{\Delta_0'} h(y)^q d\omega_0'(y)
\right)^{\frac1q}
\\
=
\left(
\aver{\Delta_0'} g(y)^q d\omega_0'(y)
\right)^{\frac1q}
\lesssim
\Upsilon\frac{\omega_0(\Delta_0')^{-\frac1{q'}}}{\omega_0'(\Delta_0')^{\frac1{q}}}
\approx
\Upsilon\omega_0(\Delta_0')^{-\frac1{q'}},
\end{multline}
where the last estimate follows from Lemma \ref{lemma:proppde} part $(a)$. At this point we can repeat the computations we have done in \eqref{cop:densities} replacing $L^j$ by $L$ and $\Delta_1$ by $\Delta_3$ ---we already know that $\omega'\ll\omega_0'$ in $\Delta_3=\frac{17}{16}\Delta_0'$ where $B_0'$ was arbitrary chosen so that $B_0'\subset B_0$, hence taking $B_0'=B_0$ we conclude that $\omega\ll\omega_0$ in $\Delta_3$--- to obtain that 
\[
h(z;L,L_0,X_{0})
\approx
\frac{\omega_{L}^{X_0}(\Delta_3)}{\omega_{L_0}^{X_0}(\Delta_3)}  h(z;L,L_0,X_{\Delta_3})
\approx
\frac{\omega(\Delta_0')}{\omega_{0}(\Delta_0')}  h(z;L,L_0,X_{0}'),
\]
for $\omega_0$-a.e. $z\in \Delta_3$, and where we have used Harnack's inequality to pass from $X_0'$ to $X_{\Delta_3}$. 
This, Lemma \ref{lemma:proppde} part $(d)$, and \eqref{est-RHP-local-aux}  give
\begin{multline*}
\left(\aver{\Delta_0'} h(y;L, L_0, X_0)^q
d\omega_0(y)\right)^{\frac1q}
\approx
\frac{\omega(\Delta_0')}{\omega_0(\Delta_0')^\frac1q}
\left(\aver{\Delta_0'} h(y;L, L_0, X_0')^q
d\omega_0'(y)\right)^{\frac1q}
\\
\lesssim
\Upsilon
\frac{\omega(\Delta_0')}{\omega_0(\Delta_0')}
=
\Upsilon\aver{\Delta_0'} h(y;L, L_0, X_0)d\omega_0(y).
\end{multline*}
Since $\Delta_0'=B_0'\cap\pom$ was arbitrary with $B_0'=B(x_0',r_0')\subset B_0$ we therefore conclude that 
$\omega_L\in RH_q(\Delta_0,\omega_0)$ with 
$[\omega_{L}]_{RH_q(\Delta_0,\omega_{L_0})}\lesssim\Upsilon$ and this completes the proof. 
\end{proof}

\subsection{Proof Proposition \ref{PROP:LOCAL-VERSION}, part \texorpdfstring{$(a)$}{(a)}}\label{subsection:proof-a}

We start assuming that $\Omega$ is a \textbf{bounded} 1-sided NTA domain satisfying the CDC and whose boundary $\pom$ is bounded. We fix $\dd=\dd(\pom)$ the dyadic grid from Lemma \ref{lemma:dyadiccubes} with  $E=\pom$. As in the statement of Proposition \ref{PROP:LOCAL-VERSION} let $Lu=-\div(A\nabla u)$ and $L_0u=-\div(A_0\nabla u)$ be two real (non-necessarily symmetric) elliptic operators. Fix $x_0\in\pom$ and $0<r_0<\diam(\pom)$ and let $B_0=B(x_0,r_0)$, $\Delta_0=B_0\cap\pom$. From now on $X_0:=X_{\Delta_0}$, $\omega_0:=\omega_{L_0}^{X_0}$ and $\omega:=\omega_{L}^{X_0}$.

We first observe that we can reduce the proof to the case $0<r_0<\diam(\pom)/2$. Assuming that this has been already proved we now explain how to consider the general case. Let $B_0=B(x_0,r_0)$ with $\diam(\pom)/2\le r_0<\diam(\pom)$.  We proceed as \textbf{Case 2b} in the proof of Theorem \ref{thm:main} part $(a)$ with $M=1$ to find the corresponding collection $\{x_j\}_{j=1}^J$ with $J\le 21^{n+1}$. Let $B_j=B(x_j,\diam(\pom)/10)$ for $1\le j\le J$. Then we can easily see that Harnack's inequality yields $\sup_{1\le j\le J} \vertiii{\varrho(A,A_0)}_{B_j,\Omega_\star}\lesssim \vertiii{\varrho(A,A_0)}_{B_0}$ and since $r_{B_j}<\diam(\pom)/2$ we can apply the claimed case to conclude that $\omega_L\in RH_q(3\Delta_j,\omega_{L_0})$ (for part $(b)$, $q=p$). At this point we carry out the same argument \textit{mutatis mutandis} to conclude that $\omega_L\in RH_q(\Delta_0,\omega_{L_0})$ which completes the proof. 

We split the proof in several steps.

\subsubsection{Step 0} 
We first make a reduction which will allow us to use some qualitative properties of the elliptic measure. By Lemma \ref{lemma:LjtoL} it suffices to show that there exists $1<q<\infty$ such that for every $j$ large enough $\omega_{L^j}\in RH_q(\frac54\Delta_0,\omega_{0})$ uniformly in $j$ and in $\Delta_0$. Thus we fix  $j\in\mathbb{N}$ and let $\widetilde{L}=L^j$ be the operator defined by $\widetilde{L}u=-\div(\widetilde{A}\nabla u)$, with $\widetilde{A}=A^j$ (see \eqref{def:Aj}). 
As mentioned above $\widetilde{A}$ is uniformly elliptic with constant $\Lambda_0=\max\{\Lambda_A,\Lambda_{A_0}\}$. Also, since $\widetilde{L}\equiv L_0$ in $\{Y\in\Omega:\,\delta(Y)< 2^{-j}\}$,  by Lemma \ref{lemma:proppde} part $(f)$ and Harnack's inequality give that  $\omega_{L_0}\ll \omega_{\widetilde{L}}\ll \omega_{L_0}$, hence recalling \eqref{def-RN} we have that $h(\cdot;\widetilde{L}, L_0,X)$ exists  $\omega_0^{X}$-a.e. for every $X\in\Omega$. Moreover, fixed $\Delta_1=\Delta(x_1,r_1)$ with $x_1\in\pom$ and $0<r_1<2^{-j-2}/\kappa_0$ for every $\Delta=B\cap\pom$ with $B=B(x,r)\subset B_1$, $x\in\pom$, and $0<r<\diam(\pom)$,  we have by Lemma \ref{lemma:proppde} part $(f)$
\[
1\approx \frac{\omega_{\widetilde{L}}^{X_{\Delta_1}}(\Delta(x,r))}{\omega_{L_0}^{X_{\Delta_1}}(\Delta(x,r))}
=
\aver{\Delta(x,r)} h(y;\widetilde{L}, L_0,X_{\Delta_1})\,d\omega_{L_0}^{X_{\Delta_1}}(y).
\]
Letting $r\to 0^+$ the Lebesgue differentiation theorem (whose applicability is ensured by the fact that $\omega_{L_0}^{X_{\Delta_1}}$ is doubling in $\Delta_1$) yields 
$h(y;\widetilde{L}, L_0,X_{\Delta_1})\approx 1$,
for $\omega_{L_0}^{X_{\Delta_1}}$-a.e.~$x\in \Delta_1$. Thus, by Harnack's inequality 
$h(\cdot\,;\widetilde{L}, L_0,X)\in L_{\rm loc}^\infty(\pom , \omega_{L_0}^{Y})$ for every $X,Y\in\Omega$ ---the actual norm will depend on $X$, $Y$ and $j$, but we will use this fact in a qualitative fashion. This qualitative control will be essential in the following steps. At the end of \textbf{Step 3} we will have obtained the desired conclusion for the operator $\widetilde{L}=L^j$, with constants independent of $j\in\mathbb{N}$,  which as observed above will allow us to complete the proof by Lemma \ref{lemma:LjtoL}.

\medskip

\subsubsection{Step 1}

Let us recall that we have fixed already $x_0\in\pom$ and $0<r_0<\diam(\pom)/2$ and let $B_0=B(x_0,r_0)$, $\Delta_0=B_0\cap\pom$, $X_0=X_{\Delta_0}$, and $\omega_0=\omega_{L_0}^{X_0}$. Set $\widetilde{\omega}:=\omega_{\widetilde{L}}^{X_0}$. Fix $Q^0\in \dd^{\Delta_0}_*$ (see \eqref{D-delta-new}), so that by  \eqref{definicionkappa12}, 
\begin{equation}\label{eq:X0-TQ}
X_{0}\in\Omega\setminus B_{Q^0}^*\subset\Omega\setminus \tfrac12 B_{Q^0}^*\subset\Omega\setminus T_{Q^0}^{**}.
\end{equation}
Set $\mathcal{E}(Y):=A(Y)-A_0(Y)$, $Y\in\Omega$, and consider $\gamma=\{\gamma_Q\}_{Q\in\dd_{Q^0}}$ 
\begin{equation}\label{def-gamma:main}
\gamma_Q=\gamma_{X_0,Q} :=
\omega_{0}(Q)\sum\limits_{I\in\mathcal{W}_{Q}^{*}}\sup\limits_{Y\in I^*} \|\mathcal{E}\|_{L^\infty(I^*)}^{2},
\qquad \mbox{whenever}\, \, Q\in\mathbb{D}_{Q^0}.
\end{equation}
Lemma \ref{lemma:carleson-discrete}  yields that $\|\mathfrak{m}_{\gamma}\|_{\mathcal{C}(Q^0, \omega_0)}\lesssim\vertiii{\varrho(A,A_0)}_{B_0}<\infty$, hence $\mathfrak{m}_{\gamma}$ is a discrete Carleson measure with respect to $\omega_0$ in $Q^0$. Our goal is to show that $\widetilde{\omega}\in A_\infty^{\rm dyadic}(Q^0,\omega_0)$ and we will use Lemma \ref{lemma:extrapolation} with $\mu=\omega_0$. To this aim we fix $Q_0\in\dd_{Q^0}$ and a family of pairwise disjoint dyadic cubes $\mathcal{F}=\{Q_i\}\subset\mathbb{D}_{Q_0}$ such that 
\begin{equation}\label{hipepsilon0}
\|\mathfrak{m}_{\gamma, \mathcal{F}}\|_{\mathcal{C}(Q_0,\omega_0)}
=
\sup_{Q\in\mathbb{D}_{Q_0}}\frac{\mathfrak{m}_\gamma(\mathbb{D}_{\mathcal{F},Q})}{\omega_0(Q)}\leq\varepsilon_0,
\end{equation}
with $\varepsilon_0>0$ sufficiently small to be chosen  and where we have used the notation introduced in 
\eqref{def-car-F} and \eqref{gammauxiliar}.

We modify the operator $\widetilde{L}$ inside the region $\Omega_{\mathcal{F},Q_0}$ (see \eqref{defomegafq}), by defining $L_1=L_1^{\mathcal{F},Q_0}$ as $L_1u=-\div(A_1\nabla u)$, where
\[
A_1(Y):=
\left\{
\begin{array}{ll}
\widetilde{A}(Y) &\mbox{if}\, \, Y\in \Omega_{\mathcal{F},Q_0},
\\[5pt]
A_0(Y) &\mbox{if}\, \, Y\in \Omega\setminus \Omega_{\mathcal{F},Q_0}. 
\end{array}
\right.
\]
See Figure \ref{figure:def:A1}. Recalling that $\widetilde{A}=A^j$ (see \eqref{def:Aj}), it is clear that $\mathcal{E}_1:=A_1-A_0$ verifies $|\mathcal{E}_1|\leq|\mathcal{E}|\mathbf{1}_{\Omega_{\mathcal{F},Q_0}}$ and also $\mathcal{E}_1(Y)=0$ if $\delta(Y)<2^{-j}$ (this latter condition will be used qualitatively). Hence much as before if we write $\omega_1^X=\omega_{L_1}^X$ for every $X\in\Omega$ and $\omega_1=\omega_1^{X_0}$ we have that $\omega_1\ll\omega_0$ and hence we can write $h(\cdot\,; L_1, L_0, X_0)=d\omega_1/d\omega_0$ which is well-defined $\omega_0$-a.e. Also, as shown in \textbf{Step 0} we have that $h(\cdot\,;L_1, L_0,X_{0})\in L_{\rm loc}^\infty(\pom, \omega_0)$ (the bound depends on $X_0$ and the fixed $j$ but we will use this qualitatively).

\begin{figure}[!ht]
\centering
\begin{tikzpicture}[scale=.4]
\draw[fill=blue, opacity=.5] plot [smooth] coordinates {(-9.6,6)(-9.3,4)(-8.8,2.8)(-8.5,2.1)   (-8,1.1)(-6,.6)(-5,.45) (-4,.5)(-2,2.25)  (-.5 ,3)(3,0.1)(4,-.2)(4.5,-.4)(5.5,0.25) (6.5,-.6)  (8,-.9)   (8.5,2.1)(8.8,2.8)(9.3,4)(9.6,6)(3,6.3)(-3,5.8)(-9.6,6)};

\draw[dashed, yshift=+1cm] plot [smooth] coordinates {(-14,.5)(-12,.2)(-10,.5)(-8,1)(-6,.6) (-5,.45)(-4,.3) (-2,-.3) (0,0)(2,.3)(4,-.3)(5,-.45)(6,-.6)(8,-1)(10,-.5)(12,.2)(14,-.5)};
\draw[red, fill=gray=, opacity=.2] (-14,9)--plot [smooth] coordinates {(-14,.5)(-12,.2)(-10,.5)(-8,.98)(-6,.6) (-5,.45)(-4,.3) (-2,-.3) (0,0)(2,.3)(4,-.3)(5,-.45)(6,-.6)(8,-1)(10,-.5)(12,.2)(14,-.5)}--(14,9)--cycle;
\filldraw (0,0) circle (2pt) node[below] {$x_0$};

\draw[red] plot [smooth] coordinates {(-14,.5)(-12,.2)(-10,.5)(-8,.98)(-6,.6) (-4,.3) (-2,-.3) (0,0)(2,.3)(4,-.3)(6,-.6)(8,-1)(10,-.5)(12,.2)(14,-.5)};

\node[left, red] at (-14,.5) {$\partial\Omega$};
\draw[<->] (12,.2)--(12,1.2);
\node at (13,.5) {{\tiny $2^{-j}$}};

\node at (7,5) {$\Omega_{\mathcal{F},Q_0}$};

\node[red] at (0,7.5) {$A_0$};
\node[red] at (11,3) {$A_0$};
\node[red] at (-11,3) {$A_0$};
\node[red] at (-.8,1.8) {$A_0$};
\node[rotate=-10] at (5.1,0) {\tiny $A_0$};
\begin{scope}

\node[rotate=15] at (10,0) {\tiny $A_0$};
\node[rotate=15] at (-10,1) {\tiny $A_0$};


\node[rotate=-15] at (-6,1.1) {{\tiny $A_0$}};
\node at (2,.8) {\tiny $A_0$};
\end{scope}
\node at (3,3) {$A$};
\node at (-3,3) {$A$};

\end{tikzpicture}
\caption{Definition of the matrix $A_1$ in $\Omega$.}\label{figure:def:A1}
\end{figure}
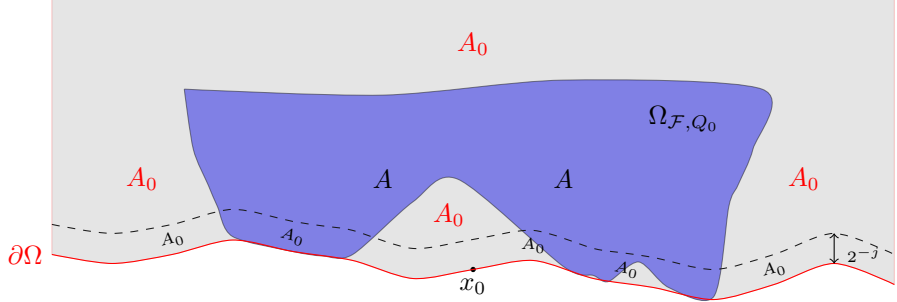

We next fix $Q_{0}^{\star}\in\dd_{Q_0}$ an define $L_1^\star=L_1^{\mathcal{F},Q_0^\star}$ as $L_1^\star u=-\div(A_1^\star\nabla u)$ where
\[
A_1^\star(Y):=\left\{
\begin{array}{ll}
\widetilde{A}(Y) & \hbox{$\text{if }\, Y\in\Omega_{\mathcal{F},Q_{0}^{\star}},$} 
\\[5pt]
A_0(Y) & \hbox{$\text{if }\,Y\in\Omega\setminus \Omega_{\mathcal{F},Q_{0}^{\star}}$.}
\end{array}
\right.
\]
Note that if $Q_0^\star=Q_0$ then $L_1^\star\equiv L_1$. Again $\mathcal{E}_1^\star:=A_1^\star-A_0$ verifies $|\mathcal{E}_1^\star|\leq|\mathcal{E}|\mathbf{1}_{\Omega_{\mathcal{F},Q}}$ and also $\mathcal{E}_1^\star(Y)=0$ if $\delta(Y)<2^{-j}$ (this latter condition will be used qualitatively). Hence if write $\omega_\star^X=\omega_{L_1^\star}^X$ for every $X\in\Omega$ we have that $\omega_\star^X\ll\omega_0^X$ for every $X\in\Omega$ and hence we can write $h(\cdot\,; L_1^\star, L_0, X)=d\omega_\star^X/d\omega_0^X$ which is well-defined $\omega_0^X$-a.e. Also, as shown in \textbf{Step 0} we have $h(\cdot\,;L_1^\star, L_0,X)\in L_{\rm loc}^\infty(\pom, \omega_0^Y)$ for every $X,Y\in\Omega$ (the bound depends on $X$, $Y$ and the fixed $j$ but we will use this qualitatively). 

Set  $X_\star:=X_{c_0^{-1}\Delta_{Q_0^\star}^*}$ which satisfies $2\kappa_0 r_{Q_0^\star}\le \delta(X_\star)<r_0$ since $\ell(Q_0^\star)\le \ell(Q_0)\le \ell(Q^0)\le\frac{c_0}{8\kappa_0}r_0$. Moreover, $X_\star\in\Omega\setminus B_{Q_0^\star}^*$. To simplify the notation set $\omega_\star=\omega_\star^{X_\star}$ and $\omega_0^\star=\omega_0^{X_\star}$.

We have two cases:

\textbf{Case 1:} $Q_0^\star\notin \dd_{\F,Q_0}$, that is, $Q_0^\star\subset Q_j\in \F$ for some $j$. Clearly, $\Omega_{\mathcal{F},Q_{0}^{\star}}=\emptyset$ and hence $L_1^\star\equiv L_0$   in $\Omega$. As a consequence, $\omega_\star^X \equiv \omega_0^X$ for every $X\in \Omega$ and $h(\cdot\,; L_1^\star, L_0, X_\star)\equiv 1$ in $\pom$. In turn we obtain 
\begin{equation}\label{con-Case1}
\|h(\cdot\,; L_1^\star, L_0, X_\star)\|_{L^{q'}(Q_0^\star,\omega_0)}=\omega_0^\star(Q_0^\star)^{\frac1{q'}}.
\end{equation}

\textbf{Case 2:} $Q_0^\star\in \dd_{\F,Q_0}$. In this case it is easy to see that 
\[
\F_\star:=\{Q_j\in\F: Q_j\cap Q_0^\star\neq\emptyset\}= \{Q_j\in\F: Q_j\subset  Q_0^\star\}\subset\dd_{Q_0^\star}.
\]
Thus, $\dd_\F\cap\dd_{Q_0^\star}=\dd_{\F_\star}\cap\dd_{Q_0^\star}$ and $\Omega_{\F, Q_0^\star}=\Omega_{\F_\star, Q_0^\star}$. On the other hand, we set
$\gamma^\star=\{\gamma_Q^\star\}_{Q\in\dd_{Q_0^\star}}$ where
\[
\gamma_Q^\star:=
\omega_{0}^{X_\star}(Q)\sum_{I\in\mathcal{W}_{Q}^{*}}\sup_{Y\in I^*} \|\mathcal{E}\|_{L^\infty(I^*)}^{2},
\qquad \mbox{whenever}\, \, Q\in\mathbb{D}_{Q_0^\star}.
\]
Using \eqref{chop-dyadic:densities} and Harnack's inequality we have that $\omega_{0}^{\star}(Q)\approx \omega_{0}(Q)/\omega_{0}(Q_0^\star)$ for $Q\in\dd_{Q_0^\star}$ where $\omega_0^\star=\omega_{0}^{X_\star}$. Hence, by \eqref{def-gamma:main},
\[
\gamma_Q^\star
\approx
\frac{\omega_{0}(Q)}{\omega_0(Q_0^\star)} \sum_{I\in\mathcal{W}_{Q}^{*}}\sup_{Y\in I^*} \|\mathcal{E}\|_{L^\infty(I^*)}^{2}
=
\frac{\gamma_Q}{\omega_0(Q_0^\star)} , \qquad Q\in\mathbb{D}_{Q_0^\star}.
\]
and, by \eqref{hipepsilon0},
\begin{multline}\label{Car-star-0}
\|\mathbf{m}_{\gamma^\star,\mathcal{F}_{\star}}\|_{\mathcal{C}(Q_0^\star,\omega_0^\star)}
=
\sup_{Q\in \mathbb{D}_{Q_0^\star}} \frac{\mathbf{m}_{\gamma^\star}(\dd_Q\cap\dd_{\F_\star})}{\omega_0^\star(Q)} 
=
\sup_{Q\in \mathbb{D}_{Q_0^\star}} \frac{\mathbf{m}_{\gamma^\star}(\dd_Q\cap\dd_{\F})}{\omega_0^\star(Q)} 
\\
\approx
\sup_{Q\in \mathbb{D}_{Q_0^\star}} \frac{\mathbf{m}_{\gamma}(\dd_{\F,Q})}{\omega_0^\star(Q)\omega_0(Q_0^\star)} 
\approx
\sup_{Q\in \mathbb{D}_{Q_0^\star}} \frac{\mathbf{m}_{\gamma}(\dd_{\F,Q})}{\omega_0(Q)} 
\le
\|\mathbf{m}_{\gamma,\mathcal{F}}\|_{\mathcal{C}(Q_0,\omega_0)}\le \varepsilon_0.
\end{multline}

We next fix $1<q<\infty$ and $0\le g\in L^q(Q_0^\star,\omega_0^\star)$ with $\|g\|_{L^q(Q_0^\star,\omega_0^\star)}=1$. Extend $g$ by 0 in $\pom\setminus Q_0^\star$. Set $g_t= P_t g$ with $0<t<\ell(Q_0^\star)/3$ (see \eqref{defgt}) and note that Lemma \ref{lemma:Pt} gives that $g_t\in \Lip(\partial\Omega)$ with $\supp(g_t)\subset  2\widetilde{\Delta}_{Q_0^\star}$. We then consider
\[
u^t_0(X)=\int_{\partial\Omega} g_t(y) d\omega_{0}^X(y) \quad \mbox{and} \quad  u^t_\star(X)=\int_{\partial\Omega} g_t(y) d\omega_\star^X(y),
\qquad X\in\Omega.
\]
Since $\Omega$ is bounded, we can use Lemma \ref{lemma:u0-u:bounded} (slightly moving $X_{\star}$ if needed). This, Lemma \ref{prop:step1-estimate}, \eqref{Car-star-0}, and Hölder's inequality yield
\begin{align*}
|u^t_\star(X_\star)-u^t_0(X_\star)|
&= 
\left| \iint_{\Omega} (A_0-A_1^\star)^\top(Y) \nabla_{Y} G_{(L_1^\star)^\top}(Y,X_\star)\cdot \nabla u^t_0(Y) dY\right|
\\
&
\leq 
\iint_{\Omega_{\F_{\star},Q_0^\star}} |\mathcal{E}(Y)|\,|\nabla_{Y} G_{(L_1^\star)^\top}(Y,X_\star)| \,|\nabla u^t_0(Y)|\,dY
\\
&
\lesssim
	\|\mut_{\gamma^\star,\F_\star}\|_{\C(Q_0^\star,\omega_0^\star)}^\frac12
\int_{Q_0^\star} M^{\mathbf{d}}_{Q_0^\star,\omega_0^\star}(\omega_1^\star)(x) \mathcal{S}_{Q_0^\star}u_0^t(x) d\omega_0^\star(x)
\\
&
\lesssim
\varepsilon_0^{\frac12} \int_{Q_0^\star} M^{\mathbf{d}}_{Q_0^\star,\omega_0^\star}(\omega_\star)(x) \mathcal{S}_{Q_0^\star}u_0^t(x) d\omega_0^\star(x)
\\
&
\le
\varepsilon_0^{\frac12}\|M^{\mathbf{d}}_{Q_0^\star,\omega_0^\star}(\omega_\star)\|_{L^{q'}(Q_0^\star,\omega_0^\star)}\,\|\mathcal{S}_{Q_0^\star}u_0^t(x)\|_{L^{q}(Q_0^\star,\omega_0^\star)}.
\end{align*}
Using the well-known fact that $M^{\mathbf{d}}_{Q_0^\star,\omega_0^\star}$ is bounded on $L^{q'}(Q_0^\star,\omega_0^\star)$ and that, as mentioned before $\omega_\star\ll\omega_0^\star$ with $h(\cdot\,; L_1^\star, L_0, X_\star)=d\omega_\star/d\omega_0^\star$, it readily follows that 
\[
\|M^{\mathbf{d}}_{Q_0^\star,\omega_0^\star}(\omega_\star)\|_{L^{q'}(Q_0^\star,\omega_0^\star)}\lesssim \|h(\cdot\,; L_1^\star, L_0, X_\star)\|_{L^{q'}(Q_0^\star,\omega_0^\star)}.
\]
On the other hand, using the square-function non-tangential estimates from \cite[Theorem 1.5, Proposition 2.57]{AHMT-I}, Lemma \ref{lemma:Pt}, Remark \ref{remark:chop-dyadic}, and Harnack's inequality to pass from $X_\star$ to $X_{Q_0^\star}$, and the fact that $\supp g \subset Q_0^\star$, yield
\begin{multline*}
\|\mathcal{S}_{Q_0^\star}u_0^t(x)\|_{L^{q}(Q_0^\star,\omega_0^\star)}
\lesssim
\|\mathcal{N}_{Q_0^\star}u_0^t\|_{L^q(Q_0^\star,\omega_0^\star)}
\lesssim
\|g_t\|_{L^q(Q_0^\star,\omega_0^\star)}
\\
\approx
\frac1{\omega_0(Q_0^\star)^{\frac1{q}}}\|g_t\|_{L^q(Q_0^\star,\omega_0)}
\lesssim
\frac1{\omega_0(Q_0^\star)^{\frac1{q}}}\|g\|_{L^q(Q_0^\star,\omega_0)}
\approx
\|g\|_{L^q(Q_0^\star,\omega_0^\star)}
=1.
\end{multline*}
Thus we conclude that $|u^t_\star(X_\star)-u^t_0(X_\star)|
\lesssim
\varepsilon_0^{\frac12} \|h(\cdot\,; L_1^\star, L_0, X_\star)\|_{L^{q'}(Q_0^\star,\omega_0^\star)},
$
hence, using the definitions of $u_0^t$ and $u_\star^t$ we arrive at
\begin{multline*}
\Big|\int_{\partial\Omega} g(y) d\omega_{\star}(y)-\int_{\partial\Omega} g(y) d\omega_{0}^\star(y)\Big|
\le
|u^t_\star(X_\star)-u^t_0(X_\star)|+\|g-g_t\|_{L^1(\pom,\omega_{0}^\star)}+\|g-g_t\|_{L^1(\pom,\omega_\star)}
\\
\lesssim
\varepsilon_0^{\frac12} \|h(\cdot\,; L_1^\star, L_0, X_\star)\|_{L^{q'}(Q_0^\star,\omega_0^\star)}+ \|g-g_t\|_{L^1(\pom,\omega_{0}^\star)}+\|g-g_t\|_{L^1(\pom,\omega_\star)}.
\end{multline*}
Since $g\in L^q(Q_0, \omega_0)$ with $\supp(g)\subset Q_0^\star$, it follows that  $\supp(g), \supp(g_t)\subset 2\widetilde{\Delta}_{Q_0^\star}$. Hence, Lemma \ref{lemma:Pt},
Harnack's inequality and \eqref{chop-dyadic:densities} give
\begin{align*}
\|g-g_t\|_{L^1(\pom,\omega_{0}^\star)}
=
\|g-P_t g\|_{L^1(2\widetilde{\Delta}_{Q_0^\star}, \omega_{0}^\star)}
\approx
\frac1{\omega_0(Q_0^\star)}\|g-P_t g\|_{L^1(2\widetilde{\Delta}_{Q_0^\star}, \omega_{0})}
\longrightarrow  0, \quad\text{as $t\to 0^+$.}
\end{align*}
Similarly, using that as mentioned above $\omega_\star\ll\omega_0$ with $h(\cdot\,;L_1^\star, L_0,X_\star)\in L_{\rm loc}^\infty(\pom, \omega_0)$
\begin{multline*}
\|g-g_t\|_{L^1(\pom,\omega_\star)}
=
\|g-P_t g\|_{L^1(2\widetilde{\Delta}_{Q_0^\star}, \omega_\star)}
\\
\le
\|h(\cdot\,;L_1^\star, L_0,X_\star)\|_{L^\infty(2\widetilde{\Delta}_{Q_0^\star},\omega_0^\star)}
\|g-P_t g\|_{L^1(2\widetilde{\Delta}_{Q_0^\star}, \omega_{0}^\star)}
\longrightarrow  0,
\qquad\mbox{as }t\to 0^+.
\end{multline*}
Combining the previous estimates and letting $t\to 0^+$ we conclude that
\begin{multline*}
0
\le 
\int_{Q_0^\star} h(y; L_1^\star, L_0, X_\star)\,g(y) d\omega_{0}^\star(y)
=
\int_{\partial\Omega} h(y; L_1^\star, L_0, X_\star)\,g(y) d\omega_{0}^\star(y)
\\
=
\int_{\partial\Omega} g(y) d\omega_\star(y)
\lesssim
\varepsilon_0^{\frac12} \|h(\cdot\,; L_1^\star, L_0, X_\star)\|_{L^{q'}(Q_0^\star,\omega_0)}
+
\int_{\partial\Omega} g(y) d\omega_0^\star(y)
\\
\le
\varepsilon_0^{\frac12} \|h(\cdot\,; L_1^\star, L_0, X_\star)\|_{L^{q'}(Q_0^\star,\omega_0^\star)}
+
\omega_0^\star(Q_0^\star)^{\frac1{q'}}.
\end{multline*}
Taking now the sup over all  $0\le g\in L^q(Q_0^\star,\omega_0^\star)$ with $\|g\|_{L^q(Q_0^\star,\omega_0^\star)}=1$ we eventually get
\begin{equation}\label{hiding-eqn}
\|h(\cdot\,; L_1^\star, L_0, X_\star)\|_{L^{q'}(Q_0^\star,\omega_0^\star)}
\lesssim
\varepsilon_0^{\frac12} \|h(\cdot\,; L_1^\star, L_0, X_\star)\|_{L^{q'}(Q_0^\star,\omega_0^\star)}+ \omega_0^\star(Q_0^\star)^{\frac1{q'}}.
\end{equation}
Since $h(\cdot\,;L_1, L_0^\star,X_\star)\in L_{\rm loc}^\infty(\pom, \omega_0^\star)$ (albeit with bounds which may depend on $X_\star$ or $j$) we can hide the first term on the right hand side and eventually obtain fixing $\varepsilon_0$ small enough (depending on $n$, the 1-sided NTA constants, the CDC constant, the ellipticity constants of $L_0$ and $L_2$, and on $q$), 
\begin{equation}\label{con-Case2}
\|h(\cdot\,; L_1^\star, L_0, X_\star)\|_{L^{q'}(Q_0^\star,\omega_0^\star)}\lesssim \omega_0^\star(Q_0^\star)^{\frac1{q'}}.
\end{equation}

Note then that by \eqref{con-Case1} we conclude that \eqref{con-Case2} holds for any $Q_0^\star\in \dd_{Q_0}$. On the other hand,
using \cite[Lemma 3.55]{HM1}  (which holds as well in our scenario), there exists $0<\widehat{\kappa}_1<\kappa_1$ (see \eqref{definicionkappa12}), depending only on the allowable parameters, such that $\widehat{\kappa}_1 B_{Q_0^\star} \cap \Omega_{\mathcal{F},Q_0} =  \widehat{\kappa}_1 B_{Q_0^\star} \cap \Omega_{\mathcal{F},Q_0^\star}$, Hence $L_1^{\star}\equiv L_1$ in  $\widehat{\kappa}_1 B_{Q_0^\star} \cap \Omega$ which, by Lemma \ref{lemma:proppde} part $(f)$ and Harnack's inequality, gives that $\omega_{\star}$ and $\omega_0^\star$ are comparable in $\eta \Delta_{Q_0^\star}$ with $\eta=\widehat{\kappa}_1/(2\kappa_0)$, thus $h(\cdot\,; L_1^\star, L_0, X_\star)\approx h(\cdot\,; L_1, L_0, X_\star)$ for $\omega_0^\star$-a.e. in $\eta \Delta_{Q_0^\star}$ (hence, also $\omega_0$-a.e.). This, Remark \ref{remark:chop-dyadic}, Harnack's inequality, and Lemma \ref{lemma:proppde} part $(c)$ yield
\begin{multline*}
h(\cdot\,; L_1, L_0, X_0)
=
\frac{d \omega_{L_1}^{X_0}}{d \omega_{L_0}^{X_0}}
=
\frac{d \omega_{L_1}^{X_0}}{d \omega_{L_1}^{X_\star}}\frac{d \omega_{L_1}^{X_\star}}{d \omega_{L_0}^{X_\star}}\frac{d \omega_{L_0}^{X_\star}}{d \omega_{L_0}^{X_0}}
\approx
\frac{\omega_1(Q_0^\star)}{\omega_0(Q_0^\star)} h(\cdot\,; L_1, L_0, X_\star)
\\
\approx
\frac{\omega_1(\eta \Delta_{Q_0^\star})}{\omega_0(\eta \Delta_{Q_0^\star})} h(\cdot\,; L_1^\star, L_0, X_\star),
\end{multline*}
and these hold $\omega_0$-a.e.~in $\eta \Delta_{Q_0^\star}$ an $\forall\, Q_0^\star\in \dd_{Q_0}$ (recall that $\omega_{1}$ and $\omega_{0}$ are mutually absolutely continuous). Eventually, \eqref{con-Case2}, Remark \ref{remark:chop-dyadic} and Harnack's inequality allow us to conclude that for all $Q_0^\star\in \dd_{Q_0}$
\begin{multline}\label{h-eta}
\left(\aver{\eta \Delta_{Q_0^\star}} h(y; L_1, L_0, X_0)^{q'}d\omega_0(y)\right)^{\frac1{q'}}
\approx
\frac{\omega_1(\eta \Delta_{Q_0^\star})}{\omega_0(\eta \Delta_{Q_0^\star})}
\left(\aver{\eta \Delta_{Q_0^\star}} h(y; L_1^\star, L_0, X_\star)^{q'}d\omega_0^\star(y)\right)^{\frac1{q'}}
\\
\lesssim
\frac{\omega_1(\eta \Delta_{Q_0^\star})}{\omega_0(\eta \Delta_{Q_0^\star})}
=
\aver{\eta \Delta_{Q_0^\star}} h(y; L_1, L_0, X_0)d\omega_0(y).
\end{multline}

Our next goal is to show that the latter implies that $\omega_1\in A_\infty^{\rm dyadic}(Q_0,\omega_0)$ and to show that we use an argument similar to \cite[Lemma 3.1]{CHM}. Let $Q\in \dd_{Q_0}$ and a Borel set $F\subset \eta \Delta_{Q}$ and note that by \eqref{h-eta} applied to $Q$
\begin{multline*}
\frac{\omega_1(F)}{\omega_0(\eta \Delta_{Q})}
=
\aver{\eta \Delta_{Q}} \mathbf{1}_F(y)h(y; L_1, L_0, X_0) d\omega_0
\\
\le
\left(\frac{\omega_0(F)}{\omega_0(\eta \Delta_{Q})}\right)^{\frac1q} 
\left(\aver{\eta \Delta_{Q}} h(y; L_1, L_0, X_0)^{q'}d\omega_0(y)\right)^{\frac1{q'}}
=
C_1\left(\frac{\omega_0(F)}{\omega_0(\eta \Delta_{Q})}\right)^{\frac1q} 
\frac{\omega_1(\eta \Delta_{Q})}{\omega_0(\eta \Delta_{Q})},
\end{multline*}
hence
\begin{equation}\label{fawfwe}
\frac{\omega_1(F)}{\omega_1(\eta \Delta_{Q})}
\le
C_1\left(\frac{\omega_0(F)}{\omega_0(\eta \Delta_{Q})}\right)^{\frac1q},
\qquad  \forall\, F\subset \eta \Delta_{Q},\ Q\in \dd_{Q_0}. 
\end{equation}

On the other hand, by Lemma \ref{lemma:proppde} part $(c)$, $\omega_0(Q)\leq C_2 \omega_0(\eta\Delta_Q)$ for all $Q\in\dd_{Q_0}$. Fix then $\alpha$, 
$0<\alpha<(C_2C_1^{q})^{-1}$, and take $F\subset Q$ such that
$\omega_0(F)>(1-\alpha)\omega_0(Q)$. Writing $F_0=\eta\Delta_Q\cap F$ and $F_1=\eta\Delta_Q\setminus F$, it is clear that
\begin{align*}
(1-\alpha)\frac{\omega_0(Q)}{\omega_0(\eta\Delta_Q)}
<
\frac{\omega_0(F)}{{\omega_0(\eta\Delta_Q)}}
\leq
\frac{\omega_0(F_0)}{\omega_0(\eta\Delta_Q)}
+
\frac{\omega_0(Q\setminus\eta\Delta_Q)}{\omega_0(\eta\Delta_Q)}
=
\frac{\omega_0(F_0)}{\omega_0(\eta\Delta_Q)}+\frac{\omega_0(Q)}{\omega_0(\eta\Delta_Q)}-1.
\end{align*}
As a result,
\begin{equation}\label{alsowehave}
\frac{\omega_0(F_1)}{\omega_0(\eta\Delta_Q)}=
1-\frac{\omega_0(F_0)}{\omega_0(\eta\Delta_Q)}
<
\alpha\,\frac{\omega_0(Q)}{\omega_0(\eta\Delta_Q)}
\le
C_2\,\alpha.
\end{equation}
Combining \eqref{fawfwe} and \eqref{alsowehave}  applied to $F_1$ we obtain $\omega_1(F_1)/\omega_1(\eta\Delta_Q)< C_1\big(C_2\alpha\big)^{\frac1q}$. This and the fact that $\omega_1(Q)\leq C_3 \omega_1(\eta\Delta_Q)$, by Lemma \ref{lemma:proppde} part $(c)$, yield
\begin{align*}
\frac{\omega_1(F)}{\omega_1(Q)}
\geq
\frac{\omega_1(\eta\Delta_Q)}{\omega_1(Q)}\,\frac{\omega_1(F_0)}{\omega_1(\eta\Delta_Q)}
\ge
\frac1{C_3}\left(1-\frac{\omega_1(F_1)}{\omega_1(\eta\Delta_Q)}\right)
>
\frac{1-C_1(C_2\alpha)^{\frac1q}}{C_3}=:1-\beta,
\end{align*}
with $0<\beta<1$ by our choice of $\alpha$. This eventually shows that $\omega_1\in A_\infty^{\rm dyadic}(Q_0,\omega_0)$ (see Definition \eqref{def:Ainfty-dyadic}) as desired. This with the help of Lemma \ref{lemm_w-Pw-:properties} allows us to obtain that $\P_\F^{\omega_0} \omega_1\in A_\infty^{\rm dyadic}(Q_0,\omega_0)$, which is the conclusion of \textbf{Step 1}.

\subsubsection{Step 2} We next define a new operator $L_2 u=-\div(A_2 \nabla u)$  where (see Figure \ref{figure:A2}): 
\[
A_2(Y):=
\left\{
\begin{array}{ll}
\widetilde{A}(Y) &\mbox{if}\, \, Y\in T_{Q_0}\setminus \Omega_{\mathcal{F},Q_0},
\\[5pt]
A_1(Y) &\mbox{if}\, \, Y\in\Omega\setminus (T_{Q_0}\setminus \Omega_{\mathcal{F},Q_0}). 
\end{array}
\right.
\]

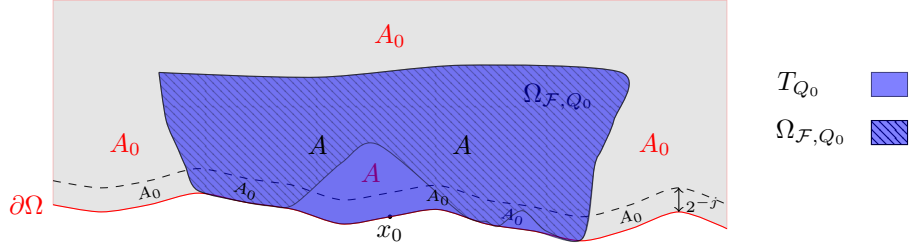
\begin{figure}[!ht]
\centering
\begin{tikzpicture}[scale=.32]

\draw[dashed, yshift=+1cm] plot [smooth] coordinates {(-14,.5)(-12,.2)(-10,.5)(-8,1)(-6,.6) (-5,.45)(-4,.3) (-2,-.3) (0,0)(2,.3)(4,-.3)(5,-.45)(6,-.6)(8,-1)(10,-.5)(12,.2)(14,-.5)};
\draw[red, fill=gray=, opacity=.2] (-14,9)--plot [smooth] coordinates {(-14,.5)(-12,.2)(-10,.5)(-8,.98)(-6,.6) (-5,.45)(-4,.3) (-2,-.3) (0,0)(2,.3)(4,-.3)(5,-.45)(6,-.6)(8,-1)(10,-.5)(12,.2)(14,-.5)}--(14,9)--cycle;
\filldraw (0,0) circle (2pt) node[below] {$x_0$};

\draw[red] plot [smooth] coordinates {(-14,.5)(-12,.2)(-10,.5)(-8,.98)(-6,.6) (-4,.3) (-2,-.3) (0,0)(2,.3)(4,-.3)(6,-.6)(8,-1)(10,-.5)(12,.2)(14,-.5)};

\node[left, red] at (-14,.5) {$\partial\Omega$};
\draw[<->] (12,.2)--(12,1.2);
\node at (13,.5) {{\tiny $2^{-j}$}};

\node at (7,5) {$\Omega_{\mathcal{F},Q_0}$};

\node[red] at (0,7.5) {$A_0$};
\node[red] at (11,3) {$A_0$};
\node[red] at (-11,3) {$A_0$};
\node[red] at (-.8,1.8) {$A$};
\node[rotate=-10] at (5.1,0) {\tiny $A_0$};
\begin{scope}

\node[rotate=15] at (10,0) {\tiny $A_0$};
\node[rotate=15] at (-10,1) {\tiny $A_0$};

\draw[fill=blue, opacity=.5] plot [smooth] coordinates {(-9.6,6)(-9.3,4)(-8.8,2.8)(-8.5,2.1)(-8,1.1)(-6,.6) (-4,.3)(-2,-.3) (0,0)(2,.3)(4,-.3)(6,-.6)(8,-.9)(8.5,2.1)(8.8,2.8)(9.3,4)(9.6,6)(3,6.3)(-3,5.8)(-9.6,6)};
\draw[fill=blue, opacity=.5, pattern=north west lines] plot [smooth] coordinates {(-9.6,6)(-9.3,4)(-8.8,2.8)(-8.5,2.1)   (-8,1.1)(-6,.6)(-5,.45) (-4,.5)(-2,2.25)  (-.5 ,3)(3,0.1)(4,-.2)(4.5,-.4)(5.5,0.25) (6.5,-.6)  (8,-.9)   (8.5,2.1)(8.8,2.8)(9.3,4)(9.6,6)(3,6.3)(-3,5.8)(-9.6,6)};


\draw[fill=blue, opacity=.5] (20,5) rectangle (21.5,6);
\node at (17,5.5) {$T_{Q_0}$};

\filldraw[pattern=north west lines] (20,3) rectangle (21.5,4);
\filldraw[fill=blue, opacity=.5] (20,3) rectangle (21.5,4);
\node at (17.5,3.5) {$\Omega_{\mathcal{F},Q_0}$};

\node[rotate=-15] at (-6,1.1) {{\tiny $A_0$}};
\node at (2,.8) {\tiny $A_0$};
\end{scope}
\node at (3,3) {$A$};
\node at (-3,3) {$A$};

\end{tikzpicture}
\caption{Definition of matrix $A_2$ in $\Omega$.} \label{figure:A2}
\end{figure}

The goal of this step is to show that $\P_\F^{\omega_0}\omega_2\in A^{\rm dyadic}_{\infty}(Q_0, \omega_0)$, where much as before let $\omega_2=\omega_{L_2}^{X_0}$.

We apply Lemma \ref{prop:Pi-proj} to obtain $Y_{Q_0}\in\Omega\cap\Omega_{\F,Q_0}$ satisfying \eqref{eq:common-cks}. For $k=1,2$ we write $\omega_k^{Y_{Q_0}}=\omega^{Y_{Q_0}}_{L_k, \Omega}$ a for the elliptic measures associated with $L_k$ for the domain $\Omega$ and with pole at $Y_{Q_0}$. Likewise, let $\omega_{k,*}^{Y_{Q_0}}=\omega^{Y_{Q_0}}_{L_k,\Omega_{\mathcal{F},Q_0}}$ be the elliptic measures associated with $L_k$  for the domain $\Omega_{\mathcal{F},Q_0}$ and with pole at $Y_{Q_0}$. 
By definition $A_2=\tilde{A}$ in $T_{Q_0}$, $A_2=A_0$ in $\Omega\setminus T_{Q_0}$, and $A_2=A_1$ in $\Omega_{\F,Q_0}$.  Hence $L_2\equiv L_1$   in $\Omega_{\F,Q_0}$, and thus
$\omega_{2,*}^{Y_{Q_0}}\equiv\omega_{1,*}^{Y_{Q_0}}$. If we now consider the associated measures $\nu_{L_1}^{Y_{Q_0}}$ and $\nu_{L_2}^{Y_{Q_0}}$ in \eqref{eq:def-nu} from Lemma \ref{lemma:DJK-sawtooth} it follows from \eqref{eq:def-nu:P} (with $\mu=\omega_0$ which is clearly (dyadically) doubling in $Q_0$ by Lemma \ref{lemma:proppde} part $(c)$) that $\P_\F^{\omega_0}\nu_{L_1}^{Y_{Q_0}}=\P_\F^{\omega_0}\nu_{L_2}^{Y_{Q_0}}$ as measures on $Q_0$.

In \textbf{Step 1} we showed that $\P_\F^{\omega_0} \omega_1\in A_\infty^{\rm dyadic}(Q_0,\omega_0)$, then there is $1<\widetilde{q}<\infty$ such that $\P_\F^{\omega_0} \omega_1\in RH_{\widetilde{q}}^{\rm dyadic}(Q_0,\omega_0)$. Note that by Remark \ref{remark:chop-dyadic} and Harnack's inequality we have that
$\P_\F^{\omega_0} \omega_k^{Y_{Q_0}}\approx \P_\F^{\omega_0} \omega_k/\omega_1(Q_0)$ for $k=1,2$. Then given $Q\in\dd_{Q_0}$ and a Borel set $F\subset Q$ we have that all these yield
\begin{multline*}
\frac{\P_\F^{\omega_0}\omega_2(F)}{\P_\F^{\omega_0}\omega_2(Q)}
\approx
\frac{\P_\F^{\omega_0}\omega_{L_2}^{Y_{Q_0}}(F)}{\P_\F^{\omega_0}\omega_{L_2}^{Y_{Q_0}}(Q)}
\lesssim
\left(\frac{\P_\F^{\omega_0}\nu_{L_2}^{Y_{Q_0}}(F)}{\P_\F^{\omega_0}\nu_{L_2}^{Y_{Q_0}}(Q)}\right)^{\frac1{\theta_2}}
=
\left(\frac{\P_\F^{\omega_0}\nu_{L_1}^{Y_{Q_0}}(F)}{\P_\F^{\omega_0}\nu_{L_1}^{Y_{Q_0}}(Q)}\right)^{\frac1{\theta_2}}
\\
\lesssim
\left(\frac{\P_\F^{\omega_0}\omega_{L_1}^{Y_{Q_0}}(F)}{\P_\F^{\omega_0}\omega_{L_1}^{Y_{Q_0}}(Q)}\right)^{\frac1{\theta_2}}
\lesssim
\left(
\frac{\P_\F^{\omega_0}\omega_{1}(F)}{\P_\F^{\omega_0}\omega_{1}(Q)}\right)^{\frac1{\theta_2}}
\lesssim
\left(
\frac{\omega_0(F)}{\omega_0(Q)}\right)^{\frac1{\theta_2\widetilde{q}'}}
\end{multline*}
where in the second and third estimates we have invoked Lemma \ref{lemma:DJK-sawtooth} respectively for $L_2$ (with parameter $\theta_2$) and $L_1$, and the last estimate follows easily from the fact that $\P_\F^{\omega_0} \omega_1\in RH_{\widetilde{q}}^{\rm dyadic}(Q_0,\omega_0)$ and Hölder's inequality. This, the fact that $\P_\F^{\omega_0}\omega_2$ is dyadic doubling in $Q_0$ by Lemma \ref{lemm_w-Pw-:properties} part $(a)$ since $\omega_2$ is indeed doubling in $4\widetilde{\Delta}_{Q_0}$ by Lemma \ref{lemma:proppde} part $(c)$, and \cite[Lemma B.7]{HM1} (which is a purely dyadic result and hence applies in our setting) gives that there exists $\theta, \theta'>0$ such that
\begin{align}
\label{e:weakAinftyforProj}
\left(\frac{\omega_0(F)}{\omega_0(Q)}\right)^{\theta} 
\lesssim 
\frac{\P_\F^{\omega_0}\omega_2(F)}{\P_\F^{\omega_0}\omega_2(Q)}
\lesssim 
\left(\frac{\omega_0(F)}{\omega_0(Q)}\right)^{\theta'},
\qquad\forall\,F\subset Q,\  Q\in \mathbb{D}_{Q_0}.
\end{align}

\subsubsection{Step 3} In this part, we change the operator outside of $T_{Q_0}$ to complete the process. To this end, let $L_3u=-\div(A_3\nabla u)$, where
$$
A_3(Y)
:=\left\{
\begin{array}{ll}
A_2(Y) & \hbox{$\text{if }\,Y\in T_{Q_0}$,} 
\\[5pt]
\widetilde{A}(Y) & \hbox{$\text{if }\,Y\in \Omega\setminus T_{Q_0},$}
\end{array}
\right.
$$
and note that $L_3\equiv\widetilde{L}$ in $\Omega$ (see Figure \ref{figure:L3}). Let $w_3^{X_{0}}:=\omega^{X_{0}}_{L_3}$ be the elliptic measure of $\Omega$ associated with the operator $L_3\equiv\widetilde{L}$ with pole at $X_0$.

\begin{figure}[!ht]
\centering
\begin{tikzpicture}[scale=.32]

\draw[dashed, yshift=+1cm] plot [smooth] coordinates {(-14,.5)(-12,.2)(-10,.5)(-8,1)(-6,.6) (-5,.45)(-4,.3) (-2,-.3) (0,0)(2,.3)(4,-.3)(5,-.45)(6,-.6)(8,-1)(10,-.5)(12,.2)(14,-.5)};
\draw[red, fill=gray=, opacity=.2] (-14,9)--plot [smooth] coordinates {(-14,.5)(-12,.2)(-10,.5)(-8,.98)(-6,.6) (-5,.45)(-4,.3) (-2,-.3) (0,0)(2,.3)(4,-.3)(5,-.45)(6,-.6)(8,-1)(10,-.5)(12,.2)(14,-.5)}--(14,9)--cycle;
\filldraw (0,0) circle (2pt) node[below] {$x_0$};

\draw[red] plot [smooth] coordinates {(-14,.5)(-12,.2)(-10,.5)(-8,.98)(-6,.6) (-4,.3) (-2,-.3) (0,0)(2,.3)(4,-.3)(6,-.6)(8,-1)(10,-.5)(12,.2)(14,-.5)};

\node[left, red] at (-14,.5) {$\partial\Omega$};
\draw[<->] (12,.2)--(12,1.2);
\node at (13,.5) {{\tiny $2^{-j}$}};

\node at (7,5) {$\Omega_{\mathcal{F},Q_0}$};

\node[red] at (0,7.5) {$A$};
\node[red] at (11,3) {$A$};
\node[red] at (-11,3) {$A$};
\node[red] at (-.8,1.8) {$A$};
\node[rotate=-10] at (5.1,0) {\tiny $A_0$};
\begin{scope}

\node[rotate=15] at (10,0) {\tiny $A_0$};
\node[rotate=15] at (-10,1) {\tiny $A_0$};

\draw[fill=blue, opacity=.5] plot [smooth] coordinates {(-9.6,6)(-9.3,4)(-8.8,2.8)(-8.5,2.1)(-8,1.1)(-6,.6) (-4,.3)(-2,-.3) (0,0)(2,.3)(4,-.3)(6,-.6)(8,-.9)(8.5,2.1)(8.8,2.8)(9.3,4)(9.6,6)(3,6.3)(-3,5.8)(-9.6,6)};
\draw[fill=blue, opacity=.5, pattern=north west lines] plot [smooth] coordinates {(-9.6,6)(-9.3,4)(-8.8,2.8)(-8.5,2.1)   (-8,1.1)(-6,.6)(-5,.45) (-4,.5)(-2,2.25)  (-.5 ,3)(3,0.1)(4,-.2)(4.5,-.4)(5.5,0.25) (6.5,-.6)  (8,-.9)   (8.5,2.1)(8.8,2.8)(9.3,4)(9.6,6)(3,6.3)(-3,5.8)(-9.6,6)};


\node[rotate=-15] at (-6,1.1) {{\tiny $A_0$}};
\node at (2,.8) {\tiny $A_0$};
\end{scope}
\node at (3,3) {$A$};
\node at (-3,3) {$A$};

\draw[fill=blue, opacity=.5] (20,5) rectangle (21.5,6);
\node at (17,5.5) {$T_{Q_0}$};

\filldraw[pattern=north west lines] (20,3) rectangle (21.5,4);
\filldraw[fill=blue, opacity=.5] (20,3) rectangle (21.5,4);
\node at (17.5,3.5) {$\Omega_{\mathcal{F},Q_0}$};

\end{tikzpicture}
\caption{Definition of the matrix $A_3$ in $\Omega$.\label{figure:L3}}
\end{figure}
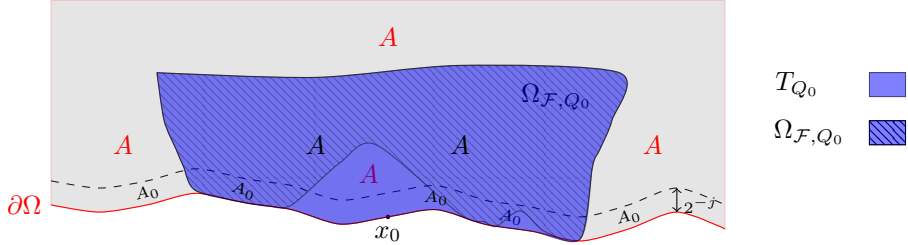

In this step we are going to need the following property: if $\tau>0$ is small enough,  there exists $C_{\tau}>1$ such that
\begin{equation}\label{est:L2-L3'}
C_{\tau}^{-1}
\frac{\omega_3(E)}{\omega_3(Q_0)} 
\le 
\frac{\omega_2(E)}{\omega_2(Q_0)}
\le
C_{\tau}
\frac{\omega_3(E)}{\omega_3(Q_0)},
\qquad\forall\, E\subset Q_0\setminus\Sigma_{\tau},
\end{equation}
where $\Sigma_{\tau}:=\big\{x\in Q_0:\:\dist(x,\partial\Omega\setminus Q_0)<\tau\ell(Q_0)\big\}$.

Assuming this momentarily, our final goal is to prove that for every $\zeta$, $0<\zeta<1$, there exists $C_{\zeta}>1$ such that 
\begin{align}
\label{e:AinftyonQ0}
F\subset Q_0, \quad\frac{\omega_0(F)}{\omega_0(Q_0)} \geq \zeta 
\quad \Longrightarrow \quad   
\frac{\mathcal{P}^{\omega_0}_{\mathcal{F}} \omega_3(F)}{\mathcal{P}^{\omega_0}_{\mathcal{F}} \omega_3(Q_0)} \geq \frac{1}{C_\zeta}.
\end{align}
Fix then $\zeta\in (0,1)$, and $F\subset Q_0$  with $\omega_0(F)\ge\zeta\omega_0(Q_0)$. Consider first the case on which $\F=\{Q_0\}$, in which case
\[
\frac{\mathcal{P}^{\omega_0}_{\mathcal{F}} \omega_3(F)}{\mathcal{P}^{\omega_0}_{\mathcal{F}} \omega_3(Q_0)}
=
\frac{\frac{\omega_0(F)}{\omega_0(Q_0)}\omega_3(Q_0)}{\frac{\omega_0(Q_0)}{\omega_0(Q_0)}\omega_3(Q_0)}
=
\frac{\omega_0(F)}{\omega_0(Q_0)}
\ge 
\zeta,
\]
which is the desired estimate with $C_\zeta=\zeta$. Thus we may assume that $\F\subset\dd_{Q_0}\setminus\{Q_0\}$. Let $\tau\ll 1$ small enough to be chosen and 
let $Q_0^\tau:=Q_0\setminus\bigcup_{Q'\in\mathcal{I}_\tau}Q'$, where
$$
\mathcal{I}_{\tau}=\big\{Q'\in\mathbb{D}_{Q_0}:\:\tau\ell(Q_0)<\ell(Q')\leq 2\tau\ell(Q_0),\ Q'\cap\Sigma_{\tau}\neq\emptyset\big\}.
$$
By construction, $\Sigma_{\tau}\subset \bigcup_{Q'\in\mathcal{I}_\tau}Q'$, and by \eqref{deltaQ} every $Q'\in\mathcal{I}_\tau$ satisfies $Q'\subset\Sigma_{(1+4\Xi)\tau}$.  Using Lemma \ref{lemma:dyadiccubes} (see \cite[Remark~2.19]{AHMT-I}), along with the fact that $\omega_0$ is doubling in $4\Delta_0$ with a constant which does not depend on $\Delta_0$ (see Lemma \ref{lemma:proppde} part $(c)$), if $\tau=\tau(\zeta)>0$ is sufficiently small then
\[
\omega_0(Q_0\setminus Q_0^\tau)\leq\omega_0(\Sigma_{(1+4\Xi)\tau})\lesssim \tau^\eta\omega_0(Q_0)\leq\frac{\zeta}{2}\,\omega_0(Q_0).
\]
Letting $F'=F\cap Q_0^\tau $, it follows that
$$
\zeta \omega_0(Q_0)\leq\omega_0(F)
\leq
\omega_0(F') 
+ 
\omega_0(Q_0\setminus Q_0^\tau )
\leq
\omega_0(F')+\frac{\zeta}{2}\omega_0(Q_0).
$$
Hence $\omega_0(F')/\omega_0(Q_0)\geq\zeta/2$ and by \eqref{e:weakAinftyforProj}, we conclude that
\begin{equation}\label{utpasant}
\frac{\mathcal{P}_\mathcal{F}^{\omega_0}\omega_2(F')}{\mathcal{P}_\mathcal{F}^{\omega_0}\omega_2(Q_0)}
\gtrsim
\bigg(\frac{\omega_0(F')}{\omega_0(Q_0)}\bigg)^\theta
\geq
\Big(\frac{\zeta}{2}\Big)^\theta.
\end{equation}

Our next goal is to show that there is $c_\zeta>0$ such that $\mathcal{P}_\mathcal{F}^{\omega_0}\omega_3(F')
\geq c_{\zeta} \mathcal{P}_\mathcal{F}^{\omega_0}\omega_{2}(F')$. To see this let $Q_k\in\mathcal{F}$ be such that $F'\cap Q_k\neq\emptyset$. We consider two cases. If
$Q_k\subset Q_0^\tau $, we can invoke \eqref{est:L2-L3'} since $Q_0^\tau\subset Q_0\setminus\Sigma_\tau$, to conclude that
\begin{equation}\label{est:Qk:inside}
\frac{\omega_2(Q_k)}{\omega_2(Q_0)}
\approx_\tau
\frac{\omega_3(Q_k)}{\omega_3(Q_0)}.
\end{equation}
Otherwise, $Q_k\setminus Q_0^\tau \neq\emptyset$, and there exists $Q'\in\mathcal{I}_\tau$ such that $Q_k\cap Q'\neq\emptyset$. Then necessarily $Q'\subsetneq Q_k$ ---if $Q_k\subset Q'$ then  $Q_k\subset Q_0\setminus Q_0^\tau $, contradicting that $F'\cap Q_k\neq\emptyset$ and $F'\subset Q_0^\tau $--- and, in particular, $\ell(Q_k)>\tau\ell(Q_0)$. Take $\widehat{Q}_k\in\mathbb{D}_{Q_k}$ with $x_{Q_k}\in\widehat{Q}_k$, $\ell(\widehat{Q}_k)=2^{-M}\ell(Q_k)$ and $M>1$ to be chosen. Note that
$\diam(\widehat{Q}_k)\approx 2^{-M}\ell(Q_k)$ (see Remark \ref{remark:diam-radius}) and clearly
\begin{multline*}
\ell(Q_k)\approx r_{Q_k}\leq\dist(x_{Q_k},\partial\Omega\setminus\Delta_{Q_k})
\leq
\diam(\widehat{Q}_k)+\dist(\widehat{Q}_k,\partial\Omega\setminus\Delta_{Q_k})
\\
\approx 2^{-M}\ell(Q_k)+\dist(\widehat{Q}_k,\partial\Omega\setminus\Delta_{Q_k}).
\end{multline*}
Taking $M\gg 1$ large enough, we conclude that
\[
c\tau\ell(Q_0)
<
c\ell(Q_k)
\le \dist(\widehat{Q}_k,\partial\Omega\setminus\Delta_{Q_k})\le \dist(\widehat{Q}_k, \partial\Omega\setminus Q_0)
\]
 and hence $\widehat{Q}_k\subset Q_0\setminus\Sigma_{c\tau}$. Using again \eqref{est:L2-L3'} (with $c\tau$ in place of $\tau$) and Lemma \ref{lemma:proppde} part $(c)$ we obtain
\begin{equation}\label{est:Qk:not-inside}
\frac{\omega_3(Q_k)}{\omega_3(Q_0)}
\geq
\frac{\omega_3(\widehat{Q}_k)}{\omega_3(Q_0)}
\approx_{\tau}
\frac{\omega_{2} (\widehat{Q}_k)}{\omega_2(Q_0)}
\gtrsim
\frac{\omega_2(Q_k)}{\omega_2(Q_0)}.
\end{equation}
Combining \eqref{est:Qk:inside}, \eqref{est:Qk:not-inside} and invoking \eqref{est:L2-L3'}, since $F'\subset Q_0^\tau\subset Q_0\setminus\Sigma_\tau$, we conclude that
\begin{multline*}
\frac{\mathcal{P}_\mathcal{F}^{\omega_0}\omega_3(F)}{\mathcal{P}_\mathcal{F}^{\omega_0}\omega_3(Q_0)} 
\ge
\frac{\mathcal{P}_\mathcal{F}^{\omega_0}\omega_3(F')}{\mathcal{P}_\mathcal{F}^{\omega_0}\omega_3(Q_0)} 
=
\frac{\omega_3(F'\setminus\bigcup_{Q_k\in\mathcal{F}}Q_k)}{\omega_3(Q_0)}
+
\sum_{Q_k\in\F} \frac{\omega_0(Q_k\cap F')}{\omega_0(Q_k)} \frac{\omega_3(Q_k)}{\omega_3(Q_0)}
\\
\gtrsim_{\zeta}
\frac{\omega_2(F'\setminus\bigcup_{Q_k\in\mathcal{F}}Q_k)}{\omega_2(Q_0)}
+
\sum_{Q_k\in\F} \frac{\omega_0(Q_k\cap F')}{\omega_0(Q_k)} \frac{\omega_2(Q_k)}{\omega_2(Q_0)}
=
\frac{\mathcal{P}_\mathcal{F}^{\omega_0}\omega_2(F')}{\mathcal{P}_\mathcal{F}^{\omega_0}\omega_2(Q_0)} 
\gtrsim
\Big(\frac{\zeta}{2}\Big)^\theta,
\end{multline*}
where we have used that $\tau=\tau(\zeta)$, that $\mathcal{P}_\mathcal{F}^{\omega_0}\omega_i (Q_0)=\omega_i (Q_0)$  for $i=2,3$, and  the last estimate follows from \eqref{utpasant}. This eventually proves \eqref{e:AinftyonQ0} in the present case and it remains to establish our claim \eqref{est:L2-L3'}.

To show \eqref{est:L2-L3'} write $r=\tau\ell(Q_0)/(8\kappa_0)$  (see \eqref{definicionkappa0}) and find a maximal collection of points $\{x_k\}_{k\in\mathcal{K}}\subset Q_0\setminus\Sigma_{\tau}$ with respect to the property that $|x_k-x_{k'}|>2 r/3$ for every $k,k'\in\mathcal{K}$ with $k\neq k'$. Write 
$\Delta_k=\Delta(x_k,r)$ and observe that  $\{\frac13\Delta_k\}_{k\in\mathcal{K}}$ is a family of pairwise disjoint surface balls such that  $Q_0\setminus\Sigma_{\tau}\subset\bigcup_{k\in\mathcal{K}}\Delta_k$. Note that by \eqref{deltaQ}, we have
$
\frac13\Delta_k
\subset
2\widetilde{\Delta}_{Q_0}
\subset
\Delta(x_k, 3\Xi\ell(Q_0)),
$
for every $k\in \mathcal{K}$, hence Lemma \ref{lemma:proppde} part $(c)$ yields
\[
\#\mathcal{K}
C_{\tau}^{-1} \omega_0(2\widetilde{\Delta}_{Q_0})
\le
\sum_{k\in\mathcal{K}}\omega_0(\tfrac13\Delta_k)
=
\omega_0 \Big(\bigcup_{k\in\mathcal{K}}\tfrac13\Delta_k\Big)
\leq
\omega_0(2\widetilde{\Delta}_{Q_0}),
\]
which eventually gives $\#\mathcal{K}\le C_{\tau} $.

We claim that $B_k^*\cap\Omega\subset T_{Q_0}$, with $B_k^*:=B_{\Delta_k}^*=B(x_k,2\kappa_0r)$ and $\kappa_0$  as in \eqref{definicionkappa0}. To see this let $Y\in B_k^*\cap\Omega$ and take $I\in\mathcal{W}$ such that $Y\in I$. Pick $y_k\in\partial\Omega$ verifying $\dist(I,\partial\Omega)=\dist(I,y_k)$ and let $R_k\in\mathbb{D}$ be the unique dyadic cube such that $y_k\in R_k$ and $\ell(R_k)=\ell(I)$, thus $I\in\mathcal{W}_{R_k}^*$. Let us see that $R_k\in\mathbb{D}_{Q_0}$. First, by \eqref{constwhitney} and our choice of $M$
\[
\ell(R_k)
=
\ell(I)
\le 
\dist(I,\partial\Omega)
\leq
|x_k-Y|
<
2\kappa_0r
=
\frac14{\tau\ell(Q_0)}
<
\frac14{\ell(Q_0)}.
\]
Also, since $x_k\in Q_0\setminus\Sigma_{\tau}$, we can write by \eqref{constwhitney}
\begin{multline*}
\tau\ell(Q_0)
\leq\dist(x_k,\partial\Omega\setminus Q_0)\leq
|x_k-Y|+\diam(I)+\dist(I,y_k)+\dist(y_k,\partial\Omega\setminus Q_0)
\\
<
\frac{1}{4}\tau\ell(Q_0)+\frac{5}{4}\dist(I,\partial\Omega)+\dist(y_k,\partial\Omega\setminus Q_0)\leq \frac{9}{16}\tau\ell(Q_0)+\dist(y_k,\partial\Omega\setminus Q_0),
\end{multline*}
and hence $y_k\in \interior(Q_0)$. Since $y_k\in Q_0\cap R_k$ and $\ell(R_k)<\ell(Q_0)/4$ it follows that $R_k\in\mathbb{D}_{Q_0}$. This and the fact that $Y\in I\in \mathcal{W}_{R_k}^*$ allow us to conclude that $Y\in T_{Q_0}$. Consequently, we have shown that $B_k^*\cap\Omega\subset T_{Q_0}$ and thus $L_2\equiv L_3$ in $B_k^*\cap\Omega$ for every $k\in\mathcal{K}$. 

Next, we observe that $\delta(X_{Q_0})\approx\ell(Q_0)$, $\delta(X_{\Delta_k})\approx\tau\ell(Q_0)$, and $|X_{Q_0}-X_{\Delta_k}|\lesssim\ell(Q_0)$. Hence, we can use Harnack's inequality to move from $X_{Q_0}$ to $X_{\Delta_k}$ with constants depending on $\tau$, Lemma \ref{lemma:proppde} part $(f)$ and Remark \ref{remark:chop-dyadic} to obtain that if $F_j\subset\Delta_j\cap Q_j$
\[
\frac{\omega_2(F_k)}{\omega_2(Q_0)}
\approx
\omega_2^{X_{Q_0}}(F_k)
\approx_\tau
\omega_2^{X_{\Delta_j}}(F_k)
\approx
\omega_3^{X_{\Delta_j}}(F_k)
\approx_\tau
\omega_3^{X_{Q_0}}(F_k)
\approx
\frac{\omega_3(F_k)}{\omega_3(Q_0)}.
\]
This and the fact  $Q_0\setminus\Sigma_{\tau}\subset\bigcup_{k\in\mathcal{K}}\Delta_k$ readily give \eqref{est:L2-L3'} and we finish \textbf{Step 3}.

\subsubsection{Step 4}
Let us recap what we have obtained so far. Fixed $x_0\in\pom$ and $0<r_0<\diam(\pom)/2$, we set $B_0=B(x_0,r_0)$, $\Delta_0=B_0\cap\pom$, $X_0=X_{\Delta_0}$, and  $\omega_0=\omega_{L_0}^{X_0}$, in \textbf{Step 0} we took an arbitrary $j$ and wrote $\widetilde{L}=L^j$, (see \eqref{def:Aj}) and $\widetilde{\omega}=\omega_{\widetilde{L}}^{X_0}$. For an arbitrary $Q^0\in \dd^{\Delta_0}_*$ (see \eqref{D-delta-new}), and for any given 
$Q_0\in\dd_{Q^0}$ we let $\mathcal{F}=\{Q_i\}\subset\mathbb{D}_{Q_0}$ be a family of pairwise disjoint dyadic cubes such that \eqref{hipepsilon0} holds with $\varepsilon_0$ small enough to be chosen. Combining \textbf{Step 1}--\textbf{Step 3} we have shown that if $\varepsilon_0$ is small enough (depending only in the allowable parameters) then 
\eqref{e:AinftyonQ0} is satisfied.  Note that keeping track of the constants one can easily see that $C_\zeta$ does not depend on $j$, $x_0$, $r_0$, $Q^0$ and $Q_0$ ---the fact that $\widetilde{L}=L^j$, which agrees with $L_0$ in small boundary strip, was mainly used, and only in a qualitative fashion, in \eqref{hiding-eqn} in \textbf{Step 1} to \textit{a priori} know that some term is finite so that it can be hidden. We can then invoke Lemma \ref{lemma:extrapolation} with the dyadically doubling measures (see Lemma \ref{lemma:proppde} part $(c)$) $\mu=\omega_0$ and $\nu=\widetilde{\omega}$ to eventually show that \eqref{e:AinftyonQ0} (recalling that $L_3\equiv\widetilde{L}$ as mentioned in \textbf{Step 3}) yields $\widetilde{\omega}\in  A_{\infty}^{\rm dyadic}(Q^0,\omega_0)$ (uniformly on the implicit $j$ and $Q^0$), that is, there exist $1<q<\infty$ and $C$ (independent of $j$ an $Q^0$) such that for every $Q\in\dd_{Q^0}$ with $Q^0\in \dd^{\Delta_0}_*$ 
\begin{equation}\label{RHp-dyad:-Step4}
\left(\aver{Q}h(y;\widetilde{L},L_0,X_0 )^q d \omega_0(y)\right)^{\frac1q} 
\leq 
C 
\aver{Q} h(y;\widetilde{L},L_0,X_0 ) d \omega_0(y) 
=
C\frac{\widetilde{\omega}(Q)}{\omega_0(Q)}.
\end{equation}
Our next goal is to see that $\widetilde{\omega}\in RH_q(\frac54\Delta_0,\omega_{0})$ (uniformly in $j$). To do this let $\Delta=B\cap\pom$ with $B=B(x,r)\subset \frac54B_0$ such that $x\in\pom$. Write $\widetilde{r}=\min\{\frac{r}{4\Xi},\frac{c_0 r_0}{32\kappa_0}\}$, where $\Xi$ is the constant in \eqref{deltaQ}, and let
\[
\widetilde{\dd}^\Delta
=
\Big\{
Q\in\dd:\ Q\cap \Delta\neq\emptyset, \ \widetilde{r}\le \ell(Q)<2\widetilde{r}
\Big\}.
\]
Clearly, $\widetilde{\dd}^\Delta$ is a family of pairwise disjoint cubes such that $\Delta\subset\bigcup_{Q\in \widetilde{\dd}^\Delta} Q\subset 2\Delta$.  Note that if $Q\in \widetilde{\dd}^\Delta$ then $\emptyset\neq Q\cap\Delta\subset Q\cap\frac54\Delta_0\subset Q\cap\frac32\Delta_0$, thus $Q\cap Q^0\neq\emptyset$ for some $Q^0\in\dd^{\Delta_0}_*$. Besides, $\ell(Q)<2\widetilde{r}<c_0r_0/(16\kappa_0)\le \ell(Q^0)$. Consequently, $Q\in\dd_{Q^0}$ and \eqref{RHp-dyad:-Step4} applies to each $Q\in \widetilde{\dd}^\Delta$. All in one we have
\begin{multline*}
\left(\aver{\Delta}h(y;\widetilde{L},L_0,X_0 )^q d \omega_0(y)\right)^{\frac1q} 
\lesssim
\sum_{Q\in \widetilde{\dd}^\Delta} \left(\aver{Q}h(y;\widetilde{L},L_0,X_0 )^q d \omega_0(y)\right)^{\frac1q} 
\\
\lesssim
\sum_{Q\in \widetilde{\dd}^\Delta} \frac{\widetilde{\omega}(Q)}{\omega_0(Q)}
\lesssim
\frac1{\omega_0(\Delta)} \widetilde{\omega}\Big(\bigcup_{Q\in \widetilde{\dd}^\Delta} Q\Big)
\lesssim
\frac{\widetilde{\omega}(2\Delta)}{\omega_0(\Delta)} 
\lesssim
\frac{\widetilde{\omega}(\Delta)}{\omega_0(\Delta)},
\end{multline*}
where we have used that $\omega_0(\Delta)\approx \omega_0(Q)$ 
for every $Q\in \widetilde{\dd}^\Delta$, and also that $\widetilde{\omega}(2\Delta)\approx \widetilde{\omega}(\Delta)$. These in turn follow from Lemma \ref{lemma:proppde} part $(c)$ and the facts that $Q$ meets $\Delta$ and $\ell(Q)\approx \widetilde{r}\approx r$ since $0<r<r_0$. This eventually establishes that $\omega_{L^j}^{X_0}=\widetilde{\omega}\in RH_q(\frac54\Delta_0,\omega_{0})$ with a constant that depends only on the allowable parameters and which is ultimately independent of $j$ and $\Delta_0$. This, as explained in \textbf{Step 0}, allows us to conclude that $\omega_L\in RH_q(\Delta_0,\omega_{0})$ with the help of Lemma \ref{lemma:LjtoL}, completing the proof of Proposition \ref{PROP:LOCAL-VERSION}, part $(a)$. \qed

\subsection{Proof Proposition \ref{PROP:LOCAL-VERSION}, part \texorpdfstring{$(b)$}{(b)}}\label{subsection:proof-b}

We start assuming that $\Omega$ is a \textbf{bounded} 1-sided NTA domain satisfying the CDC and whose boundary $\pom$ is bounded. We fix $\dd=\dd(\pom)$ the dyadic grid from Lemma \ref{lemma:dyadiccubes} with  $E=\pom$. As in the statement of Proposition \ref{PROP:LOCAL-VERSION} let $Lu=-\div(A\nabla u)$ and $L_0u=-\div(A_0\nabla u)$ be two real (non-necessarily symmetric) elliptic operators. Fix $x_0\in\pom$ and $0<r_0<\diam(\pom)$ and let $B_0=B(x_0,r_0)$, $\Delta_0=B_0\cap\pom$. From now on $X_0:=X_{\Delta_0}$, $\omega_0:=\omega_{L_0}^{X_0}$ and $\omega:=\omega_{L}^{X_0}$. As observed in the proof of part $(a)$, without loss of generality we may assume that $0<r_0<\diam(\pom)/2$. 

We fix $1<p<\infty$ and assume that $\vertiii{\varrho(A,A_0)}_{B_0}<\varepsilon$, where $\varepsilon$ is a small enough parameter to be chosen. Our goal is to obtain that $\omega\in RH_p(\Delta_0,\omega_0)$.

We split the proof in several steps.

\subsubsection{Step 0} 
Much as before Lemma \ref{lemma:LjtoL} guarantee that just need to see that for every $j$ large enough $\omega_{L^j}\in RH_p(\frac54\Delta_0,\omega_{0})$ uniformly in $j$ and in $\Delta_0$. Thus we fix  $j\in\mathbb{N}$ and let $\widetilde{L}=L^j$ be the operator defined by $\widetilde{L}u=-\div(\widetilde{A}\nabla u)$, with $\widetilde{A}=A^j$ (see \eqref{def:Aj}), and set $\widetilde{\omega}:=\omega_{\widetilde{L}}^{X_0}$.
As mentioned above $\widetilde{A}$ is uniformly elliptic with constant $\Lambda_0=\max\{\Lambda_A,\Lambda_{A_0}\}$. Also, since $\widetilde{L}\equiv L_0$ in $\{Y\in\Omega:\,\delta(Y)< 2^{-j}\}$, the analogous step in part $(a)$ showed,  $\omega_{0}\ll \omega_{\widetilde{L}}\ll \omega_{0}$ and 
$h(\cdot\,;\widetilde{L}, L_0,X)\in L_{\rm loc}^\infty(\pom , \omega_{L_0}^{Y})$ for every $X,Y\in\Omega$ ---the actual norm will depend on $X$, $Y$ and $j$, but we will use this fact in a qualitative fashion. This qualitative control will be essential in the following steps. At the end of \textbf{Step 3} we will have obtained the desired conclusion for the operator $\widetilde{L}=L^j$, with constants independent of $j\in\mathbb{N}$,  which as observed above will allow us to complete the proof by Lemma \ref{lemma:LjtoL}.

\medskip

\subsubsection{Step 1}
Consider an arbitrary surface ball $\Delta_1=\Delta(x_1,r_1)$ with $x_1\in\frac54\Delta_0$ and $0<r_1\le \tfrac{c_0}{10^5\kappa_0^3}r_0$, and let $B_1=B(x_1,r_1)$. Set
$\Delta_\star:=B_\star\cap\pom$ with $B_\star:=B(x_\star,r_\star)$ where $x_\star=x_1$ and $r_\star=2\kappa_0 r_1$ (hence $\Delta_\star=2\kappa_0\Delta_1$) satisfy $x_\star\in\frac54\Delta_0$ and $0<r_\star\le \tfrac{2c_0}{10^5\kappa_0^2}r_0$. By \eqref{definicionkappa0}, \eqref{propQ0}  we have
\begin{equation}\label{eq:X0-TQ-star
}
X_{\star}=X_{c_0^{-1}\Delta_\star^{*}}\in\Omega\setminus B_{\Delta_\star}^*\subset\Omega\setminus \tfrac12 B_{\Delta_\star}^*\subset\Omega\setminus T_{\Delta_\star}^{**}.
\end{equation}
Note also that  $2\kappa_0 r_\star\le \delta(X_\star)<r_0$. We claim that $\dd^{\Delta_\star}\subset \dd^{\Delta_0}_{**}:=\bigcup_{Q^0\in \dd^{\Delta_0}_*} \dd_{Q^0}$  (see \eqref{D-delta} and \eqref{D-delta-new}). To see this, let $Q_0\in\dd^{\Delta_\star}$  and pick  $y_\star\in Q_0\cap 2\Delta_\star$. Then
\[
|y_\star-x_0|
\le
|y_\star-x_\star|+|x_\star-x_0|
<
2 r_\star+\frac54 r_0
\le
\Big(\frac{4 c_0}{10^5\kappa_0^2}+\frac54\Big)r_0
<
\frac32 r_0,
\]
hence $y_\star\in \frac32 \Delta_0$ and there exists a unique $Q^0\in \dd^{\Delta_0}_*$ such that $y_\star\in Q^0$. Moreover, by construction
\[
\ell(Q_0)
=
2^{-k(\Delta_\star)}
<
400 r_\star
\le 
\frac{c_0}{125\kappa_0^2}r_0
<
\frac{c_0}{16\kappa_0}r_0
<\ell(Q^0),
\] 
and therefore $Q_0\in\dd_{Q^0}$ as desired. 

Set $\mathcal{E}(Y):=A(Y)-A_0(Y)$, $Y\in\Omega$, and consider $\gamma=\{\gamma_Q\}_{Q\in\dd^{\Delta_0}_{**}}$ 
\begin{equation}\label{def-gamma:main:(b)}
\gamma_Q=\gamma_{X_0,Q} :=
\omega_{0}(Q)\sum\limits_{I\in\mathcal{W}_{Q}^{*}}\sup\limits_{Y\in I^*} \|\mathcal{E}\|_{L^\infty(I^*)}^{2},
\qquad \mbox{whenever}\, \, Q\in\dd^{\Delta_0}_{**}.
\end{equation}
Lemma \ref{lemma:carleson-discrete}  yields that for every $Q_0\in \dd^{\Delta_\star}$, if  $Q^0\in \dd^{\Delta_0}_*$ is selected so that $Q_0\in \dd_{Q^0}$
\begin{equation}\label{Car-small-(b)}
\|\mathfrak{m}_{\gamma}\|_{\mathcal{C}(Q_0, \omega_0)}
\le
\|\mathfrak{m}_{\gamma}\|_{\mathcal{C}(Q^0, \omega_0)}
\lesssim
\vertiii{\varrho(A,A_0)}_{B_0}
<\varepsilon,
\end{equation}
where the last inequality is our main assumption in the current scenario and $\varepsilon$ is to be chosen. 

We  also set $\omega_0^\star=\omega_0^{X_\star}$ and 
$\gamma^\star=\{\gamma_Q^\star\}_{Q\in\dd^{\Delta_\star}}$ where
\[
\gamma_Q^\star:=
\omega_{0}^{\star}(Q)\sum_{I\in\mathcal{W}_{Q}^{*}}\sup_{Y\in I^*} \|\mathcal{E}\|_{L^\infty(I^*)}^{2},
\qquad \mbox{whenever}\, \, Q\in\dd^{\Delta_\star}.
\]
Using \eqref{chop-dyadic:densities} and Harnack's inequality we have that $\omega_{0}^{\star}(Q)\approx \omega_{0}(Q)/\omega_{0}(Q_0^\star)$. Hence, by \eqref{def-gamma:main:(b)}
\[
\gamma_Q^\star
\approx
\frac{\omega_{0}(Q)}{\omega_0(Q_0^\star)} \sum_{I\in\mathcal{W}_{Q}^{*}}\sup_{Y\in I^*} \|\mathcal{E}\|_{L^\infty(I^*)}^{2}
=
\frac{\gamma_Q}{\omega_0(Q_0^\star)} , \qquad  Q\in\dd^{\Delta_\star}.
\]
and, by \eqref{Car-small-(b)},
\begin{multline}\label{Car-star-0:(b)}
\|\mathbf{m}_{\gamma^\star}\|_{\mathcal{C}(Q_0^\star,\omega_0^\star)}
=
\sup_{Q\in \mathbb{D}_{Q_0^\star}} \frac{\mathbf{m}_{\gamma^\star}(\dd_Q)}{\omega_0^\star(Q)} 
\approx
\sup_{Q\in \mathbb{D}_{Q_0^\star}} \frac{\mathbf{m}_{\gamma}(\dd_{Q})}{\omega_0^\star(Q)\omega_0(Q_0^\star)} 
\\
\approx
\sup_{Q\in \mathbb{D}_{Q_0^\star}} \frac{\mathbf{m}_{\gamma}(\dd_{Q})}{\omega_0(Q)} 
\le
\|\mathbf{m}_{\gamma}\|_{\mathcal{C}(Q_0,\omega_0)}
\lesssim
\varepsilon.
\end{multline}

We modify the operator $\widetilde{L}$ inside the region $T_{\Delta_\star}$ (see \eqref{def:T-Delta}), by defining $L_1=L_1^{\Delta_\star}$ as $L_1u=-\div(A_1\nabla u)$, where
\[
A_1(Y):=
\left\{
\begin{array}{ll}
\widetilde{A}(Y) &\mbox{if}\, \, Y\in T_{\Delta_\star},
\\[5pt]
A_0(Y) &\mbox{if}\, \, Y\in\Omega\setminus T_{\Delta_\star}. 
\end{array}
\right.
\]
See Figure \ref{figure:def:A1-(b)}. Write $\omega_1^X=\omega_{L_1}^X$ for every $X\in\Omega$ and $\omega_\star=\omega_{L_1}^{X_\star}$.

\begin{figure}[!ht]
	\centering
	\begin{tikzpicture}[scale=.4]
	
	\draw[fill=gray, opacity=.3] plot [smooth] coordinates {(-9.6,6)(-9.3,4)(-8.8,2.8)(-8.5,2.1)(-8,1.1)(-6,.6) (-4,.3)(-2,-.3) (0,0)(2,.3)(4,-.3)(6,-.6)(8,-.9)(8.5,2.1)(8.8,2.8)(9.3,4)(9.6,6)(3,6.3)(-3,5.8)(-9.6,6)};
	\draw[dashed, yshift=+1cm] plot [smooth] coordinates {(-14,.5)(-12,.2)(-10,.5)(-8,1)(-6,.6) (-4,.3) (-2,-.3) (0,0)(2,.3)(4,-.3)(6,-.6)(8,-1)(10,-.5)(12,.2)(14,-.5)};
	\draw[red, fill=gray=, opacity=.2] (-14,9)--plot [smooth] coordinates {(-14,.5)(-12,.2)(-10,.5)(-8,.98)(-6,.6) (-4,.3) (-2,-.3) (0,0)(2,.3)(4,-.3)(6,-.6)(8,-1)(10,-.5)(12,.2)(14,-.5)}--(14,9)--cycle;
	\filldraw (0,0) circle (2pt) node[below] {$x_\star$};

	\draw[red] plot [smooth] coordinates {(-14,.5)(-12,.2)(-10,.5)(-8,.98)(-6,.6) (-4,.3) (-2,-.3) (0,0)(2,.3)(4,-.3)(6,-.6)(8,-1)(10,-.5)(12,.2)(14,-.5)};

	\node[left, red] at (-14,.5) {$\partial\Omega$};
	\draw[<->] (12,.2)--(12,1.2);
	\node at (13,.5) {{\tiny $2^{-j}$}};
	
	\node[gray] at (7,5) {$T_{\Delta_\star}$};
	
	\node[blue] at (0,7.5) {$A_0$};
	\node[blue] at (11,3) {$A_0$};
	\node[blue] at (-11,3) {$A_0$};
	
	\begin{scope}
	
	\node[rotate=15] at (10,0) {\tiny $A_0$};
	\node[rotate=15] at (-10,1) {\tiny $A_0$};

	\draw[clip] plot [smooth] coordinates {(-9.6,6)(-9.3,4)(-8.8,2.8)(-8.5,2.1)(-8,1.1)(-6,.6) (-4,.3)(-2,-.3) (0,0)(2,.3)(4,-.3)(6,-.6)(8,-.9)(8.5,2.1)(8.8,2.8)(9.3,4)(9.6,6)(3,6.3)(-3,5.8)(-9.6,6)};
	\fill[gray] plot [smooth] coordinates {(-14, -1.5)(-14,1.5)(-12,1.2)(-10,1.5)(-8,2)(-6,1.6) (-4,1.3) (-2,.7) (0,1)(2,1.3)(4,.7)(6,.4)(8,0)(10,.5)(12,1.2)(14,.5)(14, -1.5)(-14,-1.5)};
	
	\node at (0,3) {$A$};
	
	\node[rotate=-15] at (-6,1.1) {{\tiny $A_0$}};
	\node at (2,.8) {\tiny $A_0$};
	\end{scope}

	\end{tikzpicture}
	\caption{Definition of $A_1$ in $\Omega$.\label{figure:def:A1-(b)}}
\end{figure}
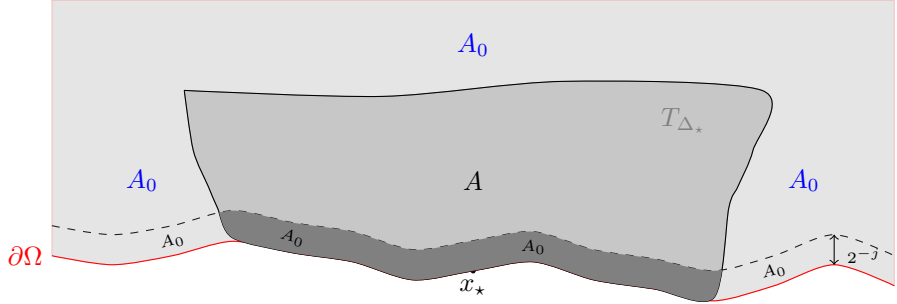

Recalling that $\widetilde{A}=A^j$ (see \eqref{def:Aj}), it is clear that $\mathcal{E}_1:=A_1-A_0$ verifies $|\mathcal{E}_1|\leq|\mathcal{E}|\mathbf{1}_{T_{\Delta_\star}}$ and also $\mathcal{E}_1(Y)=0$ if $\delta(Y)<2^{-j}$ (this latter condition will be used qualitatively). Hence much as before if write $\omega_1^X=\omega_{L_1}^X$ for every $X\in\Omega$ we have that $\omega_1^X\ll\omega_0^X$ for every $X\in\Omega$ and hence we can write $h(\cdot\,; L_1, L_0, X)=d\omega_1^X/d\omega_0^X$ which is well-defined $\omega_0^X$-a.e. Also, as shown in \textbf{Step 0} we have that $h(\cdot\,;L_1, L_0,X)\in L_{\rm loc}^\infty(\pom, \omega_0^Y)$ for every $X,Y\in\Omega$ (the bound depends on $X, Y$ and the fixed $j$ but we will use this qualitatively).  

In order to simplify the notation, we recall \eqref{definicionkappa0}, \eqref{propQ0}, and set $\widehat{\Delta}_\star:=\tfrac{1}{2}\Delta_\star^*=\Delta(x_\star,\kappa_0 r_\star)$ and let $0\le g\in L^{p'}(\widehat{\Delta}_\star,\omega_0^\star)$ with $\|g\|_{L^{p'}(\widehat{\Delta}_\star,\omega_0^\star)}=1$. Extend $g$ by 0 in $\pom\setminus \widehat{\Delta}_\star$. Set $g_t= P_t g$ with $0<t<\kappa_0 r_1/3$ (see \eqref{defgt}).
It is easy to see that $\widehat{\Delta}_\star\subset \frac32\Delta_0$, hence  $\widehat{\Delta}_\star$ can be covered by the cubes in $\dd^{\Delta_0}_*$. This and the fact that $r_\star/3<c_0 r_0/(16 \kappa_0)$ guarantee that Lemma \ref{lemma:Pt} applies to give $g_t\in \Lip(\partial\Omega)$ with $\supp(g_t)\subset  \Delta_\star^*$. We then consider
\[
u^t_0(X)=\int_{\partial\Omega} g_t(y) d\omega_{0}^X(y) \qquad \mbox{and} \qquad  u_1^t(X)=\int_{\partial\Omega} g_t(y) d\omega_1^X(y),
\qquad X\in\Omega.
\]
Since $\Omega$ is bounded, we can use Lemma \ref{lemma:u0-u:bounded} (slightly moving $X_\star$ if needed). This, Lemma \ref{prop:step1-estimate} with $\F=\emptyset$, \eqref{Car-star-0:(b)}, and Hölder's inequality yield
\begin{align}\label{error:(b)}
|u^t_1(X_\star)-u^t_0(X_\star)|
&= 
\left| \iint_{\Omega} (A_0-A_1)^\top(Y) \nabla_{Y} G_{L_1^\top}(Y,X_\star)\cdot \nabla u^t_0(Y)dY\right|
\\ \nonumber
&
\leq 
\iint_{T_{\Delta_\star}} |\mathcal{E}(Y)|\,|\nabla_{Y} G_{L_1^\top}(Y,X_\star)| \,|\nabla u^t_0(Y)|\,dY
\\ \nonumber
&
\le
\sum_{Q_0\in\dd^{\Delta_\star}}\iint_{T_{Q_0}} |\mathcal{E}(Y)|\,|\nabla_{Y} G_{L_1^\top}(Y,X_\star)| \,|\nabla u^t_0(Y)|\,dY
\\ \nonumber
&
\le
\sum_{Q_0\in\dd^{\Delta_1}}
\|\mut_{\gamma^\star}\|_{\C(Q_0,\omega_0^\star)}^\frac12
\int_{Q_0} M^{\mathbf{d}}_{Q_0^\star,\omega_0^\star}(\omega_\star)(x) \mathcal{S}_{Q_0^\star}u_0^t(x) d\omega_0^\star(x)
\\ \nonumber
&
\le
\varepsilon^{\frac12} \sum_{Q_0\in\dd^{\Delta_\star}}\int_{Q_0} M^{\mathbf{d}}_{Q_0,\omega_0^\star}(\omega_\star)(x) \mathcal{S}_{Q_0}u_0^t(x) d\omega_0^\star(x)
\\ \nonumber
&
\lesssim
\varepsilon^{\frac12} \sum_{Q_0\in\dd^{\Delta_\star}} \|M^{\mathbf{d}}_{Q_0,\omega_0^\star}(\omega_\star)\|_{L^{p}(Q_0,\omega_0^\star)}\,\|\mathcal{S}_{Q_0}u_0^t(x)\|_{L^{p'}(Q_0,\omega_0^\star)}.
\end{align}
Using the well-known fact that $M^{\mathbf{d}}_{Q_0,\omega_0^\star}$ is bounded on $L^{p}(Q_0,\omega_0^\star)$ and that, as mentioned before $\omega_\star\ll\omega_0^\star$ with $h(\cdot\,; L_1^\star, L_0, X_\star)=d\omega_\star/d\omega_0^\star$, it readily follows that 
\[
\|M^{\mathbf{d}}_{Q_0,\omega_0^\star}(\omega_\star)\|_{L^{p}(Q_0,\omega_0^\star)}\lesssim \|h(\cdot\,; L_1, L_0, X_\star)\|_{L^{p}(Q_0,\omega_0^\star)}.
\]
On the other hand, given $Q_0\in\dd^{\Delta_\star}$, let $Q^0\in \dd^{\Delta_0}_*$ be such that $Q_0\subset Q^0$. We claim that  $\Delta_\star^*\subset 2\widetilde{\Delta}_{Q^0}$ and hence $\supp g_t\subset 2\widetilde{\Delta}_{Q^0}$. Indeed, if $y\in \Delta_\star^*$ and we recall that $y_\star\in Q_0\cap 2\Delta_\star$ we obtain 
\begin{multline*}
|y-x_{Q^0}|
\le
|y-x_\star|+|x_\star-y_\star|+|y_\star-x_{Q^0}|
<
2(\kappa_0 +1) r_\star+Cr_{Q^0}
\\
\le 
\frac{8 c_0}{10^5\kappa_0}r_0 
+
\Xi r_{Q^0}
<
\frac{128}{10^5}\ell(Q^0)+\Xi r_{Q^0}
<
2\Xi  r_{Q^0},
\end{multline*}
thus $y\in 2\widetilde{\Delta}_{Q^0}$ as desired. On the other hand, observe that $X_0\in\Omega\setminus2\kappa_0 B_{\Delta_\star}^*=B(x_\star,2\kappa_0^2 r_\star)$, for otherwise we would get a contradiction:
\[
c_0r_0
\le 
\delta(X_0)
\le 
|X_0-x_\star|
<
2\kappa_0^2 r_\star
\le 
\frac{4c_0}{10^5}r_0.
\]
Hence Lemma \ref{lemma:proppde} part $(d)$ and Harnack's inequality to pass from $X_\star$ to $X_{\Delta_\star^*}$
\begin{equation}\label{est-den:(b)}
\frac{d\omega_0^{\star}}{d\omega_0}
\approx \frac1{\omega_0(\Delta_\star^*)},
\qquad \mbox{$\omega_0$-a.e. in }\Delta_\star^*. 
\end{equation}
After all these observations we use Harnack's inequality to pass from $X_\star$ to $X_{Q^0}$ and from $X_{Q^0}$ to $X_0$, Remark \ref{remark:chop-dyadic}, the square-function non-tangential estimates from \cite[Theorem 1.5, Proposition 2.57]{AHMT-I}, and Lemmas \ref{lemma:Pt} and \ref{lemma:proppde} to conclude
\begin{multline*}
\omega_0(Q_0)^{\frac1{p'}}\|\mathcal{S}_{Q_0}u_0^t(x)\|_{L^{p'}(Q_0,\omega_0^\star)}
\lesssim
\|\mathcal{S}_{Q^0}u_0^t(x)\|_{L^{p'}(Q^0,\omega_0^{X_{Q^0}})}
\lesssim
\|g_t\|_{L^{p'}(Q^0,\omega_0^{X_{Q^0}})}
\approx
\|g_t\|_{L^{p'}(Q^0,\omega_0)}
\\
\lesssim
\|g\|_{L^{p'}(3\widetilde{\Delta}_{Q^0},\omega_0)}
=
\|g\|_{L^{p'}(\widehat{\Delta}_{\star},\omega_0)}
\approx
\omega_0(\Delta_\star^*)^{\frac1{p'}}
\|g\|_{L^{p'}(\widehat{\Delta}_{\star},\omega_0^\star)}
\approx \omega_0(Q_0)^{\frac1{p'}}.
\end{multline*}

Plugging the obtained estimates into \eqref{error:(b)} we conclude that 
\begin{align*}
|u_1^t(X_\star)-u^t_0(X_\star)|
\lesssim
\varepsilon^{\frac12} \sum_{Q_0\in\dd^{\Delta_\star}} \|h(\cdot\,; L_1, L_0, X_\star)\|_{L^{p}(Q_0,\omega_0^\star)}
\lesssim
\varepsilon^{\frac12}  \|h(\cdot\,; L_1, L_0, X_\star)\|_{L^{p}(\widehat{\Delta}_{\star},\omega_0^\star)},
\end{align*}
where we have used \eqref{propQ0} and that $\dd^{\Delta_\star}$ has bounded cardinality, which follows from 
$\omega_0(Q_0)\approx \omega_0(\widehat{\Delta}_{\star})$ for every $Q_0\in\dd^{\Delta_\star}$ and \eqref{propQ0}.
Using then  the definitions of $u_0^t$ and $u_1^t$ we conclude that
\begin{multline}\label{est-error:t-(b)}
\Big|\int_{\partial\Omega} g\, d\omega_{\star}-\int_{\partial\Omega} g\, d\omega_{0}^\star\Big|
\le
|u^t_1(X_\star)-u^t_0(X_\star)|+\|g-g_t\|_{L^1(\pom,\omega_{0}^\star)}+\|g-g_t\|_{L^1(\pom,\omega_\star)}
\\
\lesssim
\varepsilon^{\frac12} \|h(\cdot\,; L_1, L_0, X_\star)\|_{L^{p}(\widehat{\Delta}_{\star},\omega_0^\star)}+ \|g-g_t\|_{L^1(\pom,\omega_{0}^\star)}+\|g-g_t\|_{L^1(\pom,\omega_\star)}.
\end{multline}

Fix $Q_0\in\dd^{\Delta_\star}$,  we showed before that if we pick $Q^0\in \dd^{\Delta_0}_*$ so that $Q_0\subset Q^0$, then $\Delta_\star^*\subset 2\widetilde{\Delta}_{Q^0}$. Recalling that $0\le g\in L^{p'}(\widehat{\Delta}_\star,\omega_0^\star)$, with $\supp(g), \supp(g_t)\subset \Delta_\star^*$, then
\eqref{est-den:(b)} and Lemma \ref{lemma:Pt} give 
\begin{multline}
\label{g-gt:0-(b)}
\|g-g_t\|_{L^1(\pom,\omega_{0}^\star)}
=
\|g-g_t\|_{L^1(\Delta_\star^*,\omega_{0}^\star)}
\approx
\frac1{\omega_0(\Delta_\star^*)}
\|g-g_t\|_{L^1(\Delta_\star^*,\omega_{0})}
\\
\le
\frac1{\omega_0(\Delta_\star^*)} \|g-P_t g\|_{L^1(2\widetilde{\Delta}_{Q^0}, \omega_{0})}
\longrightarrow  0,
\quad\mbox{as }t\to 0^+.
\end{multline}
Similarly, using also that as mentioned above $\omega_1\ll\omega_0$ with $h(\cdot\,;L_1, L_0,X_\star)\in L_{\rm loc}^\infty(\pom, \omega_0)$
\begin{multline}\label{g-gt:1-(b)}
\|g-g_t\|_{L^1(\pom,\omega_\star)}
=
\|g-P_t g\|_{L^1(\Delta_\star^*, \omega_\star)}
\\
\le
\|h(\cdot\,;L_1, L_0,X_\star)\|_{L^\infty(\Delta_\star^*,\omega_0^\star)}
\|g-P_t g\|_{L^1(\Delta_\star^*, \omega_{0}^\star)}
\longrightarrow  0,
\qquad\mbox{as }t\to 0^+.
\end{multline}
Combining \eqref{est-error:t-(b)}, \eqref{g-gt:0-(b)}, \eqref{g-gt:1-(b)} and letting $t\to 0^+$ we conclude that
\begin{multline*}
0
\le
\int_{\widehat{\Delta}_\star} g(y) d\omega_\star(y)
=
\int_{\partial\Omega} g(y) d\omega_\star(y)
=\int_{\partial\Omega} g(y)\, h(y; L_1, L_0, X_\star)\, d\omega_{0}^\star(y)
\\
\le
\varepsilon^{\frac12} \|h(\cdot\,; L_1, L_0, X_\star)\|_{L^{p}(\widehat{\Delta}_{\star},\omega_0^\star)}
+
\int_{\partial\Omega} g(y) d\omega_0^\star(y)
\\
\le
\varepsilon^{\frac12} \|h(\cdot\,; L_1, L_0, X_\star)\|_{L^{p}(\widehat{\Delta}_{\star},\omega_0^\star)}
+
\omega_0^\star(\widehat{\Delta}_\star)^{\frac1{p}}.
\end{multline*}
Taking now the sup over all  $0\le g\in L^{p'}(\widehat{\Delta}_\star,\omega_0^\star)$ with $\|g\|_{L^{p'}(\widehat{\Delta}_\star,\omega_0^\star)}=1$  we eventually get
\begin{equation}\label{hiding-eqn-(b)}
\|h(\cdot\,; L_1, L_0, X_\star)\|_{L^{p}(\widehat{\Delta}_{\star},\omega_0^\star)}
\lesssim
\varepsilon^{\frac12} \|h(\cdot\,; L_1, L_0, X_\star)\|_{L^{p}(\widehat{\Delta}_{\star},\omega_0^\star)}
+ 
\omega_0^\star(\widehat{\Delta}_\star)^{\frac1{p}}.
\end{equation}
Since $h(\cdot\,;L_1, L_0,X_\star)\in L_{\rm loc}^\infty(\pom, \omega_0^\star)$ (albeit with bounds which may depend on $X_\star$ or $j$) we can hide the first term on the right hand side and eventually obtain fixing $\varepsilon$ small enough (depending on $n$, the 1-sided NTA constants, the CDC constant, the ellipticity constants of $L_0$ and $L_2$, and on $p$) 
\begin{equation}\label{conStep1-(b)}
 \|h(\cdot\,; L_1, L_0, X_\star)\|_{L^{p}(\widehat{\Delta}_{\star},\omega_0^\star)}
\lesssim 
\omega_0^\star(\widehat{\Delta}_\star)^{\frac1{p}}.
\end{equation}

\subsubsection{Step 2} 

Let us next define 
\[
A_2(Y):=\left\{
\begin{array}{ll}
A_1(Y) & \hbox{$\text{if }Y\in T_{\Delta_\star}$,} 
\\[5pt]
\widetilde{A}(Y) & \hbox{$\text{if }Y\in \Omega\setminus T_{\Delta_\star}$,}
\end{array}
\right.
\]
and set $L_2u:=-\div(A_2\nabla u)$. Note that $L_2\equiv\widetilde{L}$ in $\Omega$  (see Figure \ref{figure:def:A2-(b)}). Since $\widetilde{L}\equiv L_0$ in $\{Y\in\Omega:\,\delta(Y)< 2^{-j}\}$ we have already mentioned in \textbf{Step 0} that $\omega_{L_2}=\omega_{\widetilde{L}}$ and $\omega_{L_0}$ are mutually absolutely continuous with
$h(\cdot\,;\widetilde{L}, L_0,X)\in L_{\rm loc}^\infty(\pom , \omega_{L_0}^{Y})$ for every $X,Y\in\Omega$.

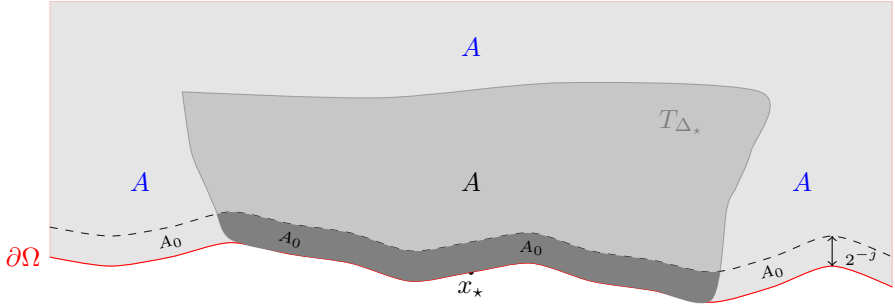
\begin{figure}[!ht]
\centering
\begin{tikzpicture}[scale=.4]

\draw[fill=gray, opacity=.3] plot [smooth] coordinates {(-9.6,6)(-9.3,4)(-8.8,2.8)(-8.5,2.1)(-8,1.1)(-6,.6) (-4,.3)(-2,-.3) (0,0)(2,.3)(4,-.3)(6,-.6)(8,-.9)(8.5,2.1)(8.8,2.8)(9.3,4)(9.6,6)(3,6.3)(-3,5.8)(-9.6,6)};
\draw[dashed, yshift=+1cm] plot [smooth] coordinates {(-14,.5)(-12,.2)(-10,.5)(-8,1)(-6,.6) (-4,.3) (-2,-.3) (0,0)(2,.3)(4,-.3)(6,-.6)(8,-1)(10,-.5)(12,.2)(14,-.5)};
\draw[red, fill=gray=, opacity=.2] (-14,9)--plot [smooth] coordinates {(-14,.5)(-12,.2)(-10,.5)(-8,.98)(-6,.6) (-4,.3) (-2,-.3) (0,0)(2,.3)(4,-.3)(6,-.6)(8,-1)(10,-.5)(12,.2)(14,-.5)}--(14,9)--cycle;
\filldraw (0,0) circle (2pt) node[below] {$x_\star$};

\draw[red] plot [smooth] coordinates {(-14,.5)(-12,.2)(-10,.5)(-8,.98)(-6,.6) (-4,.3) (-2,-.3) (0,0)(2,.3)(4,-.3)(6,-.6)(8,-1)(10,-.5)(12,.2)(14,-.5)};

\node[left, red] at (-14,.5) {$\partial\Omega$};
\draw[<->] (12,.2)--(12,1.2);
\node at (13,.5) {{\tiny $2^{-j}$}};

\node[gray] at (7,5) {$T_{\Delta_\star}$};

\node[blue] at (0,7.5) {$A$};
\node[blue] at (11,3) {$A$};
\node[blue] at (-11,3) {$A$};

\begin{scope}

\node[rotate=15] at (10,0) {\tiny $A_0$};
\node[rotate=15] at (-10,1) {\tiny $A_0$};

\path[clip] plot [smooth] coordinates {(-9.6,6)(-9.3,4)(-8.8,2.8)(-8.5,2.1)(-8,1.1)(-6,.6) (-4,.3)(-2,-.3) (0,0)(2,.3)(4,-.3)(6,-.6)(8,-.9)(8.5,2.1)(8.8,2.8)(9.3,4)(9.6,6)(3,6.3)(-3,5.8)(-9.6,6)};
\fill[gray] plot [smooth] coordinates {(-14, -1.5)(-14,1.5)(-12,1.2)(-10,1.5)(-8,2)(-6,1.6) (-4,1.3) (-2,.7) (0,1)(2,1.3)(4,.7)(6,.4)(8,0)(10,.5)(12,1.2)(14,.5)(14, -1.5)(-14,-1.5)};

\node at (0,3) {$A$};

\node[rotate=-15] at (-6,1.1) {{\tiny $A_0$}};
\node at (2,.8) {\tiny $A_0$};
\end{scope}

\end{tikzpicture}
\caption{Definition of $A_2$ in $\Omega$.\label{figure:def:A2-(b)}}
\end{figure}

Note that by construction $B_1=\frac1{2\kappa_0}B_\star$.  Besides, by \eqref{definicionkappa0}, $2\kappa_0 B_1\cap\Omega\subset\tfrac{5}{4} B_\star\cap\Omega\subset T_{\Delta_\star}$ and since $\widetilde{L}\equiv L_2\equiv L_1$ in $T_{\Delta_0}$, Lemma \ref{lemma:proppde} part $(f)$ and Harnack's inequality give that $\omega_{\widetilde{L}}^{X_\star}$ and $\omega_{L_1}^{X_\star}=\omega_\star$ are comparable in $\Delta_1$, thus $h(\cdot\,; L_1, L_0, X_\star)\approx h(\cdot\,; \widetilde{L}, L_0, X_\star)$ for $\omega_0^\star$-a.e. $y\in  \Delta_1$ (and also $\omega_0$-a.e.). 
On the other hand using that as shown above $X_0\in\Omega\setminus2\kappa_0 B_{\Delta_\star}^*\subset \Omega\setminus2\kappa_0 B_1$ we can invoke Lemma \ref{lemma:proppde} part $(d)$ and Harnack's inequality to see that
\begin{multline*}
h(\cdot\,; \widetilde{L}, L_0, X_0)
=
\frac{d \omega_{\widetilde{L}}^{X_0}}{d \omega_{L_0}^{X_0}}
=
\frac{d \omega_{\widetilde{L}}^{X_0}}{d \omega_{\widetilde{L}}^{X_\star}}
\,
\frac{d \omega_{\widetilde{L}}^{X_\star}}{d \omega_{L_0}^{X_\star}}
\,\frac{d \omega_{L_0}^{X_\star}}{d \omega_{L_0}^{X_0}}
\approx
\frac{\omega_1(\Delta_1)}{\omega_0(\Delta_1)} h(\cdot\,; \widetilde{L}, L_0, X_\star)
\\
\approx
\frac{\widetilde{\omega}(\Delta_1)}{\omega_0(\Delta_1)} h(\cdot\,; L_1, L_0, X_\star),
\end{multline*}
for $\omega_0$-a.e. $y\in  \Delta_1$ (recall that  $\omega_{\widetilde{L}}$ and $\omega_{0}$ are mutually absolutely continuous). 
This, the fact that $\Delta_1\subset \widehat{\Delta}_{\star}$, \eqref{conStep1-(b)} and Lemma \ref{lemma:proppde} part $(d)$ yield
\begin{align}\label{conc-Step2-(b)}
\left(\aver{\Delta_1} h(y; \widetilde{L}, L_0, X_0)^{p}d\omega_0(y)\right)^{\frac1{p}}\!
\approx
\frac{\widetilde{\omega}(\Delta_1)}{\omega_0(\Delta_1)}
\left(\aver{\Delta_1} h(y; \widetilde{L}, L_0, X_\star)^{p}d\omega_0^\star(y)\right)^{\frac1{p}}
\lesssim
\frac{\widetilde{\omega}(\Delta_1)}{\omega_0(\Delta_1)}.
\end{align}

\subsubsection{Step 3}
Let us summarize what we have obtained up to this point. We fixed $x_0\in\pom$ and $0<r_0<\diam(\pom)/2$, we set $B_0=B(x_0,r_0)$, $\Delta_0=B_0\cap\pom$, $X_0=X_{\Delta_0}$, and  $\omega_0=\omega_{L_0}^{X_0}$. We also fix $1<p<\infty$ and assumed that $\vertiii{\varrho(A,A_0)}_{B_0}<\varepsilon$ with $\varepsilon$ small enough at our disposal. In \textbf{Step 0} we took an arbitrary $j$ and wrote $\widetilde{L}=L^j$, (see \eqref{def:Aj}) and $\widetilde{\omega}=\omega_{\widetilde{L}}^{X_0}$. For an arbitrary surface ball $\Delta_1=\Delta(x_1,r_1)$ with $x_1\in\frac54\Delta_0$ and $0<r_1\le \tfrac{c_0}{10^5\kappa_0^3}r_0$ we have obtained, combining \textbf{Step 1} and \textbf{Step 2}, that provided $\varepsilon$ is small enough (independently of $j$ and $\Delta_1$) then  \eqref{conc-Step2-(b)} holds. 

Our next goal is to see that \eqref{conc-Step2-(b)} holds as well with $\frac54\Delta_0$ replacing $\Delta_1$. To do this 
$r=\tfrac{c_0}{10^5\kappa_0^3}r_0$ and find a maximal collection of points $\{x_k\}_{k\in\mathcal{K}}\subset \frac54\Delta_0$ with respect to the property that $|x_k-x_{k'}|>2 r/3$ for every $k,k'\in\mathcal{K}$ with $k\neq k'$. Write 
$\Delta_k=\Delta(x_k,r)$ and note that  $\{\frac13\Delta_k\}_{k\in\mathcal{K}}$ is a family of pairwise disjoint surface balls such that  $\frac54 \Delta_0\subset\bigcup_{k\in\mathcal{K}}\Delta_k\subset\frac32\Delta_0$. Note that since $r\approx r_0$ and $x_k\in\frac54\Delta_0$ it follows from Lemma \ref{lemma:proppde} part $(c)$ that 
$\omega_0(\frac54\Delta_0)\approx\omega_0(\Delta_k)$ and $\widetilde{\omega}(\frac32\Delta_0)\approx \widetilde{\omega}(\frac54\Delta_0)\approx \widetilde{\omega}(\Delta_k)\approx\widetilde{\omega}(\frac13\Delta_k)$ for every $k\in\mathcal{K}$. Thus using \eqref{conc-Step2-(b)} for every $\Delta_k$ (whose applicability is ensure by the facts that $x_k\in \frac54\Delta_0$ and $r_{\Delta_k}=r=\tfrac{c_0}{10^5\kappa_0^3}r_0$) it follows that 
\begin{multline}\label{RHP-top-scale}
\left(\aver{\frac54\Delta_0} h(y; \widetilde{L}, L_0, X_0)^{p}d\omega_0(y)\right)^{\frac1{p}}
\lesssim
\sum_{k\in\mathcal{K}} \left(\aver{\Delta_k} h(y; \widetilde{L}, L_0, X_0)^{p}d\omega_0(y)\right)^{\frac1{p}}
\\
\lesssim
\sum_{k\in\mathcal{K}}
\frac{\widetilde{\omega}(\Delta_k)}{\omega_0(\Delta_k)}
\approx
\frac1{\omega_0(\frac54\Delta_0)} 
\widetilde{\omega}\Big(\bigcup_{k\in\mathcal{K}}\tfrac13\Delta_k\Big)
\le
\frac{\widetilde{\omega}(\frac32\Delta_0)} {\omega_0(\frac54\Delta_0)} 
\approx
\frac{\widetilde{\omega}(\frac54\Delta_0)} {\omega_0(\frac54\Delta_0)}. 
\end{multline}

We now have all the ingredients to show that $\widetilde{\omega}\in RH_p(\frac54\Delta_0,\omega_{L_0})$ (uniformly in $j$) and to do this we let $\Delta=B\cap\pom$ with $B=B(x,r)\subset \frac54 B_0$ and $x\in\pom$. If $r_\Delta<1<\tfrac{c_0}{10^5\kappa_0^3}r_0$ then we can invoke \eqref{conc-Step2-(b)} with $\Delta_1=\Delta$ and this gives us the desired estimate. Assume otherwise that $r_\Delta\ge 1\tfrac{c_0}{10^5\kappa_0^3}r_0$, hence $r_\Delta\approx r_0$ since $B\subset \frac54 B_0$ implies that $r_\Delta<\frac54 r_0$. In that scenario using that $\Delta\subset\frac54 \Delta_0$ and that $\omega_0(\Delta)\approx\omega_0(\frac54\Delta_0)$, $\widetilde{\omega}(\Delta)\approx\widetilde{\omega}(\frac54\Delta_0)$ by Lemma \ref{lemma:proppde} part $(c)$ we obtain that \eqref{RHP-top-scale} gives as desired
\begin{align*}
\left(\aver{\Delta} h(y; \widetilde{L}, L_0, X_0)^{p}d\omega_0(y)\right)^{\frac1{p}}
\lesssim
\left(\aver{\frac54\Delta_0} h(y; \widetilde{L}, L_0, X_0)^{p}d\omega_0(y)\right)^{\frac1{p}}
\lesssim
\frac{\widetilde{\omega}(\frac54\Delta_0)} {\omega_0(\frac54\Delta_0)}
\approx
\frac{\widetilde{\omega}(\Delta)} {\omega_0(\Delta)}. 
\end{align*}
All in one, we have shown that $\widetilde{\omega}\in RH_p(\frac54\Delta_0,\omega_{L_0})$, where the implicit constant depends only on the allowable parameters and which is ultimately independent of $j$ and $\Delta_0$. This, as argued in \textbf{Step 0}, permits us to show that $\omega_L\in RH_p(\Delta_0,\omega_{L_0})$ with the help of Lemma \ref{lemma:LjtoL}. The proof of Proposition \ref{PROP:LOCAL-VERSION}, part $(b)$ is then complete. \qed

\section{Domains with Ahlfors-regular boundary}\label{appendix-CAD}

Throughout this section we assume that $\Omega\subset\ree$, $n\ge 2$, is a 1-sided CAD (cf.~Definition \ref{defi:CAD}). This means that $\Omega$ is a 1-sided NTA domain (it satisfies the Corkscrew and Harnack Chain conditions) and $\pom$ is AR. As mentioned in Section \ref{ssdefs}, the latter condition implies that $\Omega$ satisfies the CDC, hence the theory we have developed in this paper applies to $\Omega$. On the other hand, the fact that Ahlfors regularity condition says that the surface measure $\sigma:= \mathcal{H}^{n}|_{\partial \Omega} $ is a well-behaved object. The goal of this section is to show how some earlier perturbation results, valid in Lipschitz, NTA or 1-sided NTA settings, can be obtained easily from our results. Before giving the precise statements let us present some definition:

\begin{definition}[Reverse Hölder and $A_\infty$ classes with respect to surface measure]
	\label{d:RHp:surface}
	Given $p$, $1<p<\infty$, we say that $\omega_L\in RH_p(\pom,\sigma)$, provided that $\omega_L\ll \sigma$ on $\pom$, and there exists $C\geq 1$ such that, writing $k_L=\frac{d\omega_L}{d\sigma}$ for the associated Radon-Nikodym, for every $\Delta_0=B_0\cap \pom$ where $B_0=B(x_0,r_0)$ with $x_0\in\pom$ and $0<r_0<\diam(\pom)$
	\[
	\left(\aver{\Delta} k_L^{X_{\Delta_0}}(y)^p\, d \sigma(y)\right)^{\frac1p} 
	\leq 
	C 
	\aver{\Delta} k_L^{X_{\Delta_0}}\,d \sigma(y)=
	C\,\frac{\omega_L^{X_{\Delta_0}}(\Delta)}{\sigma(\Delta)}
	\]
	for every $\Delta=B\cap \partial\Omega$ where $B\subset B_0$, $B=B(x,r)$ with  $x\in \partial\Omega$, $0<r<\diam(\partial\Omega)$. The infimum of the constants $C$ as above is denoted by $[\omega_{L}]_{RH_p(\pom,\sigma)}$. 	
	
	We also define 
	\[
	A_\infty(\partial\Omega,\sigma )=\bigcup_{p>1} RH_p(\partial\Omega,\sigma)
	.\]
\end{definition}

These are the results that we can reprove with our methods:

\begin{corollary}
	\label{corol:main}
	Let $\Omega\subset\mathbb{R}^{n+1}$, $n\ge 2$, be a 1-sided CAD. 
	Consider $Lu=-\div(A\nabla u)$ and $L_0u=-\div(A_0\nabla u)$ two real (non-necessarily symmetric) elliptic operators. Define the disagreement between $A$ and $A_0$ in $\Omega$ by
	\begin{equation}\label{discrepancia:corol}
	\varrho(A, A_0)(X)
	:=
	\|A-A_0\|_{L^\infty (B(X,\delta(X)/2))},\qquad X\in\Omega,
	\end{equation}
	where $\delta(X):=\dist(X,\partial\Omega)$, and 
	\begin{equation}\label{def-varrho:corol}
	\vertiii{\varrho(A,A_0)}_\sigma
	:=
	\sup_{B}
	\frac{1}{\sigma (\Delta)}
	\iint_{B\cap\Omega}\frac{\varrho(A,A_0)(X)^2}{\delta(X)}\,dX,
	\end{equation}
	where $\Delta=B\cap\pom$, and the sup is taken over all balls $B=B(x,r)$ with $x\in \pom$ and $0<r<\diam(\pom)$.

	\begin{list}{$(\theenumi)$}{\usecounter{enumi}\leftmargin=1cm \labelwidth=1cm \itemsep=0.2cm \topsep=.2cm \renewcommand{\theenumi}{\alph{enumi}}}
		
		\item Assume that $\vertiii{\varrho(A, A_0)}_\sigma<\infty$. If $\omega_{L_0}\in A_\infty(\pom,\sigma)$, then $\omega_L\in A_\infty(\pom,\sigma)$. More precisely, if $\omega_{L_0}\in RH_p(\pom,\sigma)$ for some $p$, $1<p<\infty$, then $\omega_{L}\in RH_q(\pom,\sigma)$ for some $q$, $1<q<\infty$. Here, $q$ and $[\omega_{L}]_{RH_q(\pom,\sigma)}$ depend only on dimension, the 1-sided CAD constants, the ellipticity constants of $L_0$ and $L$, $\vertiii{\varrho(A, A_0)}_\sigma$, $p$, and $[\omega_{L_0}]_{RH_p(\pom,\sigma)}$.

		\item If $\omega_{L_0}\in RH_p(\pom,\sigma)$, for some $p$, $1<p<\infty$, there exists $\varepsilon_p>0$ (depending only on dimension, the 1-sided CAD constants, the ellipticity constants of $L_0$ and $L$, $p$, and $[\omega_{L_0}]_{RH_p(\pom,\sigma)}$) such that if $\vertiii{\varrho(A, A_0)}_\sigma \leq\varepsilon_p$, then $\omega_{L}\in RH_p(\pom,\sigma)$. Here, $[\omega_{L}]_{RH_q(\pom,\sigma)}$ depends only on  dimension, the 1-sided CAD constants, the ellipticity constants of $L_0$ and $L$, $p$, and $[\omega_{L_0}]_{RH_p(\pom,\sigma)}$.
	\end{list}
\end{corollary}

\begin{corollary}\label{corol:main-SF} 
	Let $\Omega\subset\mathbb{R}^{n+1}$, $n\ge 2$, be a 1-sided CAD. Consider $Lu=-\div(A\nabla u)$ and $L_0u=-\div(A_0\nabla u)$ two real (non-necessarily symmetric) elliptic operators, and recall the definition of $\mathcal{A}_\alpha(\varrho(A,A_0))$ in \eqref{SF-def} for any given $\alpha>0$.
	
	\begin{list}{$(\theenumi)$}{\usecounter{enumi}\leftmargin=1cm \labelwidth=1cm \itemsep=0.2cm \topsep=.2cm \renewcommand{\theenumi}{\alph{enumi}}}
		\item Assume that $\mathcal{A}_\alpha(\varrho(A,A_0))\in L^\infty(\sigma)$. If $\omega_{L_0}\in A_\infty(\pom,\sigma)$, then $\omega_L\in A_\infty(\pom,\sigma)$. More precisely, if $\omega_{L_0}\in RH_p(\pom,\sigma)$ for some $p$, $1<p<\infty$, then $\omega_{L}\in RH_q(\pom,\sigma)$ for some $q$, $1<q<\infty$. Here, $q$ and $[\omega_{L}]_{RH_q(\pom,\sigma)}$ depend only on dimension, the 1-sided CAD constants, the ellipticity constants of $L_0$ and $L$, $\alpha$, $\|\mathcal{A}_\alpha(\varrho(A,A_0))\|_{L^\infty(\sigma)}$, $p$, and  $[\omega_{L_0}]_{RH_p(\pom,\sigma)}$.

		\item If $\omega_{L_0}\in RH_p(\pom,\sigma)$, for some $p$, $1<p<\infty$,  there exists $\varepsilon_p>0$ (depending only on dimension, the 1-sided CAD constants, the ellipticity constants of $L_0$ and $L$, $p$, and $[\omega_{L_0}]_{RH_p(\pom,\sigma)}$), such that if 
		$\mathcal{A}_\alpha(\varrho(A,A_0))\in L^\infty(\sigma)$ with $\|\mathcal{A}_\alpha(\varrho(A,A_0))\|_{L^\infty(\sigma)}\le \varepsilon_p$,  		
		then $\omega_L\in RH_p(\pom,\sigma)$. Here, $[\omega_{L}]_{RH_p(\pom,\sigma)}$ depends only on dimension, the 1-sided CAD constants, the ellipticity constants of $L_0$ and $L$, $\alpha$, $p$, and  $[\omega_{L_0}]_{RH_p(\pom,\sigma)}$.

	\end{list}	
\end{corollary}

In the case of symmetric operators, part $(b)$ of Corollary \ref{corol:main} has been proved for the unit ball in \cite{D}, for bounded CAD in \cite{MPT}, and for 1-sided CAD domains in \cite{CHM}. On the other hand,  part $(a)$ of Corollary \ref{corol:main} can be found for Lipschitz domains in \cite{FKP} and for bounded CAD in \cite{MPT}, both in the case of symmetric operators (but we would expect that similar arguments could be carried over to the non-symmetric case as well). The corresponding result in the setting of 1-sided CAD has been obtained in \cite{CHM} for symmetric operators and then extended to the general case in \cite{CHMT}. Note then that Corollary \ref{corol:main} part $(b)$ seems to be new in the case of non-symmetric operators in 1-sided CAD. Regarding Corollary \ref{corol:main-SF}, part $(a)$ for symmetric operators  was proved in \cite{F} in the unit ball and in \cite{MPT} in the setting of bounded CAD.

Before proving the previous results we need the following auxiliary lemma:

\begin{lemma}\label{lemma:weights}
	Let $\Omega\subset\mathbb{R}^{n+1}$ be a 1-sided CAD and consider $Lu=-\div(A\nabla u)$ and $L_0u=-\div(A_0\nabla u)$ two real (non-necessarily symmetric) elliptic operators. 	If $\omega_{L_0}\in A_\infty(\pom,\sigma)$ and $\omega_{L}\in A_\infty(\pom,\omega_{L_0})$ then $\omega_{L}\in A_\infty(\pom,\sigma)$. More precisely, if $\omega_{L_0}\in RH_p(\pom,\sigma)$, $1<p<\infty$, and $\omega_{L}\in RH_q(\pom,\omega_{L_0})$, $1<q<\infty$, then $\omega_{L}\in RH_r(\pom,\sigma)$ with $r=\frac{p\,q}{p+q-1}\in (1, \min\{p,q\})$ and, moreover,
	\[
	[\omega_{L}]_{RH_r(\pom,\sigma)}\le  [\omega_{L}]_{RH_q(\pom,\omega_0)}\, [\omega_{L_0}]_{RH_p(\pom,\sigma)}^{\frac1{q'}}.
	\]
\end{lemma}

\begin{proof}
Fix $\Delta_0=B_0\cap \pom$ where $B_0=B(x_0,r_0)$ with $x_0\in\pom$ and $0<r_0<\diam(\pom)$. Write $\omega_0=\omega_{L_0}^{X_{\Delta_0}}$ and $\omega=\omega_{L}^{X_{\Delta_0}}$. By definition $\omega_0\ll\sigma$ and $\omega\ll\omega_0$, hence $\omega\ll\sigma$. 
Given $\Delta=B\cap \partial\Omega$ where $B\subset B(x_0,r_0)$, $B=B(x,r)$ with  $x\in \partial\Omega$, $0<r<\diam(\partial\Omega)$,  by Hölder's inequality with exponent $\frac{q}{r}>1$ we obtain
\begin{align*}
&\left(\aver{\Delta} \bigg(\frac{d\omega}{d\sigma}\bigg)^r d \sigma\right)^{\frac1r} 
=
\left(\aver{\Delta} \bigg(\frac{d\omega}{d\omega_0}\,\frac{d\omega_0}{d\sigma}\bigg)^r\, d\sigma\right)^{\frac1r} 
\\
&\qquad=
\left(\aver{\Delta} \bigg(\frac{d\omega}{d\omega_0}\bigg)^r\, \bigg(\frac{d\omega_0}{d\sigma}\bigg)^{\frac{r}{q}}\,\bigg(\frac{d\omega_0}{d\sigma}\bigg)^{\frac{r}{q'}}\,  d\sigma\right)^{\frac1r} 
\\
&\qquad\le 
\left(\frac{\omega_0(\Delta)}{\sigma(\Delta)}\right)^{\frac1q}\,
\left(\aver{\Delta} \bigg(\frac{d\omega}{d\omega_0}\bigg)^q \,d \omega_0\right)^{\frac1q} 
\left(\aver{\Delta} \bigg(\frac{d\omega_0}{d\sigma}\bigg)^{\frac{r}{q'}\,(\frac{q}{r})'}\,  d\sigma\right)^{\frac1{r\,(\frac{q}{r})'}} 
\\
&
\qquad=
\left(\frac{\omega_0(\Delta)}{\sigma(\Delta)}\right)^{\frac1q}\,
\left(\aver{\Delta} \bigg(\frac{d\omega}{d\omega_0}\bigg)^q \,d \omega_0\right)^{\frac1q} 
\left(\aver{\Delta} \bigg(\frac{d\omega_0}{d\sigma}\bigg)^{p}\,  d\sigma\right)^{\frac1{q'\,p}} 
\\
&\qquad\le 
[\omega_{L}]_{RH_q(\pom,\omega_0)}\, [\omega_{L_0}]_{RH_p(\pom,\sigma)}^{\frac1{q'}}\,
\left(\frac{\omega_0(\Delta)}{\sigma(\Delta)}\right)^{\frac1q}\,
\left(\aver{\Delta} \frac{d\omega}{d\omega_0}\,d \omega_0\right) 
\left(\aver{\Delta} \frac{d\omega_0}{d\sigma}\,  d\sigma\right)^{\frac1{q'}} 
\\
&\qquad=
[\omega_{L}]_{RH_q(\pom,\omega_0)}\, [\omega_{L_0}]_{RH_p(\pom,\sigma)}^{\frac1{q'}}\,
\frac{\omega(\Delta)}{\sigma(\Delta)}.
\end{align*}
Thus we conclude that  $\omega_{L}\in RH_r(\pom,\sigma)$ with 
\[
[\omega_{L}]_{RH_r(\pom,\sigma)}\le  [\omega_{L}]_{RH_q(\pom,\omega_0)}\, [\omega_{L_0}]_{RH_p(\pom,\sigma)}^{\frac1{q'}},
\]
and the proof is complete.
\end{proof}

\begin{proof}[Proof of Corollary \ref{corol:main}]
Assume that $\omega_{L_0}\in A_\infty(\pom,\sigma)$. Our first goal is to show that using the notation in \eqref{def-varrho} we have
\begin{equation}\label{carleson-w-sigma}
\vertiii{\varrho(A,A_0)}
\lesssim
\vertiii{\varrho(A,A_0)}_\sigma.
\end{equation}

To see this we take some ideas from	the proof of Theorem \ref{thm:main-SF}. Let $\dd=\dd(\pom)$ be the dyadic grid from Lemma \ref{lemma:dyadiccubes} with  $E=\pom$. 
For any $Q\in\dd$ we set
\[
\gamma_Q=\frac1{\sigma(Q)} 	\iint_{U_Q}\frac{\varrho(A,A_0)(X)^2}{\delta(X)}\,dX.
\]

Fix $B_0=B(x_0,r_0)$ with $x_0\in\pom$ and $0<r_0<\diam(\pom)$. Let $\Delta=B\cap\pom$ with $B=B(x,r)$, $x\in2\Delta_0$, and $0<r<r_0 c_0/4$, here $c_0$ is the Corkscrew constant. Write $X_0=X_{\Delta_0}$ and $\omega_0=\omega_{L_0}^{X_{0}}$. Note that this choice guarantees that $X_0\notin 4B$. Define
\[
\W_B=\{I\in\W: I\cap B\neq\emptyset\}
\]
and for every $I\in\W_B$ let $X_I\in I\cap B$ so that $4\diam(I)\le \dist(I,\pom)\le\delta(X_I)<r$ and hence $I\subset \frac54B$. 
Pick $x_I\in\pom$ such that $|X_I-x_I|=\delta(X_I)\le\diam(I)+\dist(I,\pom)$ and let $Q_I\in\dd$ be such that $x_I\in Q_I$ and $\ell(I)=\ell(Q_I)$. By Lemma \ref{lemma:proppde} parts $(a)$--$(c)$, Harnack's inequality and the fact that $\pom$ is AR one has 
\[
\frac{G_{L_0}(X_{0},Y)}{\delta(Y)}
\approx 
\frac{G_{L_0}(X_{0},X_I)}{\ell(I)}
\approx
\frac{\omega_0(Q_I)}{\ell(I)^n}
\approx
\frac{\omega_0(Q_I)}{\sigma(Q_I)},
\qquad\forall Y\in I.
\]
Using this
\begin{multline*}
\mathcal{I}_B
:=
\iint_{B\cap\Omega}\varrho(A,A_0)(Y)^2\frac{G_{L_0}(X_{0},Y)}{\delta(Y)^2}\,dY
\lesssim
\sum_{I\in\W_B} \iint_{I}\frac{\varrho(A,A_0)(Y)^2}{\delta(Y)}\,dY\,\frac{\omega_0(Q_I)}{\sigma(Q_I)}
\\
\le 
\sum_{I\in\W_B} \iint_{U_{Q_I}}\frac{\varrho(A,A_0)(Y)^2}{\delta(Y)}\,dY\,\frac{\omega_0(Q_I)}{\sigma(Q_I)}
=
\sum_{I\in\W_B} \gamma_{Q_I}\,\omega_0(Q_I),
\end{multline*}
where we have used that by construction $I\subset U_{Q_I}\in \W_{Q_I}$.

Note that $\ell(Q_I)=\ell(I)<\diam (Q_I)<r/4$. Also if $z\in Q_I$, then by \eqref{deltaQ} and \eqref{constwhitney}
\begin{multline*}
|z-x|
\le 
|z-x_I|+|x_I-X_I|+|X_I-x|
\\
\le
\Xi\ell(Q_I)+\delta(X_I)+\frac{r}{4}
<
\Xi\ell(Q_I)+\diam(I)+\dist(I,\pom)+\frac{r}{4}
<12\,\Xi\,r
\end{multline*}
and therefore  $Q_I\subset 12\,\Xi \Delta$. Write then $\F_\Delta=\{Q\in\dd: \frac{r}4\le \ell(Q)< \frac{r}2, Q\cap 12\,\Xi \Delta\neq\emptyset\}$, so that $\F_\Delta$ is a family  of pairwise disjoint dyadic cubes with uniformly bounded cardinality and so that $12\,\Xi \Delta\subset\cup_{Q\in\F_\Delta} Q\subset 13\,\Xi \Delta$. By construction, if $I\in\W_B$, then $Q_I\subset Q$ for some $Q\in \F_\Delta$. Introducing the notation 
\[
\vertiii{\gamma}_{\omega_0,\Delta}:=\sup_{Q\in\F_\Delta} \sup_{Q'\in\dd_{Q}} \frac1{\omega_0(Q')}\sum_{Q''\in\dd_{Q'}}\gamma_{Q''}\,\omega_0(Q''),
\]
it follows that 
\begin{multline}\label{1qf3f}
\mathcal{I}_B
\le 
\sum_{Q\in\F_\Delta} \sum_{Q'\in\dd_{Q}}\gamma_{Q'}\,\omega_0(Q')
\le
\vertiii{\gamma}_{\omega_0,\Delta}\,\sum_{Q\in\F_\Delta} \omega_0(Q)
\\
\le 
\vertiii{\gamma}_{\omega_0,\Delta}\,\omega_0(13\,\Xi \Delta)
\lesssim
\vertiii{\gamma}_{\omega_0,\Delta}\,\, \omega_0(\Delta),
\end{multline}
where we have used Lemma \ref{lemma:proppde}. 

We next estimate $\vertiii{\gamma}_{\omega_0,\Delta}$. Since we have assumed that $\omega_{L_0}\in A_\infty(\pom,\sigma)$, it follows that $\omega_{L_0}\in RH_p(\pom,\sigma)$ for some $p$, $1<p<\infty$, then it is straightforward to see using Lemma \ref{lemma:proppde}  that $\omega_{L_0}^{X_{Q}}\in RH_p^{\rm dyadic}(Q, \sigma)$ for every $Q\in\dd$ (cf. Definition \ref{def:Ainfty-dyadic}). In particular, for every $Q'\in\dd_Q$ with $Q\in\dd$ and for every $F\subset Q'$ we have
\begin{multline*}
\frac{\omega_{L_0}^{X_Q}(F)}{\sigma(Q')}
=
\aver{Q'} \mathbf{1}_F \, k_{L_0}^{X_{Q}}\,d \sigma(y)
\le
\bigg(\frac{\sigma(F)}{\sigma(Q')}\bigg)^{\frac1{p'}}\, \bigg(\aver{Q'} k_{L_0}^{X_Q}(y)^{p}\,d\sigma(y)\bigg)^{\frac{1}{p}} 
\\
\le 
[\omega_{L_0}]_{RH_p(\pom,\sigma)} \, \bigg(\frac{\sigma(F)}{\sigma(Q')}\bigg)^{\frac1{p'}}\, \frac{\omega_{L_0}^{X_Q}(Q')}{\sigma(Q')},
\end{multline*}
where $C>1$ is a uniform constant. Take then $\alpha=\frac12$, $\beta=(2\,C\, [\omega_{L_0}]_{RH_p(\pom,\sigma)})^{-p'}\in (0,1)$, and apply Lemma \ref{lemma:Carleson-mu-nu} with $\mu=\omega_{L_0}^{X_Q}$ and $\nu=\sigma$ to obtain
\begin{multline*}
\sup_{Q'\in\dd_{Q}} \frac1{\omega_{L_0}^{X_Q}(Q')}\sum_{Q''\in\dd_{Q'}}\gamma_{Q''}\,\omega_{L_0}^{X_Q}(Q'')
\lesssim
\sup_{Q'\in\dd_{Q}} \frac1{\sigma(Q')}\sum_{Q''\in\dd_{Q'}}\gamma_{Q''}\,\sigma(Q'')
\\
=
\sup_{Q'\in\dd_{Q}} \frac1{\sigma(Q')}\sum_{Q''\in\dd_{Q'}} \iint_{U_Q}\frac{\varrho(A,A_0)(X)^2}{\delta(X)}\,dX
\\
\lesssim
\sup_{Q'\in\dd_{Q}} \frac1{\sigma(\Delta_Q^*)} \iint_{B_Q^*}\frac{\varrho(A,A_0)(X)^2}{\delta(X)}\,dX
\le
\vertiii{\varrho(A,A_0)}_\sigma,
\end{multline*}
where we have used that the family $\{U_{Q'}\}_{Q'\in\dd}$ has bounded overlap, \eqref{definicionkappa12}, the AR property of $\sigma$ and \eqref{def-varrho:corol}. Invoke once again Lemma \ref{lemma:proppde} and Harnack's inequality to conclude that \eqref{1qf3f} along with the previous estimate readily yield
\begin{align*}
\mathcal{I}_B
\lesssim 
\omega_0(\Delta)\, \sup_{Q\in\F_\Delta} \sup_{Q'\in\dd_{Q}} \frac1{\omega_{L_0}^{X_Q'}(Q')}\sum_{Q''\in\dd_{Q'}}\gamma_{Q''}\,\omega_{L_0}^{X_Q'}(Q'')
\lesssim
\omega_0(\Delta)\, \vertiii{\varrho(A,A_0)}_\sigma,
\end{align*}
Taking then the sup over all $B$ and $B_0$ as above we have shown that \eqref{carleson-w-sigma} holds.

With \eqref{carleson-w-sigma} at hand we are now ready to prove $(a)$ and $(b)$ in the statement. To prove $(a)$ note that by assumption $\vertiii{\varrho(A,A_0)}_\sigma<\infty$ and $\omega_{L_0}\in A_\infty(\pom,\sigma)$. Hence, \eqref{carleson-w-sigma} says that  $\vertiii{\varrho(A,A_0)}<\infty$ and Theorem \ref{thm:main} part $(a)$ yields $\omega_{L}\in A_\infty(\pom, \omega_{L_0})$. In turn, Lemma \ref{lemma:weights} implies that $\omega_{L}\in A_\infty(\pom, \sigma)$ as desired. 

To prove $(b)$ we proceed as follows. Assume that $\omega_{L_0}\in RH_p (\pom,\sigma)$. By Gehring's lemma \cite{G} (see also \cite{CF-1974}) there exists $s>1$ such that $\omega_{L_0}\in RH_{p\,s} (\pom,\sigma)$. Set $q:=\frac{s\,p-1}{s-1}>1$ and note that by \eqref{carleson-w-sigma} and Theorem \ref{thm:main} part $(b)$ we can find  $\varepsilon_p>0$ sufficiently small (depending only on dimension, the 1-sided CAD constants, the ellipticity constants of $L_0$ and $L$, $p$, and $[\omega_{L_0}]_{RH_p(\pom,\sigma)}$) so that if $\vertiii{\varrho(A,A_0)}_\sigma<\epsilon_p$ then $\omega_{L}\in RH_q(\pom, \omega_{L_0})$. If we apply Lemma \ref{lemma:weights}  with $p\,s$ and our choice of $q$ we conclude that $\omega_{L}\in RH_r(\pom, \sigma)$ where
$r=\frac{p\,s\,q}{p\,s+q-1}=p$. This completes the proof.
\end{proof}

\begin{proof}[Proof of Corollary \ref{corol:main-SF}]
Note first that in both cases $(a)$ and $(b)$, the fact that $\omega_{L_0}\in A_\infty(\pom,\sigma)$ implies $\omega_{L_0}\ll \sigma$. On the other hand, since the $A_\infty$ property is symmetric we clearly have that $\sigma\ll \omega_{L_0}$. It is important to emphasize that by Harnack's inequality $\omega_{L}^X\ll\omega_{L}^Y$ for every $X,Y\in\Omega$, hence we do not need to specify the pole in $\omega_L$. All these show that $\|\cdot\|_{L^\infty(\sigma)}=\|\cdot\|_{L^\infty(\omega_{L_0})}$.

To prove $(a)$ we then observe that the assumption $\mathcal{A}_\alpha(\varrho(A,A_0))\in L^\infty(\sigma)$ gives at once that $\mathcal{A}_\alpha(\varrho(A,A_0))\in L^\infty(\omega_{L_0})$ and by Theorem \ref{thm:main-SF} part $(a)$ we conclude that $\omega_{L}\in A_\infty(\pom, \omega_{L_0})$. This, the fact that $\omega_{L_0}\in A_\infty(\pom,\sigma)$, and Lemma \ref{lemma:weights} readily gives that $\omega_{L}\in A_\infty(\pom, \sigma)$ as desired. 

To prove $(b)$ we proceed much as in the corresponding case in the proof of Corollary \ref{corol:main}. Assume that $\omega_{L_0}\in RH_p (\pom,\sigma)$ and invoke once again Gehring's lemma to find $s>1$ such that $\omega_{L_0}\in RH_{p\,s} (\pom,\sigma)$. Set $q:=\frac{s\,p-1}{s-1}>1$ and note that if $\|\mathcal{A}_\alpha(\varrho(A,A_0))\|_{L^\infty(\sigma)}=\|\mathcal{A}_\alpha(\varrho(A,A_0))\|_{L^\infty(\omega_{L_0})}$ is sufficiently small, Theorem \ref{thm:main-SF} part $(b)$ says that  $\omega_{L}\in RH_q(\pom, \omega_{L_0})$. We next apply Lemma \ref{lemma:weights}  with $p\,s$ and our choice of $q$ to conclude that $\omega_{L}\in RH_p(\pom, \sigma)$ much as we did before. 
\end{proof}

\bibliographystyle{plain}

\bibliography{myref}

\end{document}